\documentclass[twoside]{article}

\usepackage{mathrsfs}
\usepackage{amssymb}
\usepackage{amsthm,amsmath,amscd,amsxtra}
\usepackage{color}
\allowdisplaybreaks
\usepackage{amsfonts}
\usepackage{latexsym}
\usepackage{graphicx}
\usepackage{verbatim}

\def\rr{{\mathbb R}}
\def\rn{{{\rr}^n}}

\def\cc{{\mathbb C}}
\def\nn{{\mathbb N}}
\def\zz{{\mathbb Z}}

\def\cd{{\mathcal D}}

\def\cg{{\mathcal G}}

\def\cj{{\mathcal J}}

\def\cm{{\mathcal M}}

\def\cp{{\mathscr P}}
\def\cq{{\mathcal Q}}

\def\cs{{\mathcal S}}

\def\cx{{\mathcal X}}
\def\cy{{\mathcal Y}}

\def\fz{\infty}
\def\az{\alpha}

\def\supp{{\mathop\mathrm{\,supp\,}}}

\def\diam{{\mathop\mathrm{\,diam\,}}}

\def\loc{{\mathop\mathrm{\,loc\,}}}

\def\aoti{{\mathrm {\,ATI\,}}}

\def\lz{\lambda}

\def\dz{\delta}

\def\ez{\epsilon}

\def\bz{\beta}
\def\gz{{\gamma}}

\def\vz{\varphi}
\def\kz{{\kappa}}
\def\tz{\theta}
\def\sz{\sigma}
\def\wz{\widetilde}

\def\ls{\lesssim}
\def\gs{\gtrsim}
\def\laz{\langle}
\def\raz{\rangle}

\def\wg{\wedge}

\def\gfz{\genfrac{}{}{0pt}{}}
\def\hs{\hspace{0.3cm}}

\def\dint{\displaystyle\int}

\def\r{\right}
\def\lf{\left}
\def\gfz{\genfrac{}{}{0pt}{}}

\newtheorem{thm}{Theorem}[section]
\newtheorem{prop}[thm]{Proposition}
\newtheorem{cor}[thm]{Corollary}
\newtheorem{lem}[thm]{Lemma}
\theoremstyle{definition}

\newtheorem{defn}[thm]{Definition}
\newtheorem{rem}[thm]{Remark}

\numberwithin{equation}{section}

\begin{document}

\arraycolsep=1pt

\title{\bf Besov-Type and Triebel--Lizorkin-Type Spaces
Associated with  Heat Kernels
\footnotetext{\hspace{-0.22cm}2010
{\it Mathematics Subject Classification}. Primary 46E35;
Secondary 42B35, 42B25.
\endgraf {\it Key words and phrases}. Besov space, Triebel--Lizorkin space, metric measure space, heat kernel, Peetre maximal function, frame.
\endgraf
This project is  supported
by the National Natural Science
Foundation of China (Grant Nos. 11471042, 11171027 and 11361020), the Specialized Research Fund for the Doctoral Program of Higher Education
of China (Grant No. 20120003110003),
the Fundamental Research Funds for Central Universities of China (Nos. 2013YB60 and 2014KJJCA10)
and the Alexander von Humboldt Foundation.
\endgraf $^\ast$ Corresponding author.}}
\author{Liguang Liu,  Dachun Yang\,$^\ast$ and Wen Yuan}
\date{}
\maketitle

\vspace{-1cm}

\begin{center}
\begin{minipage}{14cm}\small
\noindent{\bf Abstract.}
Let $(M, \rho,\mu)$ be a space of homogeneous type
satisfying the reverse doubling condition and the non-collapsing condition.
In this paper, the authors introduce
Besov-type spaces $B_{p,q}^{s,\tau}(M)$ and Triebel--Lizorkin-type spaces $F_{p,q}^{s,\tau}(M)$
associated to a   nonnegative self-adjoint operator $L$ whose heat kernel
satisfies sub-Gaussian upper bound  estimate, H\"older continuity, and
stochastic completeness. The novelty in this article is that the indices
$p,q,s,\tau$ here can be take full range of all possible values as in
the Euclidean setting.
Characterizations of these spaces
via Peetre maximal functions and the heat semigroup are established
for full range of possible indices. Also, frame characterizations of these spaces are given.
When $L$  is the Laplacian operator on $\mathbb R^n$, these spaces
coincide with the Besov-type and Triebel-Lizorkin-type spaces on $\mathbb R^n$
studied in [Lecture Notes in Mathematics 2005, Springer-Verlag, Berlin, 2010].
In the case  $\tau=0$ and the smoothness index $s$ is around zero,
comparisons of  these spaces with
the Besov and Triebel--Lizorkin spaces studied in [Abstr. Appl. Anal. 2008, Art. ID 893409, 250 pp]
are also presented.
\end{minipage}
\end{center}

%%%%%%%%%%%%%%%%%%%%%%%%%%%%%%%%%%%%%%%%%%%%%%%%%%%%%%%%%%%%%%%%%%%%%

\section{Introduction\label{s1}}

\hskip\parindent
 The tremendous development of theories of function spaces
 in the last few decades
has resulted in extraordinary accomplishments
in several fields of mathematics such as
potential theory, partial differential equations, approximation theory
and so on.
Besov and Triebel--Lizorkin spaces,
known so far, are very general scales
of functions spaces. They cover various types of function spaces
such as Lebesgue spaces, Sobolev spaces,
Hardy spaces and $\mathrm{BMO}$ (see, for example, \cite{t83,t92,rs,FJ90, FJW}).
In recent years, due to the applications in partial differential equations
such as heat and Navier--Stokes equations,
the scale of Besov and Triebel--Lizorkin
spaces was further extended to Besov--Morrey spaces and Triebel--Lizorkin--Morrey
spaces, via replacing the Lebesgue norm in the definition of
 Besov--Triebel--Lizorkin spaces by the Morrey norm (see, for example,
 \cite{ax02,KY,ma03,TX,st,s012,s013}). The classical Morrey spaces and
many other Morrey-type spaces, such as
 Hardy--Morrey spaces and Sobolev--Morrey spaces, are proved to belong to this scale.
 A more general scale of function
spaces is the Besov-type and Triebel--Lizorkin-type spaces
introduced in \cite{yy1,yy2,YSY}, which unify Besov and Triebel--Lizorkin
spaces, Triebel--Lizorkin--Morrey
spaces (see \cite{TX,st,s09,s10}), and the $Q_\az$ spaces (see \cite{ejpx,dx,xj01,xj06}).
For  more properties on these generalized Besov and Triebel-Lizorkin
spaces and their applications in
partial differential equations such as heat and Navier-Stokes equations,
we refer to
\cite{rs,x07,syy,yy4,lsuyy,lsuyy2,lz10,lxy12,yyz13,ysy13,yhsy} and, especially,
to the excellent monograph \cite{t13} by Triebel and two excellent surveys \cite{s012,s013}
by Sickel for many unsolved questions on
this subject.

The main aim of this article is
to develop a theory for
the Besov-type and the
Triebel--Lizorkin-type spaces on general metric measure spaces.
The setup of the underlying space is as follows.
Let $(M, \rho)$ be a locally compact complete metric space
with a metric $\rho$.
Suppose that $\mu$ is a positive regular Borel measure such
that the following \emph{doubling condition} holds true:
there exists a positive constant $K$ such that, for all $x\in M$ and $r\in(0,\fz)$,
\begin{equation}\label{doubling1}
\mu(B(x, 2r))\le K\mu(B(x, r)).
\end{equation}
The triple $(M, \rho, \mu)$ is called a \emph{space of homogeneous
type} in the sense of Coifman and Weiss \cite{CW1, CW2}.
(Notice that a space of homogeneous
type in \cite{CW1,CW2} is endowed with a quasi-metric. But, throughout
this article, we \emph{always assume} that $\rho$ is a metric for simplicity.)
Condition \eqref{doubling1} implies that, for all $x\in M$, $r\in(0,\fz)$ and $\lz\in(1,\fz)$,
\begin{equation}\label{doubling2}
\mu(B(x, \lz r))\le K \lz^d\mu(B(x, r)),
\end{equation}
where $d:= \log_2 K>0$ is a ``dimension" constant.
Also, we assume  the  \emph{reverse doubling condition}:
there exists a constant $K_\ast\in(1,\fz)$ such that, for all $x\in M$ and $0< r\le \frac{\diam M}{3}$,
\begin{equation}\label{rdoubling1}
\mu(B(x, 2 r))\ge K_\ast\mu(B(x, r)).
\end{equation}
Condition \eqref{rdoubling1} implies that,
for all $x\in M$, $\lz\in[1,\fz)$ and $0< r\le \frac{2\diam M}{3\lz }$,
\begin{equation}\label{rdoubling2}
\mu(B(x, \lz r))\ge K_\ast^{-2}\lz^\kz \mu(B(x, r)),
\end{equation}
where $\kz:= \log_2 K_\ast>0$ also measures the ``dimension" of $(M, \rho, \mu)$
in some sense.
The doubling and the reverse doubling conditions make $(M, \rho, \mu)$ into
an RD-space originally introduced in \cite{HMY2}.
Moreover, we require the following \emph{non-collapsing condition}: there exists a positive constant $c_0$ such that
\begin{equation}\label{non-collapsing}
\inf_{x\in M}\mu(B(x, 1))\ge c_0.
\end{equation}
On $(M,\rho,\mu)$, we always assume that there exists a nonnegative definite self-adjoint operator $L$ whose {\em domain} ${\rm{Dom}}(L)$ is dense in $L^2(M)$.
By the spectral theory, $L$  has a spectral resolution $\{E_\lz\}_{\lz\ge 0}$
such that, for any bounded Borel measurable function $f$,
$f(L)=
\int_0^\infty f(\lz)\,dE_\lz.$
The heat semigroup $\{e^{-tL}\}_{t>0}$ arising from $L$ is assumed to be a family of
integral operators which is associated to the {\em heat kernel}
$\{p_t\}_{t>0}$ in the following way:
$$e^{-tL}f(x)=\int_M p_t(x,y)f(y)\,d\mu(y), \quad \forall\ x\in M,$$
at least for functions $f\in L^2(\mu)$. Obviously,
the heat kernel $\{p_t\}_{t>0}$ is {\em symmetric}, that is, for all $t>0$ and $x,\,y\in M$,
$p_t(x,y)=p_t(y,x).$
It is also easy to observe that $\{p_t\}_{t>0}$ satisfies the {\em heat semigroup property:\,} for all $s,\,t>0$ and $x,\,y\in M$,
$$p_t(x,y)=\int_M p_t(x,z)p_t(z,y)\,d\mu(z).$$
Further, assume that there exist
constants  $C^*$, $c^*$, $\az_0\in(0,\infty)$ and $\beta_0\in[2,\fz)$ such that the following hold:

\begin{enumerate}
\item[\bf (UE)] {\bf Upper bound estimate:\,}
 for all $t\in(0,1]$ and $x,\, y\in M$,
\begin{equation}\label{GUB}
|p_t(x,y)|
\le C^\ast\frac{\exp\big(-c^\ast\big[\frac{\rho(x,y)}{t^{1/\beta_0}}\big]
^{\frac{\beta_0}{\beta_0-1}}\big)}
{\sqrt{\mu(B(x,t^{1/\beta_0})) \,\mu(B(y, t^{1/\beta_0}))}}.
\end{equation}

\item[\bf (HE)] {\bf H\"older continuity estimate:\,}
for all $t\in(0,1]$ and $x,\, y,\, y'\in M$ satisfying
$\rho(y, y')\le  t^{1/\beta_0}$,
\begin{equation}\label{HC}
|p_t(x,y)-p_t(x,y')|
\le C^\ast \left[\frac{\rho(y,y')}{t^{1/\beta_0}}\right]^{\alpha_0}
\frac{\exp\big(-c^\ast\big[\frac{\rho(x,y)}{t^{1/\beta_0}}
\big]^{\frac{\beta_0}{\beta_0-1}}\big)}
{\sqrt{\mu(B(x,t^{1/\beta_0})) \,\mu(B(y, t^{1/\beta_0}))}}.
\end{equation}

\item[\bf (SC)]
{\bf Stochastic Completeness:\,}  for all $t\in(0,\infty)$ and $x\in M$,
\begin{equation}\label{Markov}
\int_M p_t(x,y)\,d\mu(y)=1.
\end{equation}
\end{enumerate}
Throughout this article, we  fix the aforementioned parameters
$K,\, d,\, K_\ast,\, \kappa,\, c_0$,
$C^*$, $c^*$, $\az_0$ and $\beta_0$.

For the case $\bz_0=2$, an interesting example for the previous setup
arises from the second order elliptic operator $L:=-{\rm div} (A \nabla),$
where $A:=\{a_{i,j}(x)\}_{1\le i,j\le d}$ is a uniformly elliptic symmetric
matrix-valued real function on $\rr^d$ or complex function on $\rr^d$
with $d\in\{1,2\}$. Another typical example is the interval $[-1,1]$
endowed with the measure $du(x):=w_{\az,\bz}(x)\,dx$ and $L$ being
the Jacobi operator, where $w_{\az,\bz}(x):=(1-x)^\az(1+x)^\bz$, $x\in[-1,1]$,
with $\az,\,\bz\in(-1,\fz)$, is the classical Jacobi weight on $[-1,1]$ (see \cite{CKP}).
Other examples for $\bz_0=2$
can be given by  geodesically complete Riemannian manifolds with nonnegative Ricci curvature.
For $\bz_0>2$, a bunch of examples are provided by a
large family of various fractals; see  \cite{Gri}. For example, the Sierpinski gasket has a natural
Hausdorff measure $\mu$ and a
diffusion process which
has a transition  density $p_t$ such that
\begin{equation}\label{diffusion}
p_t(x,y) \asymp \frac{C}{t^{d/\beta}}
\exp\left(-c\left[\frac{\rho(x,y)}{t^{1/\beta}}\right]
^{\frac{\beta}{\beta-1}}\right),
\end{equation}
where $d$ is the Hausdorff dimension of the fractal
and $\beta$  the {\em walk dimension} which is larger than $2$.
Here the notation $\asymp$ in \eqref{diffusion} means that
both $\le$ and $\ge$ can be used, but the positive constants
$C$ and $c$  may be different in upper and lower bounds.
Indeed, there are examples so that $\bz$ in \eqref{diffusion} is possible
to take all values in the range $[2, \, d+1)$.
Since all metric balls are precompact as a priori,
if the heat kernel $p_t$ satisfies  \eqref{diffusion},
then it satisfies {\bf (HE)}; see  \cite[Theorem 3.1 and Corollary 4.2]{BGK} and  \cite[Theorem~7.4]{GT}.
The topic of heat kernels has been studied intensively in lots of articles;
see, for example, \cite{BGK, Gri, GT, GH, GHL} and the references therein.
For more details of these examples, we refer the reader to
\cite[Section~1.3]{GL}.

For the above setup, Coulhon, Kerkyacharian and Petrushev \cite{CKP,KP}
considered the special case $\bz_0=2$ and they developed
a theory of Besov and Triebel-Lizorkin spaces.
Assume that there exists a positive constant $c$
such that $\Phi_0, \Phi \in C^\infty([0,\fz))$ such that
\begin{equation}\label{e2.1'}
\supp \Phi_0\subset [0, 2],\quad
\Phi_0^{(2\nu+1)}(0)=0\ \,\textup{for all } \nu\ge0,\,\, \qquad
|\Phi_0(\lz)|\ge c\ \,\,\textup{for}\,\,\lz\in[0, 2^{-3/4}],
\end{equation}
and
\begin{equation}\label{e2.2'}
\supp \Phi\subset [2^{-1}, 2],\quad \, \quad
|\Phi(\lz)|\ge c\ \,\,\textup{for}\,\,\lz\in[2^{-3/4}, 2^{3/4}].
\end{equation}
For all $j\in\nn$, let
\begin{equation}\label{eq:2.16xxx}
\Phi_j(\cdot):=\Phi(2^{-j}\cdot).
\end{equation}
For $s\in\rr$ and $p,q\in(0,\infty]$,
the \emph{Besov space $B_{p,q}^s(M)$}  is defined to be the collection of all {\em distributions} $f$ such that
\begin{equation}\label{besov}
\|f\|_{B^s_{p,q}(M)}:=
\bigg\{\sum_{j=0}^\infty\big\|2^{js}\Phi_j(\sqrt L)f\big\|_{L^p(M)}^q\bigg\}^{1/q}<\fz,
\end{equation}
where a usual modification is made when $p$ or $q$ is infinity.
Here we refer the reader to Section \ref{sec-3} below for the definition of {\em ``distributions"}, and to \cite[Section~5]{KP} for a detailed discussion.
Analogously, for $s\in\rr$, $p\in(0,\infty)$ and $q\in(0,\infty]$, the \emph{Triebel--Lizorkin space $F^s_{p,q}(M)$}
is defined to be  the collection of all distributions
$f$ such that
\begin{equation}\label{triebel}
\|f\|_{F^s_{p,q}(M)}:=
\bigg\|\bigg\{\sum_{j=0}^\infty|2^{js}\Phi_j(\sqrt L)f|^q\bigg\}^{1/q}\bigg\|_{L^p(M)}<\fz,
\end{equation}
where a usual modification is made when $q=\infty$.
Similarly, the spaces
$\wz B^s_{p,q}(M)$ and $\wz F^s_{p,q}(M)$ are defined
with $2^{js}$ replaced by
$[\mu(B(\cdot, 2^{-j}))]^{-s/d}$ in the above two (quasi-)norms.
As was proved in \cite[Propositions~6.3 and 7.2]{KP}, one may also equivalently
define these Besov and Triebel--Lizorkin spaces
by replacing the dilation $2$ in the definition of $\Phi_j$
(see \eqref{eq:2.16xxx}) with a general number $b\in(1,\fz)$.
Frame decompositions of these spaces are considered in \cite{KP} by using certain Calder\'on reproducing
formula.

It should be pointed out that there exists some history on the study for
function spaces related to operators different from the Laplace operators;
see, for example, Peetre \cite{p76} and  Triebel
\cite{t83,t95} for using spectral decompositions induced by a selfadjoint
positive operator to introduce (inhomogeneous) Besov spaces.
On the other hand, it is known that
Riesz transforms defined via a general operator $L$
may not be bounded on the classical Hardy spaces.
To solve such problems, Auscher, Duong and McIntosh \cite{ADM}
made some prominent contributions, which include
a theory of Hardy spaces associated with a  general operator $L$
whose heat kernels satisfy pointwise Possion upper bounds;
see also Duong and Yan \cite{DY0, DY1, DY2}.
On metric measure spaces whose measures satisfy a polynomial growth condition,
Bui, Duong and Yan in \cite{BDY}  studied
homogeneous Besov space $\dot B_{p,q}^{s}$ for $|s|<1$ and $p,q\in[1,\infty]$
associated with an operator $L$
whose heat kernels satisfy the upper bound Gauss estimate and the H\"older
continuity.
The advantage of the theory of Besov and Triebel--Lizorkin spaces developed in \cite{KP} lies in that it concerns full range of the indices $s,p$ and $q$.

Inspired by \cite{KP} and \cite{yy1,yy2,YSY,t13},
the main aim of this article is to develop
the Besov-type and the Triebel--Lizorkin-type spaces for full range of possible indices on the metric measure space $(M,\rho,\mu)$ which satisfies \eqref{doubling1}, \eqref{rdoubling1} and \eqref{non-collapsing},
with further assumptions that,
on $(M,\rho,\mu)$, there is a  nonnegative self-adjoint  operator $L$
whose heat kernels $\{p_t\}_{t>0}$
satisfy  conditions \textbf{(UE)}, \textbf{(HE)}  and
 \textbf{(SC)}.
To achieve this goal, we stick to the philosophy used in \cite{BH, FJ90, t83, YSY}
(see also \cite{s012,s013}).
The obstacle here is that, on general metric measure spaces, it is difficult to consider functions
with smooth order strictly larger than $1$.
However, the assumptions on the heat kernel provide kind of ``differential" structure
on the metric measure space. Indeed, due to $ {\bf (UE)}$ and ${\bf (HE)}$,
the smooth functional calculus induced by
heat kernels still have fast decay (even can be exponential decay) at infinity.
This idea was developed in \cite{KP, CKP} for the case $\bz_0=2$, and here it is generalized to general $\bz_0$ in Propositions \ref{prop2.10x} and \ref{prop2.12x} below. Such smooth functional calculus plays a role of the Schwartz functions as in the Euclidean space.
This standpoint is reconfirmed by a new observation in this article, that is,
the pointwise
off-diagonal estimate presented in Proposition \ref{prop2.14x} below.
Thus, it is possible to consider functions
with a higher smooth order.
Due to the smooth functional calculus, one can establish the continuous Calder\'on reproducing formula (see Section \ref{sec-2.2}, and \cite{CKP,KP} for the case $\bz_0=2$).
This formula is a powerful tool, so that it can be used to establish the
Peetre maximal function characterization and
the heat semigroup characterization of these spaces (see Section \ref{sec-6}).
Further, we build a new
discrete Calder\'on reproducing formula (see Theorem \ref{thm-CRF} below)
that  is much more
parallel to the  one used in the classical setting \cite{BH, FJ90, t83, YSY}.
Consequently, in Section \ref{sec-7},  frames decompositions of such Besov-type and Triebel-Lizorkin-type spaces are considered. The framework we build in this article
generalizes the function spaces in \cite{YSY} (see also \cite{s012, s013}) to metric measure spaces, and also generalizes the work of \cite{KP, CKP} to more general scale of functions spaces.

This article is organized as follows.

Section \ref{sec-2} is devoted to some auxiliary estimates.
In Section \ref{sec-2.1}, we present some basic estimates which hold
true on any  metric measure space $(M,\rho,\mu)$.
Then Section \ref{sec-2.2} gives some estimates related to
the smooth functional calculus induced by the heat kernel, including
an off-diagonal estimate (see Proposition \ref{prop2.14x} below).
The exponential decay of the functional calculus for certain smooth functions in
Proposition \ref{prop2.12x} below is crucial for establishing the discrete
Calder\'on reproducing formula in Theorem \ref{thm-CRF}
below (see the proof of Lemma \ref{lem9.5x} below).
The  continuous Calder\'on
reproducing formula is given at the end of Section \ref{sec-2.2}.

In Section \ref{sec-3},
we introduce the Besov-type spaces $B_{p,q}^{s,\tau}(M)$, $\wz B_{p,q}^{s,\tau}(M)$,
and the Triebel--Lizorkin-type spaces
$F_{p,q}^{s,\tau}(M)$, $\wz F_{p,q}^{s,\tau}(M)$,
where $\tau\in(0,\infty)$, $s\in\rr$, $p\in(0,\infty)$, $q\in(0,\infty)$, and $q$ can take $\infty$ for the spaces
$B_{p,q}^{s,\tau}(M)$ and $\wz B_{p,q}^{s,\tau}(M)$.
When $\tau=0$ and $\bz_0=2$, these spaces are actually the Besov and
the Triebel--Lizorkin
spaces introduced in \cite{KP}.

Applying the smooth functional calculus in Section \ref{sec-2}, in Section \ref{sec-4.1},
we control the Peetre maximal functions (see Proposition
\ref{prop4.2x} below) by the Hardy-Littlewood maximal function,
which essentially generalizes  \cite[Lemma 6.4]{KP} and is
 used elsewhere  in this article.
The estimate in Proposition
\ref{prop4.2x}, which is not restricted to be elements in
the spectral space as in \cite[Lemma 6.4]{KP},
is valid for general distributions.
In Section \ref{sec-4.2}, we prove some embedding properties
of the Besov-type and the Triebel--Lizorkin-type spaces,
and then classify these spaces for the index $\tau$ in different ranges
in  Section \ref{sec-4.3}.

In Section \ref{sec-6}, applying
the estimates of the Peetre maximal functions
and the continuous Calder\'on
reproducing formula,
we  characterize the Besov-type and the Triebel--Lizorkin-type spaces
via the Peetre maximal functions (see Theorem \ref{thm6.2x} below),
which also indicates that these spaces are well defined.
By using this Peetre maximal function characterization, we further
 establish the heat semigroup characterization of the Besov-type
and the Triebel--Lizorkin-type spaces in both discrete and continuous
versions (see Theorems \ref{thm6.7x} and \ref{thm6.8x} below).
Comparing with the
continuous heat semigroup characterization for Besov and Triebel-Lizorkin
spaces in \cite[Theorems~6.7 and 7.5]{KP}, wherein $p\in[1,\fz]$,
there is no restriction on $p$ here in the discrete
heat semigroup characterization of the Besov-type
and the Triebel--Lizorkin-type spaces.

Section \ref{sec-7} is devoted to the frame characterization
of these new scales of function spaces.  The frame structure we considered here
relies on Christ's dyadic cubes in $M$, which is different from those in \cite{KP},
and hence we need to establish a new discrete Calder\'on
reproducing formula associated with Christ's dyadic cubes and the
functions $\Phi_0$ and $\Phi$ in \eqref{e2.1'} and \eqref{e2.2'}
(see Theorem \ref{thm-CRF} below, whose proof is presented in Section 8).
As an application,
we show that $F_{p,q}^{s,1/p}(M)$ and $\wz F_{p,q}^{s,1/p}(M)$
are indeed the endpoint case $F_{\infty,q}^s(M)$ and $\wz F_{\infty,q}^s(M)$ of the Triebel--Lizorkin
spaces, where $p\in(0,\infty)$, $q\in(0,\infty]$ and $s\in\rr$.

Let $(M,\rho,\mu)$ be the Euclidean space
and $L$ the Laplacian operator. In Section \ref{sec-8.1},
we prove that  the Besov-type
and the Triebel--Lizorkin-type spaces
 defined in this article coincide with those spaces
 introduced by Yuan, Sickel and Yang \cite{YSY},
 by using their heat semigroup characterizations. Hence, the article here generalizes the work of \cite{YSY} (see also \cite{s012,s013}).

Recall that Besov spaces and Triebel--Lizorkin spaces with smooth order smaller than $1$
on RD-spaces was studied systematically in \cite{HMY2}. It was asked as an open question in
\cite{KP} whether these spaces coincide with the ones introduced in \cite{KP}.
 In Section \ref{sec-8.2}, we give an affirmative answer
  to this question when smooth order is close to zero.
 To be precise, when $\tau=0$ and $\bz_0=2$ in ${\bf (UE)}$ and ${\bf (HE)}$,
we show that the (quasi-)norms of the Besov and
the Triebel--Lizorkin spaces
on RD-spaces
defined in \cite{HMY2} coincide exactly
with those
in \cite{KP} when $s$ is around zero.  The proof needs
the discrete  Calder\'on
reproducing formula obtained in Section \ref{sec-7} and
the corresponding  one on RD-spaces obtained in \cite{HMY2}.
Thus, this article also generalizes both the works of \cite{HMY2} and \cite{KP}.

We finally make some conventions on notation.
Let $\nn:=\{1,2,\dots\}$, $\zz_+:=\{0\}\cup\nn$,
$\rr_+:=[0,\infty)$,  and   $\cc_+:=\{a+ib:\, a>0,\ b\in\rr\}$.
Denote by $C$ a positive constant depending at most on the parameters
$K,\, d,\, K_\ast,\, \kappa,\, c_0$,
$C^*$, $c^*$, $\az_0$ and $\beta_0$ appearing in \eqref{doubling1}
through \eqref{Markov}, but the value of $C$ may be different on  each occasion.
Occasionally, we may write $C:=C(\az, \bz,\dots)$ which means that
$C$ depends not only on the aforementioned parameters in \eqref{doubling1} through \eqref{Markov}
but also on the parameters $\az,\,\bz,\,\dots$ .
For any numbers $s,\,t\in\rr$, let
$(s\vee t):=\max\{s,t\}$ and $(s\wedge t):=\min\{s,t\}.$
The notation
$A\lesssim B$ means $A\le C B$ and, similarly, for $A\gtrsim B$.
If $B\ls A\ls B$, then write $A\sim B$.
If an operator $T$ is bounded from a (quasi)-Banach space
$\cx$ to a (quasi)-Banach space $\cy$, then
we denote by $\|T\|_{\cx\to\cy}$ its \emph{operator norm}.
For notational simplicity,
we let $|E|:=\mu(E)$ for any measurable set $E\subset M$.

%%%%%%%%%%%%%%%%%%%%%%%%%%%%%%%%%%%%%%%%%%%%%%%%%%%%%%%%%%%%%%%%%%%%%

\section{Smooth functional calculus}\label{sec-2}

\hskip\parindent
In this section, we give the smooth functional calculus
induced by the heat kernels, pointwise off-diagonal estimates, and the continuous
Calder\'on reproducing formula. Some of ideas come from
the articles \cite{CKP, KP} where the special case $\bz_0=2$ was considered.
For the completeness of the article, we provide detailed proofs of
these results for general $\bz_0$.

\subsection{Basic estimates related to metric measure spaces}\label{sec-2.1}

\hskip\parindent
For any $\delta,\,\sigma \in(0,\infty)$, let
\begin{equation*}
D_{\delta, \sigma}(x,y)
:= \frac{1}
{\sqrt{|B(x, \delta)|\, |B(y,\delta)|} } \frac1{\left[1+\delta^{-1}{\rho(x,y)}\right]^{\sigma}},\qquad\, x,\,y\in M.
\end{equation*}
By \eqref{doubling2} and $B(x,\delta)\subset B(y,\delta+\rho(x,y))$, we see that
\begin{equation*}
D_{\delta, \sigma}(x,y) \le \,
\frac{\sqrt{K}} {|B(x, \delta)|} \,
\frac1{\left[1+\delta^{-1}{\rho(x,y)}\right]^{\sigma-d/2}},\qquad\, x,\,y\in M.
\end{equation*}
Notice that the roles of $x$ and $y$ can be reversed in the above inequality.
Furthermore, by \eqref{doubling1} and \eqref{rdoubling1},
we can deduce the estimates listed in the following lemma.

\begin{lem}\label{lem2.1x}
\begin{enumerate}
\item[\rm (i)] Let $\sz\in(d,\fz)$. Then,
for all  $\delta\in(0,\fz)$
and   $x\in M$,
$$
\frac{1}{|B(x, \delta)|}\int_M  \frac{1}
{[1+\delta^{-1}\rho(x,y)]^{\sz}}\,d\mu(y)\le \frac{K}{1-2^{d-\sz}},
$$
$$
\int_M  \frac{1}
{|B(y, \delta)|[1+\delta^{-1}\rho(x,y)]^{\sz}}\,d\mu(y)\le \frac{K2^d}{1-2^{d-\sz}}\quad
and\quad
\int_M   D_{\delta, \sigma}(x,y)
\,d\mu(y)\le \frac{K2^d}{1-2^{d-\sz}}.$$

\item[\rm (ii)] Let $\sz\in(d,\fz)$. Then,
 for all  $s,\,t\in(0,\fz)$
and   $x,\,y\in M$,
$$
\int_M   D_{s, \sigma}(x,z) D_{t, \sigma}(z,y)
\,d\mu(z)\le \frac{K^2\,2^{d+\sz+1}}{1-2^{d-\sz}} \max\{(t^{-1}s)^{d},\,(s^{-1}t)^{d}\}\,
D_{s\vee t, \sigma}(x,y).$$
\end{enumerate}
\end{lem}

\begin{proof}
For (i), by similarity, we only prove the third inequality in (i).
Fix $\sz>d$. Write
\begin{eqnarray*}
\int_M   D_{\delta, \sigma}(x,y)
\,d\mu(y)
=\sum_{j=0}^\infty \int_{\rho(x,y)\sim 2^j\dz}
\frac{1}
{\sqrt{|B(x, \delta)|\, |B(y,\delta)|} } \frac1{\left[1+\delta^{-1}{\rho(x,y)}\right]^{\sigma}}\,d\mu(y),
\end{eqnarray*}
where the notation $\rho(x,y)\sim 2^j\dz$ means that
$2^{j-1}\dz\le \rho(x,y)<2^{j}\dz$ when $j\in\nn$ and
that $\rho(x,y)<\dz$ when $j=0$. Then
\begin{eqnarray*}
\int_M   D_{\delta, \sigma}(x,y)
\,d\mu(y)
\le \sum_{j=0}^\infty 2^{-j\sz}\int_{\rho(x,y)< 2^{j}\dz}
\frac{1}
{\sqrt{|B(x, \delta)|\, |B(y,\delta)|} } \,d\mu(y).
\end{eqnarray*}
For $\rho(x,y)<2^{j}\dz$, using \eqref{doubling2}, we obtain
$$
\frac{1}
{\sqrt{|B(x, \delta)|\, |B(y,\delta)|} }
= \frac{1}
{\sqrt{|B(x, \delta)|\, |B(x, 2^{j}\delta)| } } \sqrt{\frac{|B(x, 2^{j}\delta)|}{|B(y,\delta)|}}
\le  \frac{\sqrt{K2^{(j+1)d}}}
{\sqrt{|B(x, \delta)|\, |B(x, 2^{j}\delta)| } },
$$
and hence, by $\sigma\in(d,\fz)$,
\begin{eqnarray*}
\int_M   D_{\delta, \sigma}(x,y)
\,d\mu(y)
\le \sum_{j=0}^\infty 2^{-j\sz}
\sqrt{K2^{(j+1)d}}  \sqrt{\frac{|B(x, 2^{j}\delta)|}{|B(x,\delta)|}}
\le K2^d \sum_{j=0}^\infty 2^{-j(\sz-d)}=\frac{K2^d}{1-2^{d-\sz}}.
\end{eqnarray*}
The remaining two inequalities in (i) follow in a similar way.

To prove (ii), by symmetry, we may as well assume that $s\le t$. By \eqref{doubling2}, we have
$$
D_{s, \sigma}(x,z) \le K ( ts^{-1})^d \,D_{t, \sigma}(x,z),\qquad x,\,z\in M.
$$
Thus,
\begin{equation}\label{eq:2.1x}
\int_M   D_{s, \sigma}(x,z) D_{t, \sigma}(z,y)
\,d\mu(z)\le K ( ts^{-1})^d \int_M   D_{t, \sigma}(x,z) D_{t, \sigma}(z,y)
\,d\mu(z).
\end{equation}
Notice that any $x,\,y,\,z\in M$
satisfy $\rho(z,x)\ge \frac12\rho(x,y)$ or  $\rho(z,y)\ge \frac12\rho(x,y)$.
Therefore,
\begin{eqnarray*}
&&\int_M   D_{t, \sigma}(x,z) D_{t, \sigma}(z,y)
\,d\mu(z)\\
&&\quad\le  \frac{1}
{\sqrt{|B(x, t)|\, |B(y,t)|} }
\int_{\rho(z,x)\ge \frac12\rho(x,y)}
\frac1{|B(z,t)|}
\frac{1}{[1+t^{-1}{\rho(x,z)}]^\sz[1+t^{-1}{\rho(z,y)}]^\sz}\,d\mu(z)\\
&&\qquad
+\frac{1}
{\sqrt{|B(x, t)|\, |B(y,t)|} } \int_{\rho(z,y)\ge \frac12\rho(x,y)}\cdots=:\cj_1+\cj_2.
\end{eqnarray*}
Since $\sz>d$, from the second inequality in (i), we deduce that
\begin{eqnarray*}
\cj_1
&&\le \frac{1}{\sqrt{|B(x, t)|\, |B(y,t)|} }
\frac{1}{[1+(2t)^{-1}{\rho(x,y)}]^\sz}
\int_M \frac1{|B(z,t)|}
\frac{1}{[1+t^{-1}{\rho(z,y)}]^\sz}\,d\mu(z)\\
&&\le 2^\sz \, \frac{K2^d}{1-2^{d-\sz}}\, D_{t,\sz}(x,y).
\end{eqnarray*}
Likewise, the same estimate also holds for $
\cj_2$.
From these and \eqref{eq:2.1x}, it follows the desired
estimate in (ii) under the assumption
that $s\le t$, which completes the proof of Lemma \ref{lem2.1x}.
\end{proof}

The following conclusion is just \cite[Proposition~2.9]{CKP}.

\begin{lem}\label{lem2.2x}
Let $\delta\in(0,\infty)$
and $\sigma\in[d+1,\fz)$.
Assume that the kernels of the integral operators
$U$ and $V$ satisfy
$|U(x,y)|\le D_{\delta, \sigma}(x,y)$ and $|V(x,y)|\le D_{\delta, \sigma}(x,y)$
for all $x,\,y\in M$. Then, for any operator $R$ which is bounded on $L^2(M)$,
the operator $URV$ is an integral operator with its kernel satisfying
$$|URV(x,y)|
\le\|U(x,\cdot)\|_{L^2(M)}\|R\|_{L^2(M)\to L^2(M)}
\|V(\cdot,y)\|_{L^2(M)}, \qquad x,\,y\in M.$$
\end{lem}

%%%%%%%%%%%%%%%%%%%%%%%%%%%%%%%%%%%%%%%%%%%%%%%%%%%%%%%%%%%%%%%%%%

\subsection{Smooth functional calculus induced by the heat kernels}\label{sec-2.2}

\hskip\parindent According to \cite[Theorem~7.3]{Ou}, the doubling condition \eqref{doubling1} and the assumption ${\bf (UE)}$ imply that the Gaussian upper bound can be extended to the open right half-plane $\cc^+$. Indeed, there exist positive constants
$C$ and $c$ such that, for all    $z:=t+iu$ with $t\in(0,1]$ and $u\in\rr$, and all $x,\,y\in M$,
\begin{eqnarray}\label{eq:2.2x}
|p_z(x,y)|
\le C \frac{\exp\big(-c\big[\frac{\rho(x,y)}{|z|^{1/\beta_0}}\big]
^{\frac{\beta_0}{\beta_0-1}}\cos \tz\big)}
{\sqrt{|B(x,t^{1/\beta_0})| \,|B(y, t^{1/\beta_0})|}},
\end{eqnarray}
where $\theta:={\rm arg}\, z$. The positive
constants $C$ and $c$ in \eqref{eq:2.2x} may depend on $\bz_0,\, C^\ast$ and $c^\ast$.

\begin{lem}\label{lem2.3x}
Let $k\in\nn$ and $g:\,\rr\to\cc$ be such that
 $ |g(\cdot)|(1+|\cdot|)^k\in L^1(\rr)$. Then there exists a positive constant
 $C$, independent of $g$ and $k$, such that,
 for all  $\dz\in(0,1]$ and  $x,\,y\in M$,
\begin{eqnarray*}
\lf|\int_\rr g(u) p_{\dz^{\bz_0}(1-iu)}(x,y)\,du\r|
\le (Ck)^k \,\| g(\cdot)(1+|\cdot|)^k\|_{L^1(\rr)}\, D_{\dz,k}(x,y).
\end{eqnarray*}
\end{lem}

\begin{proof}
For $z:= \dz^{\bz_0}(1-iu)$ with $\dz\in(0,1]$ and $u\in\rr$, we have $|z|\sim \dz^{\bz_0}(1+|u|)$ and
$\cos\arg z\sim [1+|u|]^{-1}$. Thus, for a different small positive
constant $c'\in(0,1]$, which may depend on $\bz_0$ and the constant $c$ in \eqref{eq:2.2x},  we have
\begin{eqnarray*}
\exp\left(-c\lf[\frac{\rho(x,y)}{|z|^{1/{\bz_0} }}\r]
^{\frac{{\bz_0} }{{\bz_0} -1}}\cos \tz\right)
&&\le \exp\left(-c'\lf[\frac{\rho(x,y)}{\dz(1+|u|)}\r]
^{\frac{{\bz_0} }{{\bz_0} -1}}\right).
\end{eqnarray*}
Let $\sz:=k({\bz_0}-1)/{\bz_0}$. We use $c'\in(0,1]$, together with the inequalities
$e^{-x}\le (\sz e^{-1})^{\sz} x^{-\sz}$ for all $x\in(0,\infty)$
and $(a+b)^q\le 2^{q-1} (a^q+b^q)$ for all $a,\,b\in\rr_+$ and $q\in[1,\infty)$
to conclude that
the right-hand side of  the above inequality is bounded by
\begin{eqnarray*}
e^{c'} (\sz (c'e)^{-1})^{\sz} \lf(1+\lf[\frac{\rho(x,y)}{\dz(1+|u|)}\r]
^{\frac{\bz_0}{{\bz_0} -1}}\r)^{-\sz}
&&\le e (\sz (c'e)^{-1})^{\sz} 2^{\sz(\frac{{\bz_0}}{{\bz_0}-1}-1)}
\lf[1+\frac{\rho(x,y)}{\dz}\r]^{-k}
\lf[1+|u|\r]^{k}.
\end{eqnarray*}
Notice that $e (\sz (c'e)^{-1})^{\sz} 2^{\sz(\frac{{\bz_0}}{{\bz_0}-1}-1)}
\le (\wz c \,k)^k$,
where $\wz c$ is some positive constant independent of  $k$.
From this and \eqref{eq:2.2x}, it follows that
\begin{eqnarray*}
\lf|\int_\rr g(u) p_{\dz^{\bz_0}(1-iu)}(x,y)\,du\r|
&&\le  (\wz c \,k)^k\, \frac{[ 1+  \frac{\rho(x,y)}{\dz}
]^{-k}}
{\sqrt{|B(x,\dz)| \,|B(y, \dz)|}} \,\int_\rr |g(u)|[1+|u|]^k\,du,
\end{eqnarray*}
as desired.
\end{proof}

Applying  Lemma \ref{lem2.3x},
we obtain the following result parallel to  \cite[Theorem 3.1]{CKP},
but invoking the new parameter $\bz_0\in[2,\fz)$.

\begin{lem}\label{lem2.4x}
 Let $k\in\nn$ and $g:\,\rr\to\cc$ be such that
$ |\widehat g(\cdot)|(1+|\cdot|)^k\in L^1(\rr)$.
For any $\dz\in(0,1]$, the operator $g(\dz^{\bz_0}L)e^{-\dz^{\bz_0} L}$
is an integral operator whose kernel $g(\dz^{\bz_0}L)e^{-\dz^{\bz_0} L}(x,y)$ satisfying
\begin{enumerate}
\item[\rm(i)]  for all   $x,\,y\in M$,
\begin{equation*}
|g(\dz^{\bz_0}L)e^{-\dz^{\bz_0} L}(x,y)|
\le (C k)^k \,\| \widehat g(\cdot)(1+|\cdot|)^k\|_{L^1(\rr)}\, D_{\dz,k}(x,y),
\end{equation*}
where $\widehat g$ denotes the Fourier transform of $g$;
\item[\rm(ii)]  when $k>2d$, for all $x,\,y,\,y'\in M$ with $\rho(y,y')\le \dz$,
\begin{eqnarray*}
&&|g(\dz^{\bz_0}L)e^{-\dz^{\bz_0} L}(x,y)-g(\dz^{\bz_0}L)e^{-\dz^{\bz_0} L}(x,y')|\\
&&\quad\le  (C k)^k \,\|\widehat g(\cdot)(1+|\cdot|)^k\|_{L^1(\rr)}\,
\lf[\frac{\rho(y,y')}{\dz}\r]^{\az_0}  D_{\dz,k}(x,y); \notag
\end{eqnarray*}
\item[\rm(iii)] for all $x\in M$,
\begin{equation*}
\int_Mg(\dz^{\bz_0}L)e^{-\dz^{\bz_0} L}(x,y)\,d\mu(y)=g(0).
\end{equation*}
\end{enumerate}
Here the positive constants $C$  in (i) and (ii) depend only on
$K,\, \az_0,\, \bz_0,\, C^\ast$ and $c^\ast$,
but independent of $g,\,k$, $\dz$ and $x,\,y,\,y'$.
\end{lem}

\begin{proof}
In order to see that
$g(\dz^{{\bz_0}}L)e^{-\dz^{\bz_0} L}$ has a kernel,
by \cite[Theorem~6, p.\,503]{DS} and the density of $L^1(M)\cap L^2(M)$ in $L^2(M)$,
we only need to prove that
\begin{equation}\label{eq:2.3x}
|\laz g(\dz^{{\bz_0}}L)e^{-\dz^{\bz_0} L} \phi,\, \psi\raz|
\le C \|\phi\|_{L^1(M)}\|\psi\|_{L^1(M)},\qquad \phi,\,\psi\in L^1(M)\cap L^2(M).
\end{equation}
Since $L$ is self-adjoint and positive definite, by the spectral resolution $\{E_\lz\}_{\lz\ge0}$ of the operator $L$, we have
\begin{eqnarray}\label{eq:2.4x}
\lf\laz g(\dz^{{\bz_0}}L)e^{-\dz^{\bz_0} L} \phi,\, \psi\r\raz
= \int_0^\infty g(\dz^{{\bz_0}}\lz)e^{-\dz^{\bz_0} \lz} \,d\laz E_\lz\phi,\, \psi\raz
= \int_\rr \widehat g(\xi) \laz e^{-\dz^{\bz_0}(1-2\pi i \xi)L} \phi,\,\psi\raz\,d\xi,
\end{eqnarray}
where, in the penultimate step,
to change the order of integration, we use Fubini's theorem and
\begin{eqnarray*}
 \int_0^\infty \int_\rr |\widehat g(\xi) e^{-\dz^{\bz_0} \lz+2\pi i \xi \dz^{{\bz_0}}\lz}|
 \,d\xi\,d|\laz E_\lz\phi,\, \psi\raz|
 &&\le \|\widehat g\|_{L^1(\rr)} \|\phi\|_{L^2(M)}\|\psi\|_{L^2(M)}.
\end{eqnarray*}
The operator $e^{-\dz^{\bz_0}(1-2\pi i \xi)L}$ has a complex kernel $p_{\dz^{\bz_0}(1-2\pi i \xi)}(x,y)$
and
$$
|p_{\dz^{\bz_0}(1-2\pi i \xi)}(x,y)|
\le C\frac1{\sqrt{|B(x,\dz)| \,|B(y, \dz)|}}
\le C \dz^{-d}\frac1{\sqrt{|B(x,1)| \,|B(y, 1)|}}
\le C \dz^{-d}<\infty
$$
by using \eqref{eq:2.2x}, \eqref{doubling2}, $\dz\in(0,1]$ and \eqref{non-collapsing},
which implies that
$$
\int_M\int_M \int_\rr |\widehat g(\xi)p_{\dz^{\bz_0}(1-2\pi i \xi)}(x,y)|
|\phi(y)||\psi(x)|\,d\xi\,d\mu(x)\,d\mu(y)
\le C \dz^{-d}\|\widehat g\|_{L^1(\rr)} \|\phi\|_{L^1(M)}\|\psi\|_{L^1(M)}.
$$
From this and  Fubini's theorem, we continue to write \eqref{eq:2.4x} as follows:
\begin{eqnarray*}
\lf\laz g(\dz^{{\bz_0}}L)e^{-\dz^{\bz_0} L} \phi,\, \psi\r\raz
&&=\int_M\int_M \lf[\int_\rr \widehat g(\xi)p_{\dz^{\bz_0}(1-2\pi i \xi)}(x,y)\,d\xi\r]
\phi(y)\psi(x)\,d\mu(x)\,d\mu(y).
\end{eqnarray*}
The above argument  implies that \eqref{eq:2.3x} holds true and the kernel
of the operator $g(\dz^{{\bz_0}}L)e^{-\dz^{\bz_0} L}$ is
\begin{equation}\label{eq:2.5x}
g(\dz^{{\bz_0}}L)e^{-\dz^{\bz_0} L}(x,y)
=\int_\rr \widehat g(\xi)p_{\dz^{\bz_0}(1-2\pi i \xi)}(x,y)\,d\xi.
\end{equation}

To prove (i), by \eqref{eq:2.5x}, \eqref{eq:2.2x}  and Lemma \ref{lem2.3x},
we know that, for all $x,\,y\in M$,
\begin{equation*}
|g(\dz^{{\bz_0}}L)e^{-\dz^{\bz_0} L}(x,y)|
\le \int_\rr |\widehat g(\xi)| |p_{\dz^{\bz_0}(1-2\pi i \xi)}(x,y)|\,d\xi
\le (Ck)^k\,\| \widehat g(\cdot)(1+|\cdot|)^k\|_{L^1(\rr)} \, D_{\dz,k}(x,y),
\end{equation*}
where the positive constant $C$ is the one same as
in the inequality of Lemma \ref{lem2.3x}.

Now we prove (ii).
Since
$g(\dz^{{\bz_0}}L)e^{-\dz^{\bz_0} L}=g(\dz^{{\bz_0}}L)e^{-\frac12\dz^{\bz_0} L}e^{-\frac12\dz^{\bz_0} L},$
we apply (i), {\bf (HE)}  and Lemma \ref{lem2.1x}(ii)
to conclude  that, for all $x,\,y,\,y'\in M$ with $\rho(y,y')\le \dz$,
\begin{eqnarray*}
&&|g(\dz^{{\bz_0}}L)e^{-\dz^{\bz_0} L}(x,y)-g(\dz^{{\bz_0}}L)e^{-\dz^{\bz_0} L}(x,y')|\\
&&\quad\le \int_M
|g(\dz^{{\bz_0}}L)e^{-\frac12\dz^{\bz_0} L}(x,z) |
|p_{\dz^{\bz_0}/2}(z,y)-p_{\dz^{\bz_0}/2}(z,y')|\,d\mu(z)\\
&&\quad\le C(K, d)\,( Ck)^k\,\| \widehat g(\cdot)(1+|\cdot|)^k\|_{L^1(\rr)}
\lf[\frac{\rho(y,y')}{\dz}\r]^{\az_0}\, D_{\dz,k}(x,y),
\end{eqnarray*}
where  $C(K, d)$ in the last inequality is a positive constant coming from Lemma \ref{lem2.1x}(ii)
and $k>2d$.

Finally, we prove (iii). Since $\int_M p_t(x,y)\,d\mu(y)= 1$ for all $x\in M$,
by the analytic continuation, we know that
$\int_M p_{t+iu}(x,y)\,d\mu(y)= 1$ for all $x\in M$ and $t+iu\in\cc^+$.
Thus, \eqref{eq:2.5x} gives us that, for all $x\in M$,
\begin{equation*}
\int_Mg(\dz^{{\bz_0}}L)e^{-\dz^{\bz_0} L}(x,y)\,d\mu(y)
=\int_M\int_\rr \widehat g(\xi)p_{\dz^{\bz_0}(1-2\pi i \xi)}(x,y)\,d\xi\,d\mu(y)
=\int_\rr \widehat g(\xi)\,d\xi=g(0).
\end{equation*}
This proves (iii) and hence finishes the proof of Lemma \ref{lem2.4x}.
\end{proof}

\begin{rem}\label{rem2.5x}
(i) Notice that the above two lemmas remain valid if $k\in\nn$ therein is replaced by
$\sz\in[1,\infty)$, but with the previous constants $(Ck)^k$ in
the inequalities in (i) and (ii) of Lemma \ref{lem2.4x} replaced
by $(C\sz)^\sz$.

(ii) For any $k>2d$, as was proved in \cite[p.\,1017 and Remark~3.2]{CKP}, we have
$$\| \widehat g(\cdot)(1+|\cdot|)^k\|_{L^1(\rr)}
\le C 2^k \lf(\|g\|_{L^1(\rr)} +\|g^{(k+2)}\|_{L^1(\rr)}\r)$$
for some positive constant $C$ independent of $k$ and $g$.
\end{rem}

Specially, as a corollary of Lemma \ref{lem2.4x}, we easily obtain the following result;
see \cite[Corollary~3.3]{CKP} for the case $\bz_0=2$.

\begin{cor}\label{cor2.6x}
Let $m\in\zz_+$ and $\sz\in(0,\infty)$.
Then there exists a positive constant
$C:=C(m,\sz)$ such that,
for all $\dz\in(0,1]$,
the kernel of the operator $(\dz^{\bz_0} L)^me^{-\dz^{\bz_0} L}$ satisfies
\begin{enumerate}
\item[\rm(i)]  for all   $x,\,y \in M$,
$|(\dz^{\bz_0} L)^me^{-\dz^{\bz_0} L}(x,y)|
\le C\, D_{\dz,\sz}(x,y);$
\item[\rm(ii)]  when $\sz>2d$, for all $x,\,y,\,y'\in M$ with $\rho(y,y')\le \dz$,
\begin{equation*}
|(\dz^{\bz_0} L)^me^{-\dz^{\bz_0} L}(x,y)-(\dz^{\bz_0} L)^me^{-\dz^{\bz_0} L}(x,y')|
\le C \lf[\frac{\rho(y,y')}{\dz}\r]^{\az_0} D_{\dz,\sz}(x,y).
\end{equation*}
\end{enumerate}
\end{cor}

\begin{proof}
Let $\theta\in C^\infty(\rr)$ such that $\supp\theta\subset [-1,\infty)$,
$0\le\theta\le1$ and $\theta(\lz)\equiv 1$ on $[0,\infty)$. For $\lz\in\rr$, let
$g(\lz):= \lz^m \theta(\lz)e^{-\lz}.$
By Remark \ref{rem2.5x}(ii), it is easy to show that
$\| \widehat g(\cdot)(1+|\cdot|)^\sz\|_{L^1(\rr)}=:C(m,\sz) <\infty.$ Also,
$(\dz^{\bz_0} L)^me^{-\dz^{\bz_0} L}
=2^mg(2^{-1}\dz^{\bz_0} L)e^{-2^{-1}\dz^{\bz_0} L},$
which, combined with Lemma \ref{lem2.4x}(i) and \eqref{doubling2},
yields that, for all $x,\, y\in M$,
\begin{eqnarray*}
|(\dz^{\bz_0} L)^me^{-\dz^{\bz_0} L}(x,y)|
=2^m|g(2^{-1}\dz^{\bz_0} L)e^{-2^{-1}\dz^{\bz_0} L}(x,y)|\le C(m,\sz)\, D_{\dz,\sz}(x,y).
\end{eqnarray*}
This proves (i).

The proof for (ii) follows in a similar way,
but using Lemma \ref{lem2.4x}(ii), the details being omitted,
which completes the proof of Corollary \ref{cor2.6x}.
\end{proof}

Let $C_c^\infty(\rr)$ be the space of all infinite differential functions with compact support.
The following lemma was presented in
\cite[Proposition~2.6]{KP}
(see also \cite[Theorem~2.3]{IPX}).

\begin{lem}\label{lem2.7x}
Let $\ez\in(0,1]$. Then there exists $\phi\in C_c^\infty(\rr)$ satisfying
the following conditions:
$\supp\phi\subset[0,2]$;
$\phi^{(m)}(0)=0$ for all $m\in\nn$;
$\phi\equiv 1$ on $[0,1]$, or $\supp\phi\subset [1/2, 2]$,
or $\supp\phi\subset[1/2,2]$ and $\sum_{j=0}^\infty \phi(2^{-j}t)=1$ for all $t\in[1,\infty)$;
$\|\phi\|_{L^\infty(\rr_+)}\le 1$;
$\|\phi^{(k)}\|_{L^\infty(\rr_+)}\le 8(16\ez^{-1} k^{1+\ez})^k$ for all $k\in\nn$.
\end{lem}

Using Lemmas \ref{lem2.4x} and \ref{lem2.7x}, with
slight modifications of the proofs of
\cite[Theorems~3.4]{CKP}, we obtain the following result.
Here we present the details of the proof because of the sensitivity of
the coefficient constants in the
inequalities in (i) and (ii) below.
When $\bz_0=2$, such delicate estimates were also given in \cite[Theorem~3.1]{KP} by using
the {\em finite speed propagation property} of the heat kernel,
which might fail when $\bz_0\neq 2$, so we avoid to use it
in the proof below.

\begin{lem}\label{lem2.8x}
Let $\ez\in(0,1]$, $k\in\nn$ and $k> 2d$. Assume that $f\in C^{2k+4}(\rr_+)$,
$\supp f\subset [0, R]$ with $R\ge1$, and
$f^{(2\nu+1)}(0)=0$ for $\nu\in\{0,1,\dots, k+1\}$.
Then the operator $f(\dz^{\bz_0/2}\sqrt L)$, $\dz\in(0,1]$,
is an integral operator with kernel $f(\dz^{\bz_0/2} \sqrt L)(x,y)$ satisfying
 \begin{enumerate}
\item[\rm(i)]  for all   $x,\,y\in M$,
$|f(\dz^{\bz_0/2}\sqrt L)(x,y)|
\le c_k\, D_{\delta,\,k}(x,y),$
where
$$c_k:=c_1 (R^{2/\bz_0})^{2k+d+4} (c_2 k^{1+\ez})^{2k+2} \lf[\|f\|_{L^\infty(\rr_+)}+\|f^{(2k+4)}\|_{L^\infty(\rr_+)}
+\max_{0\le j\le k+2} |f^{(2j)}(0)|\r].$$
\item[\rm (ii)] for all $x,\,y,\,y'\in M$ such that $\rho(y,y')\le \delta$,
\begin{equation*}
|f(\delta^{\bz_0/2} \sqrt L)(x, y)-f(\delta^{\bz_0/2} \sqrt L)(x, y')|
\le c_k' \lf[\frac{\rho(y,y')}{\delta}\r]^{\alpha_0}
D_{\delta, \,k}(x,y),
\end{equation*}
where $c_k':= c_3 \,c_k \,R^{2\az_0/\bz_0}.$
\item[\rm (iii)] for all $x\in M$,
$\int_M  f(\delta^{\bz_0/2} \sqrt L)(x, y)\,d\mu(y)=f(0).$
\end{enumerate}
Here the positive constants $c_1,\, c_2$ and $c_3$  in (i)
and (ii)  depend only on $\ez,\, d,\, K,\, \az_0,\, \bz_0,\, C^\ast$ and $c^\ast$,
but independent of $\dz,\, f,\, R$ and $k$.
\end{lem}

\begin{proof} Without loss of generality, we may assume that $R=1$;
otherwise, we let $h(\lz):=f(R\lz)$ and  do the argument
same as in the beginning of the proof of
\cite[Theorem~3.4]{CKP}.

Let $R=1$. Due to Lemma \ref{lem2.7x}, we choose an even smooth cut-off function $\eta$ such that
$\supp\eta\subset[-1,1]$, $\eta=1$ on $[-1/2,1/2]$, $0\le\eta\le1$
and
$\|\eta^{(k)}\|_{L^\infty(\rr)}\le C(C_\ez k^{1+\ez})^k$ for all $k\in\nn,$
where $C_\ez$ is a positive constant depending only on $\ez$.
Split $f=f_1+f_2$, where
$$f_1(\lz):= \eta(\lz^2) \sum_{j=0}^{k+2} \frac{f^{(2j)}(0)}{(2j)!}\lz^{2j}, \qquad \lz\in[0,1].$$
Write $f_i(\lz)=:g_i(\lz^2)e^{-\lz^2}$, $i\in\{1,2\}$. Then
$$g_1(\lz)= \eta(\lz) e^\lz \sum_{j=0}^{k+2} \frac{f^{(2j)}(0)}{(2j)!}\lz^{j}$$
and, by the Taylor expansion formula, we see that
\begin{eqnarray*}
g_2(\lz)
= e^\lz \lf\{ [1-\eta(\lz)]\sum_{j=0}^{k+2} \frac{f^{(2j)}(0)}{(2j)!}\lz^{j}
+\frac{(-1)^{2k+4}}{(2k+4)!} \int_0^{\sqrt \lz} (t-\sqrt \lz)^{2k+4} f^{(2k+4)}(t)\,dt\r\}.
\end{eqnarray*}
For $i\in\{1,2\}$, by (i) and (ii) of Lemma \ref{lem2.4x} and Remark \ref{rem2.5x}(ii), we obtain
\begin{eqnarray*}
|f_i(\dz^{\bz_0/2}\sqrt L)(x,y)|
\le C (Ck)^k  \lf[\|g_i\|_{L^1(\rr)}+ \|g_i^{(k+2)}\|_{L^1(\rr)}\r]\, D_{\dz, k}(x,y)
\end{eqnarray*}
and, when $\rho(y,y')\le \dz$,
\begin{eqnarray*}
&&|f_i(\delta^{\bz_0/2} \sqrt L)(x, y)-f_i(\delta^{\bz_0/2} \sqrt L)(x, y')|\\
&&\quad\le C (Ck)^k  \lf[\|g_i\|_{L^1(\rr)}+ \|g_i^{(k+2)}\|_{L^1(\rr)}\r] \lf[\frac{\rho(y,y')}{\delta}\r]^{\alpha_0}
D_{\delta, \,k}(x,y),
\end{eqnarray*}
where the positive
constant $C$ depends only on $d,\, K,\, \bz_0,\, \az_0,\, C^\ast$ and $c_\ast$.
Observe that the functions $f,\,f_1,\,f_2,\, g_1,\, g_2 $ and $\eta$ are all supported in $[-1,1]$. Thus, to obtain (i) and (ii) of Lemma \ref{lem2.8x}, we only need to prove that, for $i\in\{1,2\}$,
\begin{eqnarray}\label{eq:2.6x}
&&\|g_i\|_{L^\infty([0,1])}+ \|g_i^{(k+2)}\|_{L^\infty([0,1])}\notag\\
&&\quad\le C (Ck^{1+\ez})^{k+2} \lf[\|f\|_{L^\infty(\rr_+)}+\|f^{(2k+4)}\|_{L^\infty(\rr_+)}
+\max_{0\le j\le k+2} |f^{(2j)}(0)|\r].
\end{eqnarray}

Now we prove \eqref{eq:2.6x}.
Clearly, for any $\lz\in[0,1]$, we have
\begin{equation}\label{eq:2.7x}
|g_1(\lz)| \le e \lf\{\max_{0\le j\le k+2} |f^{(2j)}(0)|\r\}\, \sum_{j=0}^{k+2} \frac1{(2j)!}
\le 2e \max_{0\le j\le k+2} |f^{(2j)}(0)|
\end{equation}
and, by the fact $\supp f\subset [-1,1]$, we see that
\begin{equation}\label{eq:2.8x}
|g_2(\lz)|=|e^\lz f(\sqrt \lz)-g_1(\lz)| \le e \|f\|_{L^\infty(\rr)}+ 2e \max_{0\le j\le k+2} |f^{(2j)}(0)|.
\end{equation}
It remains to estimate $\|g_i^{(k+2)}\|_{L^\infty([0,1])}$, $i\in\{1,2\}$.
By Leibniz's rule, we see that
\begin{eqnarray*}
g_1^{(k+2)}(\lz)= e^\lz \sum_{N=0}^{k+2} \sum_{i=0}^N \binom {k+2} N  \,\binom N i \eta^{(N-i)}(\lz)  \sum_{j=i}^{k+2} \frac{f^{(2j)}(0)}{(2j)!}(\lz^{j})^{(i)},
\end{eqnarray*}
which, combined with the properties of $\eta$, yields that, when $\lz\in[0,1]$,
\begin{eqnarray*}
|g_1^{(k+2)}(\lz)|
\le 2C (3 C_\ez k^{1+\ez})^{k+2} \max_{0\le j\le k+2} |f^{(2j)}(0)|,
\end{eqnarray*}
by using the fact that $\sum_{N=0}^{k+2} \sum_{i=0}^N \binom {k+2} N  \,\binom N i = [1+(1+x)]^{k+2}\Big|_{x=1}=3^{k+2}$. Thus, we obtain
\begin{equation}\label{eq:2.9x}
\|g_1^{(k+2)}\|_{L^\infty([0,1])}
\le 2C ( 3C_\ez k^{1+\ez})^{k+2} \max_{0\le j\le k+2} |f^{(2j)}(0)|.
\end{equation}
To estimate $\|g_2^{(k+2)}\|_{L^\infty([0,1])}$,
via a little trivial calculations, we conclude that
\begin{eqnarray*}
g_2^{(k+2)}(\lz)
&&= e^\lz \sum_{N=0}^{k+2}
\sum_{i=0}^N \binom {k+2} N \binom N i [1-\eta(\lz)]^{(i)} \sum_{j=N-i}^{k+2}
\frac{f^{(2j)}(0)}{2^j (j-N+i)!}\,\lz^{j-N+i+1}\\
&&\quad
+ e^\lz \frac{(-1)^{2k+4}}{(2k+4)!} \sum_{N=0}^{k+2} \binom {k+2} N  \int_0^{\sqrt \lz} \frac{d^N(t-\sqrt \lz)^{2k+4}}{d\lz^N} f^{(2k+4)}(t)\,dt =: \cj_1+\cj_2.
\end{eqnarray*}
Using the properties of $\eta$, we find that
\begin{eqnarray*}
\cj_1
\le C(3C_\ez k^{1+\ez})^{k+2} \max_{0\le j\le k+2}|f^{(2j)}(0)| .
\end{eqnarray*}
For $\cj_2$,
we  use the following  estimate (see \cite[p.\,1019]{CKP}): for all $F\in C^N(\rr)$ and $\lz\in\rr$,
$$
\frac{d^N}{d\lz^N} F(\sqrt\lz) =\sum_{j=1}^N c_j \lz^{-N+j/2} F^{(j)}(\sqrt\lz),\qquad |c_j|\le N!,
$$
to conclude  that, for all $\lz\in[0,1]$,
\begin{eqnarray*}
\cj_2 &&\le e \|f^{(2k+4)}\|_{L^\infty(\rr)} \sum_{N=0}^{k+2} \binom {k+2} N \sum_{j=1}^N \frac{N!}{(2k+5-j)!}
 \le C (Ck)^{k+2}\,\|f^{(2k+4)}\|_{L^\infty(\rr)}.
\end{eqnarray*}
Combining the estimates of $\cj_1$ and $\cj_2$, we know that
\begin{equation}\label{eq:2.10x}
\|g_2^{(k+2)}\|_{L^\infty([0,1])} \le
C (C_\ez k^{1+\ez})^{k+2}\,
\lf[\sup_{0\le j\le k+2}|f^{(2j)}(0)|+\|f^{(2k+4)}\|_{L^\infty(\rr)}\r].
\end{equation}
From \eqref{eq:2.7x}, \eqref{eq:2.8x}, \eqref{eq:2.9x} and \eqref{eq:2.10x},
it follows \eqref{eq:2.6x}. This finishes the proofs of (i) and (ii).

Finally, (iii) follows from Lemma \ref{lem2.4x}(iii). Thus, we complete the proof of
Lemma \ref{lem2.8x}.
\end{proof}

As a consequence of Lemmas \ref{lem2.8x} and \ref{lem2.2x}, if we proceed as in the proof of \cite[Theorem~3.7]{CKP}, then we obtain
properties of kernels of operators of the form $f(\sqrt L)$, with $f$ being some non-smooth compactly supported function, the details of the proof being omitted.

\begin{cor}\label{cor2.9x}
Let $f$ be a bounded measurable function on $\rr_+$ with $\supp f\subset [0,\tau^{\bz_0/2}]$
for some $\tau\ge1$. Then $f(\sqrt L)$ is an integral operator with kernel $f(\sqrt L)(x,y)$ satisfying
\begin{eqnarray*}
|f(\sqrt L)(x,y)| \le C_\flat\, \frac{\|f\|_{L^\infty(\rr)}}{\sqrt{|B(x,\tau^{-1})| \, |B(y, \tau^{-1})|}}, \qquad x,\,y\in M
\end{eqnarray*}
and, when $x,\ y,\ y'\in M$ such that $\rho(y,y')\le \tau^{-1}$,
\begin{eqnarray*}
|f(\sqrt L)(x,y)-f(\sqrt L)(x,y')| \le C_\flat\, \frac{[\tau\rho(y,y')]^{\az_0}\,\|f\|_{L^\infty(\rr)}}{\sqrt{|B(x,\tau^{-1})|\, |B(y, \tau^{-1})|}},
\end{eqnarray*}
where $C_\flat$ is a positive constant independent of $\tau$, $f$ and $x,\,y,\,y'\in M$.
\end{cor}

Using Lemma \ref{lem2.8x} and following the proof of
\cite[Theorem~3.4]{KP}, we find
the following properties of the kernel of the operator $f(\dz^{\bz_0/2}\sqrt L)$,
with $f$ being some smooth function with fast decay at infinity,
the details of the proof being omitted.

\begin{prop}\label{prop2.10x}
Let $f\in C^\infty(\rr_+)$ satisfy, for all $\nu\in\zz_+$ and $r\in(0,\infty)$,
there is a positive constant $C_{\nu,r}$ such that
\begin{equation}\label{eq:2.11x}
f^{(2\nu+1)}(0)=0,\qquad|f^{(\nu)}(\lz)|\le C_{\nu,r}(1+\lz)^{-r}\qquad
\textup{for all}\, \,\lz\in(0,\infty).
\end{equation}
Let $\sigma\in(0,\infty)$. Then, for any $m\in\zz_+$ and $\delta\in(0,1]$,
there exists a constant $C_{\sigma, m}\in(1,\infty)$ such that
the operator $L^mf(\delta\sqrt L)$ is an integral operator
with kernel $L^m f(\delta \sqrt L)(x, y)$
satisfying:
\begin{enumerate}
\item[\rm (i)] for all $x,\,y\in M$,
\begin{equation}\label{eq:2.12x}
|L^m f(\delta^{\bz_0/2} \sqrt L)(x, y)|
\le C_{\sigma, m}\, \delta^{-\bz_0 m} D_{\delta, \sigma}(x,y);
\end{equation}
\item[\rm (ii)] for all $x,\,y,\,y'\in M$ such that $\rho(y,y')\le \delta$,
\begin{equation}\label{eq:2.13x}
|L^m f(\delta^{\bz_0/2} \sqrt L)(x, y)-L^m f(\delta^{\bz_0/2} \sqrt L)(x, y')|
\le C_{\sigma, m}\, \delta^{-\bz_0 m} \lf[\frac{\rho(y,y')}{\delta}\r]^{\alpha_0}
D_{\delta, \sigma}(x,y);
\end{equation}
\item[\rm (iii)] for all $x\in M$,
$\int_M  f(\delta^{\bz_0/2} \sqrt L)(x, y)\,d\mu(y)= f(0).$
\end{enumerate}
\end{prop}

\begin{rem}\label{rem2.11x}
(i) In \eqref{eq:2.11x}, the condition
$f^{(2\nu+1)}(0)=0$ for all $\nu\in\zz_+$ ensures that the function
$f$ has an even extension to $\rr$.
Conversely, any even function in the Schwartz class
satisfies \eqref{eq:2.11x}.

(ii) Notice that the constant $C_{\sz,m}$ in \eqref{eq:2.12x} and \eqref{eq:2.13x} depends also on
the parameters $K$, $\az_0$, $\bz_0$, $C^\ast$, $c^\ast$ and $C_{\nu,r}$ in \eqref{eq:2.11x}; but, for simplicity, we write only $C_{\sz,m}$
since, for most of the cases, we care only the parameters $\sz$ and $m$.
We point out that $C_{\sz,m}$ is independent of $\dz\in(0,1]$.
\end{rem}

For functions $\phi$ in Lemma \ref{lem2.7x}, from
Lemma  \ref{lem2.8x} and Corollary \ref{cor2.9x}, we deduce
some nice properties of kernels of $L^m\phi(\dz^{\bz_0/2}\sqrt L)$,
via a slight modification of the proof of \cite[Theorem~3.6]{KP} as follows.

\begin{prop}\label{prop2.12x}
Let $\ez\in(0,1]$ and $\phi$ be a smooth function satisfying the conditions of Lemma \ref{lem2.7x}.
For any $m\in\nn$,
there exist positive constants $\gz:=\gz(m,\ez)$
and $C^{\flat}_{\gz,\ez,m}:=C(\gz,\ez,m)$
such that, for all  $\delta\in(0,1]$,
the operator $L^m\phi(\delta^{\bz_0/2}\sqrt L)$ is an integral operator
with kernel $L^m \phi(\delta^{\bz_0/2} \sqrt L)(x, y)$
satisfying:
\begin{enumerate}
\item[\rm (i)]  for all $x,\,y\in M$,
\begin{equation*}
|L^m \phi(\delta^{\bz_0/2} \sqrt L)(x, y)|
\le C^{\flat}_{\gz,\ez,m}
 \delta^{-\bz_0m} \frac{1}
{\sqrt{|B(x, \delta)|\, |B(y,\delta)|} }\,\exp\lf\{-\gz\left[\frac{\rho(x,y)}{\delta}\right]^{\frac1{10(1+\ez)}}\r\};
\end{equation*}

\item[\rm (ii)] for all $x,\,y,\,y'\in M$ such that $\rho(y,y')\le \delta$,
\begin{eqnarray*}
&&|L^m\phi(\delta^{\bz_0/2} \sqrt L)(x, y)-L^m\phi(\delta^{\bz_0/2} \sqrt L)(x, y')|\\
&&\qquad\le C^{\flat}_{\gz,\ez,m}
 \delta^{-\bz_0m}
\lf[\frac{\rho(y,y')}{\delta}\r]^{\alpha_0} \frac{1}
{\sqrt{|B(x, \delta)|\, |B(y,\delta)|} }\,\exp\lf\{-\gz\left[\frac{\rho(x,y)}{\delta}\right]^{\frac1{10(1+\ez)}}\r\}.
\end{eqnarray*}
\end{enumerate}
\end{prop}

\begin{proof}
For $\phi$ as in Lemma \ref{lem2.7x},
consider the function $\psi(\lz):=\lz^{2m}\phi(\lz)$ for all $\lz\in[0,\fz)$.
Since $\phi$ satisfies the conditions of Lemma \ref{lem2.7x},
it is easy to show that $\|\psi\|_{L^\infty(\rr)}\le 2^{2m}$
and $\|\psi^{(k)}\|_{L^\infty(\rr)}\le C_m\, (ck)^{k(1+\ez)}$
uniformly in $k\in\nn$ for some
positive constants $C_m$ and $c$ .
A further calculation leads to that, for the kernel $\psi(\dz^{\bz_0/2}\sqrt L)(x,y)$,
the corresponding positive constant $c_k$ in Lemma \ref{lem2.8x}(i) is controlled by $C(ck)^{(4k+6)(1+\ez)}$, where $C$ and $c$ are positive constants depending on $\ez$ and $m$, but independent of $\dz$ and $x,\,y\in M$. Thus, by Lemma \ref{lem2.8x}(i), we know
that, for all $k\in\nn$ and $x,\,y\in M$,
\begin{eqnarray*}
|\psi(\dz^{\bz_0/2}\sqrt L)(x,y)|\le C \frac{(ck)^{(4k+6)(1+\ez)}}
{\sqrt{|B(x, \delta)|\, |B(y,\delta)|} \,[1+\dz^{-1}\rho(x,y)]^k}.
\end{eqnarray*}
Observe that $(4k+6)(1+\ez)\le 10 k(1+\ez)$ for all $k\in\nn$.
Thus, when $\dz^{-1}\rho(x,y)\ge e(ck)^{10(1+\ez)}=:c_\ast k^{10(1+\ez)}$, we have
\begin{eqnarray*}
|\psi(\dz^{\bz_0/2}\sqrt L)(x,y)|\le C \frac{e^{-k}}
{\sqrt{|B(x, \delta)|\, |B(y,\delta)|} }.
\end{eqnarray*}
When $\dz^{-1}\rho(x,y)\ge 2^{20}c_\ast$,
there exists  $k\in\nn$ such that
$k-1\le [\frac{\rho(x,y)}{c_\ast\,\dz}]^{\frac1{10(1+\ez)}}<k$,
which further implies that
\begin{eqnarray*}
|\psi(\dz^{\bz_0/2}\sqrt L)(x,y)|\le C \frac{\exp\lf(-[\frac{\rho(x,y)}{c_\ast\,\dz}]^{\frac1{10(1+\ez)}} \r)}
{\sqrt{|B(x, \delta)|\, |B(y,\delta)|} }.
\end{eqnarray*}
When $\dz^{-1}\rho(x,y)< 2^{20}c_\ast$, we use Corollary \ref{cor2.9x} to conclude that
\begin{eqnarray*}
|\psi(\dz^{\bz_0/2}\sqrt L)(x,y)|\le \frac{C_\flat \|\psi\|_{L^\infty(\rr)}}
{\sqrt{|B(x, \delta)|\, |B(y,\delta)|} }
\le C \frac{\exp\lf(-[\frac{\rho(x,y)}{c_\ast\,\dz}]^{\frac1{10(1+\ez)}} \r)}
{\sqrt{|B(x, \delta)|\, |B(y,\delta)|} }.
\end{eqnarray*}
Combining the last two formulae and noticing that
$\dz^{\bz_0m}L^m \phi(\delta^{\bz_0/2} \sqrt L)=\psi(\dz^{\bz_0/2}\sqrt L)$,
we obtain  the desired estimate in (i).

The H\"older continuity in (ii) follows in a similar manner, but using the constant $c_k'$ associated to $\psi(\dz^{\bz_0/2}\sqrt L)(x,y)$ in  Lemma \ref{lem2.8x}(ii).
This finishes the proof of Proposition \ref{prop2.12x}.
\end{proof}

When $\mu(M)<\infty$, we use  $\cd(M)$
to denote the \emph{test function space}
which consists of all functions $\phi\in \cap_{m\in\zz_+} {\rm Dom}(L^m)$
with topology induced by
$\cp_m(\phi):=\|L^m\phi\|_{L^2(M)}$ for all $m\in\zz_+.$
Let $x_0$ be some fixed point in $M$. When $\mu(M)=\infty$, the \emph{test function space} $\cd(M)$ consists
of all functions $\phi\in \cap_m {\rm Dom}(L^m)$
such that, for all $m,\,\ell\in\zz_+$,
$$\cp_{m,\ell}(\phi):= \sup_{x\in M} [1+\rho(x, x_0)]^\ell|L^m\phi(x)|<\infty.$$
No matter $\mu(M)$ is finite or not,
the distribution space $\cd'(M)$ is defined to be the set of all continuous linear functionals on $\cd(M)$. We use $\laz f,\, \phi\raz$ to denote
the pairing between $f\in\cd'(M)$ and $\phi\in\cd(M)$. If both $f$ and $\phi$ are in $L^2(M)$,
 then $\laz f,\, \phi\raz$ is understood in the following way:
$$\laz f, \phi\raz= \int_M f(x)\phi(x)\,d\mu(x).$$
 Due to Proposition \ref{prop2.10x}, the proofs of  \cite[Propositions~5.2(a) and 5.4(a)]{KP}
 imply the following conclusions, the details being omitted.

\begin{prop}\label{prop2.13xx}
\begin{enumerate}
\item[\rm (i)] For any even function $\phi\in \cs(\rr)$,
the kernel $\phi(\sqrt L)(x,y)$
of the operator $\phi(\sqrt L)$
belongs to $\cd(M)$ as a function of $x$ or $y$.

\item[\rm (ii)] If $\mu(M)<\infty$, then
$f\in\cd'(M)$ if and only if
there exist $m_0\in\zz_+$ and $C_{m_0}\in(0,\fz)$
such that
\begin{equation*}
|\laz f,\,\phi\raz|
\le C_{m_0}\, \max_{0\le m\le m_0} \cp_m(\phi)
\qquad \textup{for all}\; \phi\in \cd(M).
\end{equation*}

\item[\rm (iii)] If $\mu(M)=\infty$, then $f\in\cd'(M)$ if and only if
there exist $\ell_0,\, m_0\in\zz_+$  and $C_{m_0,\ell_0}\in(0,\fz)$
such that
\begin{equation*}
|\laz f,\,\phi\raz|
\le C_{m_0,\ell_0}\, \max_{0\le m\le m_0,\, 0\le \ell\le \ell_0} \cp_{m,\ell}(\phi) \qquad\textup{for all}\; \phi\in \cd(M).
\end{equation*}
\end{enumerate}
\end{prop}

Let $\delta\in(0,1)$. Assume that there exists a positive constant $c$ such that the pair of functions $(\Phi_0, \Phi)$ in
$C^\infty(\rr_+)$ satisfies
\begin{equation}\label{eq:2.14xx}
\supp \Phi_0\subset [0, \delta^{-\bz_0/2}],\quad
\Phi_0^{(2\nu+1)}(0)=0\ \textup{for all }\ \nu\in\zz_+,\, \quad
|\Phi_0(\lz)|\ge c\ \textup{for}\ \lz\in[0, \delta^{-3\bz_0/8}]
\end{equation}
and
\begin{equation}\label{eq:2.15xx}
\supp \Phi\subset [\delta^{\bz_0/2}, \delta^{-\bz_0/2}],\quad \, \quad
|\Phi(\lz)|\ge c\ \,\,\textup{for}\,\,\lz\in[\delta^{3\bz_0/8}, \delta^{-3\bz_0/8}].
\end{equation}
For all $j\in\nn$ and $\lz\in\rr_+$, let
\begin{equation}\label{eq:2.16xx}
\Phi_j(\lz):=\Phi(\delta^{j\bz_0/2}\lz).
\end{equation}
According to \cite[Lemma~6.10]{FJW} (see also \cite[p.\,1487,\, (3.20)]{BH}),
there exist functions  $(\wz\Phi_0, \wz\Phi)\in C^\infty(\rr_+)$
satisfying \eqref{eq:2.14xx} and \eqref{eq:2.15xx}
such that
$$\sum_{j=0}^\infty \wz\Phi_j(\lz)\Phi_j(\lz)=1
\qquad \textup{for all}\,\,\lz\in\rr_+,$$
where $\wz\Phi_j$ for $j\ge1$ is defined as in \eqref{eq:2.16xx}.
Based on the proof of \cite[Lemma~6.10]{FJW},
if $\Phi_0$ and $\Phi$ are required to be nonnegative, then $\wz{\Phi}_0$
and $\wz{\Phi}$ are also nonnegative. By Proposition \ref{prop2.10x},
every $\Phi_j(\sqrt L)$ is an integral
operator  with kernel
$\Phi_j(\sqrt L)(x,\cdot)\in\cd(M)$ for any given $x\in M$,
so that it makes sense to consider $\Phi_j(\sqrt L)f (x):=
\laz f,\, \Phi_j(\sqrt{L})(x,\cdot)\raz$ for all $f\in\cd'(M)$ and $x\in M$.
Consequently, by \cite[Proposition~5.5(b)]{KP}, for all $f\in\cd'(M)$
(or $f\in L^p(M)$ with $p\in[1,\infty)$, or $f\in\cd(M)$),  we have
\begin{eqnarray}\label{c-crf}
f=\sum_{j=0}^\infty \wz\Phi_j(\sqrt L)\Phi_j(\sqrt L) f
=\sum_{j=0}^\infty \Phi_j(\sqrt L)\wz\Phi_j(\sqrt L) f
\quad \textup{in}\ \cd'(M) \quad(\textup{or}\ \ L^p(M)\ \textup{or}\  \cd(M)).
\end{eqnarray}
This is usually called the \emph{continuous Calder\'on reproducing formula}.
As was seen from the articles \cite{BH, FJ90, FJW, YSY} (see also \cite{s012,s013}), the Calder\'on reproducing formula
and the following almost orthogonal estimates
serve as  powerful tools in the study for function spaces.

\begin{prop}\label{prop2.14x}
Let $\dz\in(0,1)$.
Assume that  $(\Phi_0,\Phi)$ and $(\Psi_0,\Psi)$
satisfy \eqref{eq:2.14xx} and \eqref{eq:2.15xx}, respectively. For any $j\in\nn$,
define
$\Phi_j$
and $\Psi_j$ as in \eqref{eq:2.16xx}.
Then, for any  $m\in\zz_+$ and $\sz>d$,
there exists a constant $C_{\sz,m}\in(1,\infty)$
such that, for all $j,\,k\in\nn$  and  $x,\,y\in M$,
\begin{eqnarray}\label{eq:2.18x}
\left|(\Phi_j( \sqrt L)
\Psi_k(\sqrt L))(x, y)\right|
\le C_{\sigma, m} \delta^{|k-j|(m\bz_0-d)} D_{\delta^{k\wedge j}, \sigma}(x,y).
\end{eqnarray}
\end{prop}

\begin{proof}
If $k=j$, then \eqref{eq:2.18x} follows from Proposition \ref{prop2.10x}(i) and
Lemma \ref{lem2.1x}(ii). By symmetry, it suffices to show \eqref{eq:2.18x} for  $k>j$.
When $j>0$, the functional calculus gives us that
\begin{eqnarray*}
\Phi_j( \sqrt L)\Psi_k( \sqrt L)
=(\Phi_j\Psi_k)( \sqrt L)= \delta^{(k-j)m\bz_0} \phi_j( \sqrt L)\psi_k( \sqrt L),
\end{eqnarray*}
where $\phi_j(\sqrt\lz):= (\delta^{j\bz_0/2} \sqrt \lz)^{2m}\Phi(\delta^{j\bz_0/2}  \sqrt \lz)$
and $\psi_k(\sqrt\lz):= (\delta^{k\bz_0/2} \sqrt \lz)^{-2m}\Psi(\delta^{k\bz_0/2}  \sqrt \lz)$.
Then, Proposition \ref{prop2.10x}(i) implies that
$|\phi_j(\sqrt L)(x,y)|\ls D_{\delta^{j}, \sigma} (x,y)$ and
$|\psi_k(\sqrt L)(x,y)|\ls D_{\delta^{k},\sigma} (x,y)$
uniformly in $j,\,k\in\zz_+$ and $x,\,y\in M$.
Applying $\sigma\in(d,\fz)$ and Lemma \ref{lem2.1x}(ii), we conclude that
\begin{eqnarray*}
\lf|(\phi_j( \sqrt L)\psi_k( \sqrt L))(x,y)\r|
&& \ls  \int_M D_{\delta^{j}, \sigma} (x,u)
D_{\delta^{k},\sigma}(u,y)\,d\mu(u)
\ls \delta^{(j-k)d}D_{\delta^{j}, \sigma} (x,y).
\end{eqnarray*}
Consequently, for all $x,\,y\in M$,
\begin{eqnarray*}
|\Phi_j( \sqrt L)\Psi_k( \sqrt L)(x,y)|
\ls \delta^{(k-j)(m\bz_0-d)} D_{\delta^j, \sigma} (x,y).
\end{eqnarray*}
A similar argument as above also implies the desired result for $j=0$, the details being omitted.
This finishes the proof of Proposition \ref{prop2.14x}.
\end{proof}

\begin{rem}\label{rem2.15x}
In the hypothesis of Proposition \ref{prop2.14x}, instead of assuming that $(\Phi_0,\Phi)$ satisfy \eqref{eq:2.14xx} and \eqref{eq:2.15xx}, we
may assume  that $\Phi_0$ and $\Phi$ are smooth functions
satisfy \eqref{eq:2.11x}, then
the proof of Proposition \ref{prop2.14x} implies that
\eqref{eq:2.18x} remains valid when $k\ge j$.
\end{rem}

%%%%%%%%%%%%%%%%%%%%%%%%%%%%%%%%%%%%%%%%%%%%%%%%%%%%%%%%%%%%%%%%%%%%%

\section{Besov-type and Triebel--Lizorkin-type spaces }\label{sec-3}

\hskip\parindent
We first recall
Christ's dyadic cubes (see \cite{chr}) on the space of homogeneous type. Such dyadic cubes retain
most of the properties of the dyadic cubes in the Euclidean space.

\begin{lem}\label{lem3.1x}
Let $(M, \rho, \mu)$ be a space of homogeneous type. Then there exist a
collection $\cq:=\{Q_\az^j\subset M:\, j\in\zz,\, \az\in I_j\}$ of
open subsets, where $I_j$ is some index set, and positive constants
$\dz\in(0, 1)$ and $C_\natural>c_\natural$ such that
\begin{enumerate}
\item[ \rm (i)] for each fixed $j\in\zz$, $\mu(M\setminus\bigcup_\az Q_\az^j)=0$  and
$Q_\az^j\cap Q_\bz^j=\emptyset$ if $\az\neq\bz$;
\item[ \rm (ii)] for any $\az$, $\bz$, $j$, $\ell$ with $\ell\ge j$, either
$Q_\bz^\ell\subset Q_\az^j$ or $Q_\bz^\ell\cap Q_\az^j=\emptyset$;
\item[ \rm (iii)] for each  $(j, \az)$ and each $\ell<j$, there is a unique
$\bz$ such that $Q_\az^j\subset Q_\bz^\ell$;
\item[ \rm (iv)] $\diam Q_\az^j:=\sup_{x,\,y\in Q_\az^j}\,d(x,y)\le C_\natural\,\dz^j$ and
each $Q_\az^j$ contains some ball $B(z_\az^j, c_\natural\dz^j)$, where
$z_\az^j\in M$.
\end{enumerate}
\end{lem}

In what follows, we always use  $\delta$ to denote the constant same as in Lemma \ref{lem3.1x}.
Then any Christ dyadic cube $Q^k_\az$ with $k\in\zz$ and $\az\in I_k$
has diameter roughly $\delta^k$.

\begin{defn}\label{def3.2x}
Let $s\in \rr$, $\tau\in[0,\fz)$ and $q\in(0,\infty]$. Let
 $\dz\in(0,1)$ be as in Lemma \ref{lem3.1x}.
 Assume that  $(\Phi_0, \Phi)$ in
$C^\infty(\rr_+)$ satisfy \eqref{eq:2.14xx} and \eqref{eq:2.15xx}.
For $j\in\nn$, define $\Phi_j$ as in \eqref{eq:2.16xx}.
\begin{enumerate}
\item[ (i)] Let $p\in(0,\infty]$. The \emph{Besov-type space $B^{s,\tau}_{p, q}(M)$} is defined
to be the collection of all $f\in\cd'(M)$ such that
$$\|f\|_{B^{s,\tau}_{p, q}(M)}:=
\sup_{\gfz{k\in\zz}{\az\in I_k}} \frac1{|Q^k_\az|^\tau}
\bigg\{\sum_{j=k\vee0}^\fz\bigg[\int_{Q^k_\az}
|\delta^{-js}\Phi_j(\sqrt L) f(x)|^p\,d\mu(x)\bigg]^{q/p}\bigg\}^{1/q}<\infty$$
with a usual modification made when $p=\infty$ or $q=\infty$.
The \emph{Besov-type space $\wz B^{s,\tau}_{p, q}(M)$} is defined
to be the collection of all $f\in\cd'(M)$ such that
$$\|f\|_{\wz B^{s,\tau}_{p, q}(M)}:=
\sup_{\gfz{k\in\zz}{\az\in I_k}} \frac1{|Q^k_\az|^\tau}
\bigg\{\sum_{j=k\vee0}^\fz \bigg[\int_{Q^k_\az}
|B(x, \delta^{j})|^{-sp/d}|\Phi_j(\sqrt L) f(x)|^p\,d\mu(x)\bigg]^{q/p}\bigg\}^{1/q}
<\infty$$
with a usual modification made when $p=\infty$ or $q=\infty$.

\item[(ii)] Let $p\in(0,\infty)$. The \emph{Triebel--Lizorkin-type space $F^{s,\tau}_{p, q}(M)$} is defined
to be the collection of all $f\in\cd'(M)$ such that
$$\|f\|_{F^{s,\tau}_{p, q}(M)}:=
\sup_{\gfz{k\in\zz}{\az\in I_k}} \frac1{|Q^k_\az|^\tau}
\bigg\{\int_{Q^k_\az} \bigg[\sum_{j=k\vee0}^\fz
|\delta^{-js}\Phi_j(\sqrt L) f(x)|^q\bigg]^{p/q}\,d\mu(x)\bigg\}^{1/p}<\infty$$
with a usual modification made when  $q=\infty$.
The \emph{Triebel--Lizorkin-type space $\wz F^{s,\tau}_{p, q}(M)$} is defined
to be the collection of all $f\in\cd'(M)$ such that
$$\|f\|_{\wz F^{s,\tau}_{p, q}(M)}:=
\sup_{\gfz{k\in\zz}{\az\in I_k}} \frac1{|Q^k_\az|^\tau}
\bigg\{\int_{Q^k_\az} \bigg[\sum_{j=k\vee0}^\fz
|B(x, \delta^{j})|^{-sq/d}|\Phi_j(\sqrt L) f(x)|^q\bigg]^{p/q}\,d\mu(x)\bigg\}^{1/p}<\infty$$
with a usual modification made when  $q=\infty$.
\end{enumerate}
\end{defn}

\begin{rem}\label{rem3.3x}
(i) For $p=q\in(0,\infty)$, we have
$ B^{s,\tau}_{p, p}(M)= F^{s,\tau}_{p, p}(M)$ and
$\wz B^{s,\tau}_{p, p}(M)= \wz F^{s,\tau}_{p, p}(M).$

(ii) In general, even when $\tau=0$,
the spaces $B^{s,\tau}_{p, q}(M)$ and  $\wz B^{s,\tau}_{p, q}(M)$
 may not coincide with each other,
unless $(M,\rho,\mu)$ is an Ahlfors $d$-regular metric measure space (that is, $\mu(B(x,r))\sim r^d$ for all $x\in M$ and $r>0$).
Neither do $F^{s,\tau}_{p, q}(M)$ and $\wz F^{s,\tau}_{p, q}(M)$.
\end{rem}

\begin{rem}\label{rem3.4x}
When $\mu(M)=\infty$ and $\tau\in(-\infty,0)$, it is easy to see that
$$B^{s,\tau}_{p, q}(M)=F^{s,\tau}_{p, q}(M)=\wz B^{s,\tau}_{p, q}(M)=\wz F^{s,\tau}_{p, q}(M)=\{0\}.$$
But, when $\mu(M)<\infty$ and $\tau\in(-\infty,0)$,
it holds true that
$$B^{s,\tau}_{p, q}(M) = B^{s}_{p, q}(M),
\quad \wz B^{s,\tau}_{p, q}(M) = \wz B^{s}_{p, q}(M),
\quad F^{s,\tau}_{p, q}(M) = F^{s}_{p, q}(M),
\quad \wz F^{s,\tau}_{p, q}(M) = \wz F^{s}_{p, q}(M).$$
No matter $\mu(M)$ is finite or not, when $\tau=0$, it always holds true that
$$B^{s,0}_{p, q}(M) = B^{s}_{p, q}(M),
\quad \wz B^{s,0}_{p, q}(M) = \wz B^{s}_{p, q}(M),
\quad F^{s,0}_{p, q}(M) = F^{s}_{p, q}(M),
\quad \wz F^{s,0}_{p, q}(M) = \wz F^{s}_{p, q}(M).$$
When $\tau\in(1/p,\infty)$, in Proposition \ref{prop4.9x} below,
we  show that $$B_{p,q}^{s,\tau}(M)=F_{p,q}^{s,\tau}(M)=B_{\fz,\fz}^{s+d\tau-d/p}(M)
\quad \textup{and}\quad
\wz B_{p,q}^{s,\tau}(M)=\wz F_{p,q}^{s,\tau}(M)= \wz B_{\fz,\fz}^{s+d\tau-d/p}(M).$$
Thus, when $\tau\in(0,1/p]$,  Definition  \ref{def3.2x}
gives new scales of function spaces.
Especially, when $\tau=1/p$, it is proved in Theorem \ref{thm7.8x}
below that $F_{p,q}^{s,1/p}(M)$ and $\wz F_{p,q}^{s,1/p}(M)$
are respectively the endpoint cases of the Triebel--Lizorkin spaces
$F_{\infty,q}^s(M)$ and $\wz F_{\infty,q}^s(M)$.
\end{rem}

%%%%%%%%%%%%%%%%%%%%%%%%%%%%%%%%%%%%%%%%%%%%%%%%%%%%%%%%%%%%%%%%%%%%%

\section{Properties of Besov--Triebel--Lizorkin-type spaces }\label{sec-4}

\hskip\parindent
In this section, we first prove some estimates for the Peetre maximal function. As an application, we  establish some embedding properties of Besov-type and Triebel--Lizorkin-type spaces, and then we identify Besov-type and Triebel--Lizorkin-type spaces with the smooth index  in different ranges.

%%%%%%%%%%%%%%%%%%%%%%%%%%%%%%%%%%%%%%%%%%%%%%%%%%%%%%%%%%%%%%%%%%%%%%%%%%%%

\subsection{Peetre maximal functions}\label{sec-4.1}

\hskip\parindent
An important tool to deal with Besov-type and Triebel--Lizorkin-type spaces
is the following \emph{Peetre maximal function} (see also \cite{lsuyy2}).

\begin{defn}\label{def4.1x}
Let $P_0, P$ be even functions in  $\mathcal S(\rr)$ such  that
$|P_0|\ge c$ on $[0,b^{-\bz_0/2}]$
and $|P|\ge c$ on $[b^{\bz_0/2}, b^{-\bz_0/2}]$, where $b\in (0,1)$ and $c$ is a positive constant.
For $\ell\in\nn$,  let
$$P_\ell(\lz):=P(b^{\ell\bz_0/2}\lz),\qquad \lz\in\rr_+.$$
Let $f\in \cd'(M)$. For any $\ell\in\zz_+$, $a\in(0,\fz)$ and $\gz\in \rr$, the \emph{Peetre maximal function} $[P_\ell(\sqrt{L})]^*_{a,\gz}f$ is defined by
$$[P_\ell(\sqrt{L})]^*_{a,\gz}f(x)
:=\sup_{y\in M} \frac{|B(y,b^{\ell})|^\gamma\, |P_\ell(\sqrt L)f(y)|}
{[1+b^{-\ell}\rho(x,y)]^{a}},\quad x\in M.$$
When $\gz=0$, we simply write $[P_\ell(\sqrt{L})]^*_{a,\gz}$
as $[P_\ell(\sqrt{L})]^*_{a}$.
\end{defn}

The assumptions on $P_0$ and $P$ imply that the kernel of $P_\ell(\sqrt L)$
satisfies \eqref{eq:2.12x} and \eqref{eq:2.13x}, so that $P_\ell(\sqrt L)f(y)$ is well defined for all $f\in\cd'(M)$ and $y\in M$.
Define the \emph{Hardy-Littlewood maximal operator $\cm $} as follows: for all $g\in L_\loc^1(M)$ and $x\in M$,
$$\cm g(x):=\sup_{B\ni x}\frac1{|B|}\int_B |g(y)|\,d\mu(y),$$
where the supremum is taken over all balls $B\subset M$ containing $x$.
If $r\in(0, \infty)$,  then we define
$$\cm_rg(x):= [\cm(|g|^r)(x)]^{1/r}, \qquad x\in M.$$
It is known that $\cm$ is bounded on $L^p(M)$ when $p\in(1,\infty]$;
see \cite{CW1,CW2}.
The following proposition shows that the Peetre maximal function
can be controlled by the Hardy-Littlewood maximal operator $\cm$.
Its proof relies on the Calder\'on
reproducing formula in \eqref{c-crf} and the off-diagonal
estimates in Proposition \ref{prop2.14x}.

\begin{prop}\label{prop4.2x}
Let $\nu,\,r\in(0,\infty)$ and $\gamma\in\rr$.
Fix $x\in Q^k_\az$ with $k\in\zz$ and $\az\in I_k$.
For any $f\in \cd'(M)$ and $\ell\in\zz_+$, the Peetre maximal function
$[P_\ell(\sqrt{L})]^*_{\nu+d/r,\gz}f(x)$ defined in Definition \ref{def4.1x}
satisfies the following estimate:
\begin{eqnarray}\label{eq:4.1x}
&&[P_\ell(\sqrt{L})]^*_{\nu+d/r,\gz}f(x)\notag\\
&&\hs\le C  \sum_{j=\ell}^\fz
\sum_{i=0}^\fz  b ^{(j-\ell)\nu}  b ^{i\nu }\mathcal{M}_r
\left(|B(\cdot, b ^j)|^\gamma |P_j(\sqrt L)f|
\chi_{B(z_\az^k,  b ^{\ell-i}+\diam Q_\az^k)}\right)(x),
\end{eqnarray}
where $C:=C(b,\nu,r,\gamma)$ is a positive constant independent of $\ell,\, x,\,k,\,\az$  and $f$.
\end{prop}

Compared with \cite[Lemma 6.4]{KP},
the advantage of \eqref{eq:4.1x} lies in that it is true for  general distributions. Also, \eqref{eq:4.1x}  is  more delicate due to the existence of the characteristic function in the maximal function of its right-hand side.

To prove Proposition \ref{prop4.2x}, we begin with two technical lemmas.

\begin{lem}\label{lem4.3x}
Let $b\in(0,1)$, $\sigma\in(d,\fz)$ and $g\in L_\loc^1(M)$. Then,
for all $j\in\zz$, and $x\in Q_\az^k$ with $k\in\zz$ and $\az\in I_k$,
\begin{equation}\label{eq:4.2x}
\int_M \frac{ |g(y)|}{|B(y,  b ^j)|[1+ b ^{-j}\rho(x,y)]^\sz}\,d\mu(y)
\le C\sum_{i=0}^\infty  b ^{i(\sz-d)}
\cm(g\chi_{B(z_\az^k,  b ^{j-i}+\diam Q_\az^k)})(x),
\end{equation}
where $C$ is a positive constant independent of
$k,\,j,\, x$ and $g$.
\end{lem}

\begin{proof}
Denote by ${\rm J}$ the left-hand side of \eqref{eq:4.2x}.
Obviously,
\begin{eqnarray*}
{\rm J}
= \sum_{i=0}^\infty
\int_{\rho(x,y)\sim b ^{j-i}} \frac1{|B(y,  b ^j)|}
 \frac{|g(y)|}{[1+ b ^{-j}\rho(x,y)]^{\sz}} \,d\mu(y)=: \sum_{i=0}^\infty{\rm J}_i,
\end{eqnarray*}
where the symbol $\rho(x,y)\sim b ^{j-i}$
means that $ b ^{j-i+1}\le\rho(x,y)<  b ^{j-i}$
for $i\ge1$ and that $\rho(x,y)<  b ^{j}$
for $i=0$.
For any $i\in\zz_+$, by the fact $x\in Q_\az^k$ and \eqref{doubling2}, we see that
\begin{eqnarray*}
{\rm J}_i
\ls  b ^{i(\sz-d)} \cm\lf(
g\chi_{B(z_\az^k,  b ^{j-i}+\diam Q_\az^k)}
\r)(x).
\end{eqnarray*}
Thus, \eqref{eq:4.2x} holds true, which
completes the proof of Lemma \ref{lem4.3x}.
\end{proof}

\begin{lem}\label{lem4.4x}
Let $b\in(0,1)$ and  $\{P_k(\sqrt{L})\}_{k\in\zz_+}$ be as in  Definition \ref{def4.1x}.
Then, for any $r\in(0,\fz)$, $\sz\in(d,\fz)$ and $N\in(d/\bz_0,\fz)$,
there exists a positive constant $C:=C(r, \sz, N)$ such that, for
all $k\in\zz_+$, $f\in \cd'(M)$ and $x\in M$,
\begin{eqnarray}\label{eq:4.3x}
|P_k(\sqrt{L})f(x)|^r
&&\le C \sum_{j=k}^\fz  b ^{(j-k)[N\bz_0(r\wedge1)-d]}
\int_M \frac{1}{|B(z, b ^{k})|}
\frac{|P_j(\sqrt{L})f(z)|^r}{[1+ b ^{-k}\rho(x,z)]^{\sigma (r\wedge1)}}\,d\mu(z).
\end{eqnarray}
\end{lem}

\begin{proof} Fix $\ell\in\zz_+$.
Choose a nonnegative even function $\Phi_0\in C_c^\infty(\rr)$
such that $\supp \Phi_0\subset [-b ^{-\bz_0/2},  b ^{-\bz_0/2}]$
and $\Phi_0(\lz)=1$ when $\lz\in[-1,1]$.
Define $\Phi(\lz):= \Phi_0(\lz)-\Phi_0( b ^{-\bz_0/2}\lz)$ for all $\lz\in\rr_+$.
Notice that  $\Phi$ satisfies  \eqref{eq:2.15xx} with $\dz=b$ there.
For any $j\in\nn$, define $\Phi_j(\cdot)=\Phi(b^{j\bz_0/2}\cdot)$.
Then
$$\Phi_0( b ^{\ell\bz_0/2}\lz)+\sum_{j=\ell+1}^\fz\Phi_j(\lz)\equiv 1,
\qquad \lz\in\rr_+.$$
For $\ell\in\nn$, define $\wz{\Psi}_\ell(\cdot):=\Phi_0( b ^{\ell\bz_0/2}\cdot)$ and ${\Psi}:={\Phi}/{P}$.
The assumption on $P$ (see Definition \ref{def4.1x})
implies that $\Psi$ is well defined and it
 satisfies \eqref{eq:2.15xx}. Moreover,
\begin{equation}\label{eq:4.4x}
\wz{\Psi}_\ell(\lz)+\sum_{j=\ell+1}^\fz \Psi_j(\lz)P_j(\lz)=1,
\qquad \lz\in\rr_+,
\end{equation}
where $\Psi_j(\cdot):=\Psi( b ^{j\bz_0/2}\cdot)$ for all $j\in\nn$.
Based on \eqref{eq:4.4x} and \eqref{c-crf}, we see that
\begin{eqnarray*}
P_\ell(\sqrt{L})f&&=\wz{\Psi}_\ell(\sqrt{L})P_\ell(\sqrt{L})f +\sum_{j=\ell+1}^\fz P_\ell(\sqrt{L})\Psi_j(\sqrt{L})P_j(\sqrt{L})f \quad \textup{in}\; \cd'(M).
\end{eqnarray*}
Since $\Psi_0$ is an even Schwartz function,
we apply Proposition \ref{prop2.10x}(i) to conclude that, for all $x,\,y\in M$,
$$ |\wz{\Psi}_\ell(\sqrt{L})(x,y)|
=|\Phi_0( b ^{\ell\bz_0/2}\sqrt{L})(x,y)|
\ls D_{ b ^{\ell},\sz+d/2}(x,y).$$
Further, by Proposition \ref{prop2.14x} and Remark \ref{rem2.15x}, we see that, for any $j\ge \ell+1$ and $x,\,y\in M$,
$$ |P_\ell(\sqrt{L})\Psi_j(\sqrt{L})(x,y)|
\ls  b ^{(j-\ell)(N\bz_0-d)}  D_{ b ^{\ell}, \sigma+d/2}(x,y).$$
Consequently, for all $\ell\in\zz_+$ and $x\in M$, we have
\begin{eqnarray}\label{eq:4.5x}
|P_\ell(\sqrt{L})f(x)|
\ls  \sum_{j=\ell}^\fz  b ^{(j-\ell)(N\bz_0-d)}
\int_M \frac{1}{|B(y, b ^{\ell})|}
\frac{|P_j(\sqrt{L})f(y)|}{[1+ b ^{-\ell}\rho(x,y)]^\sigma}\,d\mu(y).
\end{eqnarray}
Thus, the estimate \eqref{eq:4.3x} for $r=1$ is proved.

When $r\in(1,\infty)$,
H\"older's inequality and Lemma \ref{lem2.1x}(i) imply that,
if we choose $\sz>d$, then
$$
\int_M \frac{1}{|B(z, b ^{\ell})|}
\frac{|P_j(\sqrt{L})f(z)|}{[1+ b ^{-\ell}\rho(x,z)]^\sigma}\,d\mu(z)
\ls \bigg\{\int_M \frac{1}{|B(z, b ^{\ell})|}
\frac{|P_j(\sqrt{L})f(z)|^r}{[1+ b ^{-\ell}\rho(x,z)]^\sigma}\,d\mu(z)\bigg\}^{1/r}.$$
Inserting this into \eqref{eq:4.5x}, we then apply H\"older's inequality,
the facts $N\bz_0>d$ and $b\in(0,1)$ to conclude that, for all $x\in M$,
\begin{eqnarray*}
|P_\ell(\sqrt{L})f(x)|^r
&&\ls
\sum_{j=\ell}^\fz  b ^{(j-\ell)(N\bz_0-d)}
\int_M \frac{1}{|B(z, b ^{\ell})|}
\frac{|P_j(\sqrt{L})f(z)|^r}{[1+ b ^{-\ell}\rho(x,z)]^\sigma}\,d\mu(z).
\end{eqnarray*}
Hence, \eqref{eq:4.3x} holds true for $r\in(1,\infty)$.

To prove \eqref{eq:4.3x} for $r\in(0,1)$, we define
$$N(x,k):=\sup_{\ell\ge k}\sup_{y\in M} b ^{(\ell-k)N\bz_0}
\frac{|P_\ell(\sqrt{L})f(y)|}{[1+ b ^{-k}\rho(x,y)]^\sigma},\qquad x\in M,\, \, k\in\zz_+.$$
If $\ell\ge k$,  then
$|B(z, b ^{\ell})|\ge  b ^{(\ell-k)d}|B(z, b ^{k})|$
and $[1+ b ^{-k}\rho(x,y)][1+ b ^{-\ell}\rho(y,z)]\ge
1+ b ^{-k}\rho(x,z)$ for all $x,\,y,\,z\in M$.
From this and \eqref{eq:4.5x},  it follows that, for all $\ell\ge k$ and $x$, $y\in M$,
\begin{eqnarray*}
\frac{|P_\ell(\sqrt{L})f(y)|}{[1+ b ^{-k}\rho(x,y)]^\sigma}
\ls\sum_{j=\ell}^\fz  b ^{(j-\ell)N\bz_0}  b ^{(k-j)d}
\int_M \frac{1}{|B(z, b ^{k})|}
\frac{|P_j(\sqrt{L})f(z)|}{[1+ b ^{-k}\rho(x,z)]^\sigma}\,d\mu(z),
\end{eqnarray*}
which, together with the definition of $N(x,k)$, further implies that
\begin{eqnarray*}
N(x,k)
\ls \lf[N(x,k)\r]^{1-r}\sum_{j=k}^\fz  b ^{(j-k)[N\bz_0r-d]}
\int_M \frac{1}{|B(z, b ^{k})|}
\frac{|P_j(\sqrt{L})f(z)|^r}{[1+ b ^{-k}\rho(x,z)]^{\sigma r}}\,d\mu(z) .
\end{eqnarray*}
If we have proved that $N(x,k)<\fz$, then
\begin{eqnarray}\label{eq:4.6x}
|P_k(\sqrt{L})f(x)|^r
&&\le [N(x,k)]^r\ls\sum_{j=k}^\fz  b ^{(j-k)[N\bz_0r-d]}
\int_M \frac{1}{|B(z, b ^{k})|}
\frac{|P_j(\sqrt{L})f(z)|^r}{[1+ b ^{-k}\rho(x,z)]^{\sigma r}}\,d\mu(z),
\end{eqnarray}
as desired.

It remains to prove $N(x,k)<\infty$ for all $\sz,\, N>0$. Consider first the case $\mu(M)<\infty$.
By Proposition \ref{prop2.13xx}(ii), there exists some $m_0\in\nn$, depending on $f$, such that, for all $y\in M$,
\begin{eqnarray*}
|P_\ell(\sqrt{L})f(y)|
\ls \max_{0\le m\le m_0}  \|L^m P_\ell (\cdot, y)\|_{L^2(M)}.
\end{eqnarray*}
For any $\wz\sz>d/2$, it follows, from Proposition \ref{prop2.10x}(i) and Lemma \ref{lem2.1x}(i),
that
\begin{eqnarray*}
\|L^m P_\ell(\sqrt L) (\cdot, y)\|_{L^2(M)}
\ls \frac{ b ^{-\bz_0 m\ell}}{|B(y,  b ^{\ell})|^{1/2}},
\end{eqnarray*}
which further implies that, when $\sz> d/2$ and $N> m_0+ d/(2\bz_0)$,
\begin{eqnarray*}
N(x,k)
\ls \sup_{\ell\ge k}\sup_{y\in M}
\max_{0\le m\le m_0}
\frac{ b ^{(\ell-k)N\bz_0-\bz_0m\ell}}{[1+ b ^{-k}\rho(x,y)]^\sigma
|B(y,  b ^{\ell})|^{1/2}}
\ls \frac{ b ^{-\bz_0m_0k}}{|B(x,  b ^k)|^{1/2}}<\infty.
\end{eqnarray*}

Now we consider the case
$\mu(M)=\infty$. By Proposition \ref{prop2.13xx}(iii), there exist $\ell_0,\, m_0\in\nn$, depending on $f$, such that, for all $y\in M$,
\begin{eqnarray*}
|P_\ell(\sqrt{L})f(y)|
&&\ls \,\max_{0\le m\le m_0,\, 0\le \nu\le \ell_0}\, \sup_{z\in M}\,
[1+\rho(z, x_0)]^\nu |L^mP_\ell(\sqrt L) (z, y)|.
\end{eqnarray*}
For any  $0\le m\le m_0$,
it follows, from Proposition \ref{prop2.10x}(i),
 that
$$|L^m P_\ell(\sqrt L) (z, y)|
\ls  \frac{ b ^{-\bz_0m\ell}}
{|B(y,  b ^{\ell})|[1+ b ^{-\ell}\rho(z,y)]^{\ell_0}}
\ls \frac{ b ^{-\ell(\bz_0m_0+ d)} }
{|B(y,1)|[1+\rho(z,y)]^{\ell_0}}$$
for all $\ell\in\zz_+$ and $z,\,y\in M$.
Thus, for all $y\in M$,
\begin{eqnarray*}
|P_\ell(\sqrt{L})f(y)|
\ls \frac{ b ^{-\ell(\bz_0m_0+ d)} [1+\rho(y, x_0)]^{\ell_0}}
{|B(y,1)|}.
\end{eqnarray*}
Therefore, if $\sz\ge\ell_0+d$ and $N\ge m_0+d/\bz_0$, we have
\begin{eqnarray*}
N(x,k)
\ls \sup_{\ell\ge k}\sup_{y\in M}
\frac{ b ^{(\ell-k)N\bz_0-\ell(\bz_0m_0+ d)}
[1+\rho(y, x_0)]^{\ell_0}  }{[1+\rho(x,y)]^\sigma|B(y,1)|
}\ls \frac{ b ^{-k(\bz_0m_0+ d)} [1+\rho(x, x_0)]^{\ell_0}}
{|B(x,1)|},
\end{eqnarray*}
which is finite.
Since $\ell_0$ and $m_0$ depend on $f$,
we have proved the boundedness
of $N(x,k)$ in both cases $\mu(M)<\fz$
and $\mu(M)=\fz$ whenever $\sigma>\sigma(f)$ and $N>N(f)$
for some positive constants $\sigma(f)$ and $N(f)$ depending on $f$.
Further, invoking \eqref{eq:4.5x} and \eqref{eq:4.6x}, we obtain,
when $\sigma>\sigma(f)$ and $N>N(f)$,
\begin{equation*}
|P_k(\sqrt{L})f(x)|^r
\le c(f)\sum_{j=k}^\fz  b ^{(j-k)[N\bz_0r-d]}
\int_M \frac{1}{|B(z, b ^{k})|}
\frac{|P_j(\sqrt{L})f(z)|^r}{[1+ b ^{-k}\rho(x,z)]^{\sigma r}}\,d\mu(z)
\end{equation*}
for a positive constant $c(f)$ depending on $f$ and for all $k\in\zz$ and $x\in M$.
Observe that this inequality holds true for all $\sigma,\ N>0$ since its  right-hand side
decreases as $\sigma$ and $N$ increase.
We now use this fact to prove the boundedness of $N(x,k)$ for all $\sz,\,N>0$.
To be precise, by \eqref{doubling2}, we have
\begin{eqnarray*}
[N(x,k)]^r
\le K\,c(f)\sum_{j=k}^\fz  b ^{(j-k)[N\bz_0r-d]}
\int_M \frac{1}{|B(z, b^{k})|}
\frac{|P_{j}(\sqrt{L})f(z)|^r}{[1+ b ^{-k}\rho(x,z)]^{\sigma r}}\,d\mu(z),
\end{eqnarray*}
which is finite; otherwise \eqref{eq:4.3x} holds trivially since its right-hand side is infinity. Altogether, we proved that $N(x,k)$ is finite for all $\sz,\,N>0$. This finishes the proof of Lemma \ref{lem4.4x}.
\end{proof}

\begin{proof}[Proof of Proposition \ref{prop4.2x}]
Since $\cm g\le \cm_r g$ for all $r\in[1,\infty)$ and $g\in L_\loc^1(M)$,
it suffices to prove \eqref{eq:4.1x} for the case $r\in(0,1]$.
Fix $r\in(0,1]$.
For sufficiently large $\sz$ and $N$, by Lemma \ref{lem4.4x},
we know that, for all $\ell\in\zz_+$ and $x,\,y\in M$,
 \begin{eqnarray}\label{eq:4.7x}
\frac{|B(y, b ^\ell)|^{\gamma r} |P_\ell(\sqrt L)f(y)|^r}
{[1+ b ^{-\ell}\rho(x,y)]^{\nu r+d}}
&&\ls
\sum_{j=\ell}^\fz  b ^{(j-\ell)[N\bz_0 r-d]}
\int_M
\frac{|B(y, b ^\ell)|^{\gamma r}}{|B(z, b ^j)|^{\gamma r}}\bigg.\notag\\
&&\quad\times\bigg.
\frac{|B(z, b ^j)|^{\gamma r}|P_j(\sqrt{L})f(z)|^r}
{|B(z, b ^{\ell})|
[1+ b ^{-\ell}\rho(x,y)]^{\nu r+d}
[1+ b ^{-\ell}\rho(y,z)]^{\sigma r}}\,d\mu(z).
\end{eqnarray}
From \eqref{doubling2} and $\ell\le j$, it follows that
$$\frac{|B(y, b ^\ell)|^{r\gamma}}{|B(z, b ^j)|^{r\gamma}}
\ls  b ^{(\ell-j)dr|\gz|}[1+ b ^{-\ell}\rho(y,z)]^{d r|\gz|}.$$
Choose $\sz$ and $N$ such that $\sz r> \nu r+d+dr|\gz|$
and $N\bz_0 r-d-dr|\gz|>\nu>0$.
Then, by  \eqref{eq:4.7x}, $r\in(0,1]$ and H\"older's inequality, we have
\begin{eqnarray*}
\frac{|B(y, b ^\ell)|^\gamma |P_\ell(\sqrt L)f(y)|}
{[1+ b ^{-\ell}\rho(x,y)]^{\nu+d/r}}
\ls
\sum_{j=\ell}^\fz  b ^{(j-\ell)[N\bz_0r-d-dr|\gz|]}
\bigg\{\int_M
\frac{|B(z, b ^j)|^{r\gamma}|P_j(\sqrt{L})f(z)|^r}
{|B(z, b ^{\ell})|
[1+ b ^{-\ell}\rho(x,z)]^{\nu r+d}}\,d\mu(z)\bigg\}^{1/r}.
\end{eqnarray*}
Further, Lemma \ref{lem4.3x} and $(N\bz_0+d/2)r-d-dr|\gz|>\nu$ imply that
\begin{eqnarray*}
\frac{|B(y, b ^\ell)|^\gamma |P_\ell(\sqrt L)f(y)|}
{[1+ b ^{-\ell}\rho(x,y)]^{\nu+d/r}}\ls
\sum_{j=\ell}^\fz \sum_{i=0}^\infty
 b ^{(j-\ell)\nu} b ^{i\nu }
\cm_r(|B(\cdot, b ^j)|^{\gamma}|P_j(\sqrt{L})f|
\chi_{B(z_\az^k,  b ^{\ell-i}+\diam Q_\az^k)})(x).
\end{eqnarray*}
Taking the supremum over all $y\in M$ and using the definition of
$[P_\ell(\sqrt{L})]^*_{\nu+d/r,\gz}f(x)$, we obtain
\eqref{eq:4.1x}. This proves
Proposition \ref{prop4.2x}.
\end{proof}

In the proof of Proposition \ref{prop4.2x}, instead of using Lemma \ref{lem4.3x},
we use the fact that the left-hand
side of \eqref{eq:4.2x} can be controlled by $\cm g(x)$, then we easily
obtain the following estimate, the details being omitted.

\begin{cor}\label{cor4.5x}
Let $\nu,\,r\in(0,\infty)$ and $\gamma\in\rr$.
For all $f\in \cd'(M)$, $\ell\in\zz_+$  and $x\in M$, the Peetre maximal function
$[P_\ell(\sqrt{L})]^*_{\nu+d/r,\gz}f(x)$ defined in Definition \ref{def4.1x}
satisfies the following estimate:%
$$
[P_\ell(\sqrt{L})]^*_{\nu+d/r,\gz}f(x)\le C  \sum_{j=\ell}^\fz
 b ^{(j-\ell)\nu} \mathcal{M}_r
\left(|B(\cdot, b ^j)|^\gamma |P_j(\sqrt L)f|
\right)(x), \qquad x\in M,
$$
where $C:=C(b,\nu,r,\gamma)$ is a positive constant independent of $\ell,\,f$ and $x$.
\end{cor}

\begin{rem}\label{rem4.6x}
When $\bz_0=2$,
the inequality in Corollary \ref{cor4.5x}
was obtained
in \cite[Lemma~6.4]{KP},
but only for functions in the following \emph{spectral space}
$$\Sigma_\lz^p(M):=
\{f\in L^p(M):\,\, \theta(\sqrt L)f=f\, \,\textup{for all}\, \,
\theta\in C_c^\infty(\rr_+),\, \theta\equiv1 \,\textup{on}\, [0,\lz]\,\}
$$
with $p\in[1,\infty]$ and $\lz\in(0,\infty)$, where
$C_c^\infty(\rr_+)$
denotes the \emph{space of all functions in $C^\infty(\rr_+)$ with compact support}.
Here, when $p=\infty$, $L^\infty(M)$ is replaced by the space of
uniformly continuous and bounded functions on $M$.
Clearly, Corollary \ref{cor4.5x} improves \cite[Lemma~6.4]{KP},
since it holds true for all distributions.
\end{rem}

%%%%%%%%%%%%%%%%%%%%%%%%%%%%%%%%%%%%%%%%%%%%%%%%%%%%%%%%%%%%%%%%%%%%%

\subsection{Embedding properties}\label{sec-4.2}

\hskip\parindent
In what follows, if the function space
$\mathcal X$ is continuously embedded into $\mathcal Y$,
then we write $\mathcal X\hookrightarrow \mathcal Y$.
For all $s\in\rr$, notice that $B_{\infty,\infty}^s(M)= B_{\infty,\infty}^{s,0}(M)$
and $\wz B_{\infty,\infty}^s(M)= \wz B_{\infty,\infty}^{s,0}(M)$. Moreover, for any $f\in\cd'(M)$,
$$\|f\|_{B_{\infty,\infty}^s(M)}
\sim\|f\|_{B_{\infty,\infty}^{s,0}(M)}
\sim\sup_{j\in\zz_+}\sup_{x\in M} \dz^{-js}|\Phi_j(\sqrt L)f(x)|$$
and
$$\|f\|_{\wz B_{\infty,\infty}^s(M)}
\sim \|f\|_{\wz B_{\infty,\infty}^{s,0}(M)}
\sim\sup_{j\in\zz_+}\sup_{x\in M} |B(x, \dz^j)|^{-s/d}|\Phi_j(\sqrt L)f(x)|$$
with equivalent positive constants independent of $f$,
where $\Phi_j$ for $j\in\zz_+$ is as in Definition \ref{def3.2x}.
We start with the following embedding property,
which was obtained in  \cite[Proposition 4.2]{s013}
when $(M,\rho,\mu)$ is the Euclidean space.

\begin{prop}\label{prop4.7x}
Let $s\in\rr$, $\tau\in[0,\infty)$ and $q\in(0,\infty]$.
\begin{enumerate}
\item[\rm (i)]  If $p\in(0,\infty]$, then $B_{p,q}^{s,\tau}(M)\hookrightarrow B_{\fz,\fz}^{s+d\tau-d/p}(M)$
and $\wz B_{p,q}^{s,\tau}(M)\hookrightarrow \wz B_{\fz,\fz}^{s+d\tau-d/p}(M)$.

\item[\rm (ii)] If $p\in(0,\infty)$, then $F_{p,q}^{s,\tau}(M)\hookrightarrow B_{\fz,\fz}^{s+d\tau-d/p}(M)$
and $\wz F_{p,q}^{s,\tau}(M)\hookrightarrow \wz B_{\fz,\fz}^{s+d\tau-d/p}(M)$.
\end{enumerate}
\end{prop}

\begin{proof}
Due to the similarity, we only consider  embeddings of
 $\wz B_{p,q}^{s,\tau}(M)$ and $\wz F_{p,q}^{s,\tau}(M)$.
By the Minkowski inequality,
we know that
\begin{equation}\label{eq:4.8x}
\wz B_{p,\min\{p,q\}}^{s,\tau}(M)\hookrightarrow\wz F_{p,q}^{s,\tau}(M)\hookrightarrow
\wz B_{p,\max\{p,q\}}^{s,\tau}(M),
\end{equation}
which, together with the monotonicity of $\wz B_{p,r}^{s,\tau}(M)$ on $r\in(0,
\fz]$, further implies that
$$\wz B_{p,q}^{s,\tau}(M),\  \wz F_{p,q}^{s,\tau}(M)\hookrightarrow \wz B_{p,\fz}^{s,\tau}(M).$$
Thus,  it suffices to prove that
$ \wz B_{p,\fz}^{s,\tau}(M)\hookrightarrow \wz B_{\fz,\fz}^{s+d\tau-d/p}(M).$
To this end, we let $(\Phi_0, \Phi)$
satisfy \eqref{eq:2.14xx} and \eqref{eq:2.15xx},
and $f\in \wz B_{p,\fz}^{s,\tau}(M)$.
Then
$$\|f\|_{\wz B_{p,\fz}^{s,\tau}(M)}
= \sup_{k\in\zz,\,\az\in I_k}\, \sup_{j\ge (k\vee0)}
\frac1{|Q_\az^k|^\tau} \bigg[ \int_{Q_\az^k}
|B(x,\delta^j)|^{-sp/d}|\Phi_j(\sqrt{L})f(x)|^p
\,d\mu(x)\bigg]^{1/p}<\infty,$$
where $\Phi_j$ for $j\in\nn$ is defined as in \eqref{eq:2.16xx}.
By Lemma \ref{lem4.4x}, for
all $k\in\zz_+$ and $x\in M$, we have
\begin{eqnarray*}
&&|B(x,\delta^k)|^{-(s+d\tau-d/p)/d}|\Phi_k(\sqrt{L})f(x)|\\
&&\quad\ls \Bigg\{\sum_{j=k}^\fz \delta^{(j-k)[N\bz_0(p\wedge1)-d]}
\int_M \frac{|B(x,\delta^k)|^{-sp/d-p\tau+1}}{|B(z,\delta^{k})|}
\frac{|\Phi_j(\sqrt{L})f(z)|^p}
{[1+\delta^{-k}\rho(x,z)]^{\sigma (p\wedge1)}}\,d\mu(z)\Bigg\}^{1/p},
\end{eqnarray*}
where $\sz$ and $N$ are large numbers satisfying the hypothesis of Lemma \ref{lem4.4x},
$\sigma (p\wedge1)-|s|p-p\tau d>2d+1$ and
$N\bz_0(p\wedge1)-d>|s|p+1$.
For all $k\in\zz_+$, $j\ge k$ and $x,\,z\in M$, by \eqref{doubling2}, we find that
\begin{eqnarray*}
\frac{|B(x,\delta^k)|^{-sp/d-p\tau+1}}{|B(z,\delta^{k})|[1+\delta^{-k}\rho(x,z)]^{\sigma (p\wedge1)}}
\ls\dz^{(k-j)|s|p}\frac{|B(z,\delta^j)|^{-sp/d}}{|B(z,\delta^k)|^{\tau p}
[1+\delta^{-k}\rho(x,z)]^{d+1} },
\end{eqnarray*}
and hence
\begin{eqnarray*}
|B(x,\delta^k)|^{-(s+d\tau-d/p)/d}|\Phi_k(\sqrt{L})f(x)|
\ls\Bigg\{\sum_{j=k}^\fz \delta^{j-k}
\sum_{\az\in I_k}\inf_{u\in Q_\az^k}\frac{1}
{[1+\delta^{-k}\rho(x,u)]^{d+1}}
\Bigg\}^{1/p} \|f\|_{\wz B_{p,\fz}^{s,\tau}(M)},
\end{eqnarray*}
which implies that $ \|f\|_{\wz B_{\fz,\fz}^{s+d\tau-d/p}(M)} \ls \|f\|_{\wz B_{p,\fz}^{s,\tau}(M)}$
if we observe that
 $\sum_{j=k}^\fz \delta^{j-k}\ls1$ and
$$\sum_{\az\in I_k}\inf_{u\in Q_\az^k}\frac{1}
{[1+\delta^{-k}\rho(x,u)]^{d+1}}
\ls \sum_{\az\in I_k}\int_{Q_\az^k} \frac1
{|B(z, \dz^k)|[1+\delta^{-k}\rho(x,z)]^{d+1}}\,d\mu(z)\ls1.$$
This proves
$ \wz B_{p,\fz}^{s,\tau}(M)\hookrightarrow \wz B_{\fz,\fz}^{s+d\tau-d/p}(M)$
and hence finishes the proof of Proposition \ref{prop4.7x}.
\end{proof}

\begin{prop}\label{prop4.8x}
Let $s\in\rr$, $\tau\in[0,\infty)$ and $q\in(0,\infty]$.
\begin{enumerate}
\item[\rm(i)]  If $p\in(0,\infty]$, then $\cd(M) \hookrightarrow B_{p,q}^{s,\tau}(M)\hookrightarrow \cd'(M)$
and $\cd(M) \hookrightarrow \wz B_{p,q}^{s,\tau}(M)\hookrightarrow \cd'(M)$.

\item[\rm(ii)] If $p\in(0,\infty)$, then $\cd(M) \hookrightarrow F_{p,q}^{s,\tau}(M)\hookrightarrow \cd'(M)$
and $\cd (M)\hookrightarrow \wz F_{p,q}^{s,\tau}(M)\hookrightarrow \cd'(M)$.
\end{enumerate}
\end{prop}

\begin{proof} We only consider the case $\mu(M)=\infty$, since the proof for the case
$\mu(M)<\infty$ follows in a similar way.
Let $(\Phi_0,\Phi)$ satisfy \eqref{eq:2.14xx} and \eqref{eq:2.15xx}.
Fix $f\in\cd (M)$.
Fix $m,\,\ell\in \nn$ (to be determined later).
For all $j\in\nn$ and $x\in M$, by Proposition \ref{prop2.10x}(i)
and Lemma \ref{lem2.1x}(ii),
we obtain
\begin{eqnarray*}%\label{eq:4.9x}
|\Phi_j(\sqrt L)f(x)|
\ls \dz^{j(\bz_0m-d)} [1+\rho(x,x_0)]^{-\ell} \cp_{m,\ell}(f),
\end{eqnarray*}
where $x_0$ is some fixed point of $M$.
For $j=0$, the above estimate \eqref{eq:4.9x} remains true if we take $m=0$. Hence,
for all $j\in\zz_+$ and $x\in M$, we have
\begin{eqnarray}\label{eq:4.9x}
|\Phi_j(\sqrt L)f(x)|
\ls [\cp_{m,\ell}(f)+\cp_{0,\ell}(f)]\, \dz^{j(\bz_0m-d)} [1+\rho(x,x_0)]^{-\ell}.
\end{eqnarray}
Then, for any cube $Q_\az^k$ with $k\in\zz$ and $\az\in I_k$,
\begin{eqnarray}\label{eq:4.10x}
{\rm J}:=&&\frac1{|Q_\az^k|^\tau}
\Bigg\{\sum_{j=k\vee0}^\infty
\bigg[\int_{Q_\az^k} |B(x,\dz^j)|^{-sp/d}
|\Phi_j(\sqrt L)f(x)|^p\,d\mu(x)\bigg]^{q/p}\Bigg\}^{1/q}\notag\\
\ls&&[\cp_{m,\ell}(f)+\cp_{0,\ell}(f)]
\Bigg[\sum_{j=k\vee0}^\infty \dz^{j(\bz_0 m-d)q}\notag\\
&&\quad\times
\bigg\{\frac{1}{|Q_\az^k|^{\tau p}}\int_{Q_\az^k} |B(x,\dz^j)|^{-sp/d}
 [1+\rho(x,x_0)]^{-\ell p} \,d\mu(x)\bigg\}^{q/p}\Bigg]^{1/q}.
\end{eqnarray}
For all $x\in Q_\az^k$,
we have
$|Q_\az^k| \sim |B(x, \dz^k)| \gs |B(x, \dz^j)|$
and
$ \dz^{jd} [1+\rho(x,x_0)]^{-d}\ls \frac{|B(x,\dz^j)|}{|B(x_0, 1)|} \ls [1+\rho(x,x_0)]^d,$
which implies that
$$|Q_\az^k|^{-\tau p}|B(x,\dz^j)|^{-sp/d}
\ls |B(x,\dz^j)|^{-\tau p-sp/d}
\ls \dz^{-j d\tau p-j|s|p} [1+\rho(x,x_0)]^{|s|p+d\tau p}.$$
Inserting this estimate into \eqref{eq:4.10x}, we conclude that
\begin{eqnarray*}
{\rm J}&& \ls [\cp_{m,\ell}(f)+\cp_{0,\ell}(f)]
\Bigg[\sum_{j=k\vee0}^\infty \dz^{jq[(\bz_0 m-d)-\tau d -|s|]}
\bigg\{\int_{M}
 [1+\rho(x,x_0)]^{-\ell p +|s|p+d\tau p} \,d\mu(x)\bigg\}^{q/p}\Bigg]^{1/q},
\end{eqnarray*}
which is controlled by a positive constant multiple of $\cp_{m,\ell}(f)$, provided that
$(\bz_0 m-d)>\tau d +|s|$ and $\ell p-|s|p-d\tau p>d$.
This proves  $\|f\|_{\wz B_{p,q}^{s,\tau}(M)} \ls [\cp_{m,\ell}(f)+\cp_{0,\ell}(f)]$
and hence $\cd(M)\hookrightarrow \wz B_{p,q}^{s,\tau}(M)$.
Similarly, we obtain $\cd(M)\hookrightarrow  B_{p,q}^{s,\tau}(M)$.
From these and \eqref{eq:4.8x}, it follows that
$\cd(M)\hookrightarrow  F_{p,q}^{s,\tau}(M)$ and $\cd(M)\hookrightarrow \wz F_{p,q}^{s,\tau}(M)$.

To show the remaining parts of Proposition \ref{prop4.8x}, by Proposition \ref{prop4.7x},
we only need to prove that
$B^s_{\infty,\infty}(M),\,\wz B^s_{\infty,\infty}(M)\hookrightarrow\cd'(M)$
for all $s\in\rr$.
For all $f\in\cd'(M)$ and $\phi\in\cd(M)$, by  \eqref{c-crf},
we have
\begin{eqnarray*}
\laz f,\phi\raz
=\sum_{j=0}^\infty \laz\Phi_j(\sqrt L)f,\,  \wz\Phi_j(\sqrt L)\phi\raz,
\end{eqnarray*}
where $(\wz\Phi_0,\wz\Phi)$ satisfy \eqref{eq:2.14xx} and \eqref{eq:2.15xx},
and $\wz\Phi_j$ with $j\in\nn$ is defined as in \eqref{eq:2.16xx}.
Given any $m,\,\ell\in \nn$
such that $m\bz_0>d+|s|$ and $\ell>|s|+d$,
applying \eqref{eq:4.9x} to $\wz\Phi_j(\sqrt L)\phi$ and Lemma \ref{lem2.1x}(i), we obtain
\begin{eqnarray*}
|\laz f,\phi\raz|
&&\ls [\cp_{m,\ell}(\phi)+\cp_{0,\ell}(\phi)]\|f\|_{\wz B_{\infty,\infty}^s(M)} \sum_{j=0}^\infty
\dz^{j(m\bz_0-d-|s|)}
\int_M \frac{1}{|B(x_0,1)|[1+\rho(x,x_0)]^{\ell-|s|}}\,d\mu(x)\\
&&\ls [\cp_{m,\ell}(\phi)+\cp_{0,\ell}(\phi)]\|f\|_{\wz B_{\infty,\infty}^s(M)},
\end{eqnarray*}
which gives $\wz B^s_{\fz,\fz}(M)\hookrightarrow \cd'(M)$.
Similarly,  $B^s_{\fz,\fz}(M)\hookrightarrow \cd'(M)$.
This finishes the proof of Proposition \ref{prop4.8x}.
\end{proof}

%%%%%%%%%%%%%%%%%%%%%%%%%%%%%%%%%%%%%%%%%%%%%%%%%%%%%%

\subsection{Classifications of Besov-type and Triebel--Lizorkin-type spaces}\label{sec-4.3}

\hskip\parindent
We  first show that, when $\tau>1/p$,
both Besov-type and Triebel--Lizorkin-type spaces coincide with Besov spaces.
Different from the corresponding result on $\rn$ in \cite{yy4}, which was obtained via
the coincidences between the related sequence spaces and the frame characterizations,
here we give a more direct proof.

\begin{prop}\label{prop4.9x}
Let $\tau\in(1/p, \fz)$, $s\in\rr$ and $p,\,q\in(0,\infty]$ ($p<\infty$ for $F_{p,q}^{s,\tau}(M)$ and $\wz F_{p,q}^{s,\tau}(M)$).
Then
$B_{p,q}^{s,\tau}(M)=F_{p,q}^{s,\tau}(M)=B_{\fz,\fz}^{s+d\tau-d/p}(M)$
and $\wz B_{p,q}^{s,\tau}(M)=\wz F_{p,q}^{s,\tau}(M)= \wz B_{\fz,\fz}^{s+d\tau-d/p}(M)$
with equivalent (quasi-)norms.
\end{prop}

\begin{proof} By similarity, we only  prove
$\wz B_{p,q}^{s,\tau}(M)=\wz F_{p,q}^{s,\tau}(M)= \wz B_{\fz,\fz}^{s+d\tau-d/p}(M)$.
Due to \eqref{eq:4.8x} and Proposition \ref{prop4.7x}, it suffices to prove
$\wz B_{\fz,\fz}^{s+d\tau-d/p}(M)\hookrightarrow \wz B_{p,q}^{s,\tau}(M)$.
Let $f\in \wz B_{\fz,\fz}^{s+d\tau-d/p}(M)$
with norm $1$. Then
$
|B(x,\dz^j)|^{-s/d-\tau+1/p} |\Phi_j(\sqrt L)f(x)| \le 1
$
for all $j\in\zz_+$ and $x\in M$.
By  \eqref{rdoubling2}, we have
\begin{eqnarray*}
&&\sup_{\gfz{k\in\zz}{\az\in I_k}}\frac1{|Q_\az^k|^\tau}
\bigg\{
\sum_{j=k\vee0}^\infty
\bigg[ \int_{Q_\az^k}
|B(x,\dz^j)|^{-sp/d} |\Phi_j(\sqrt L)f(x)|^p\,d\mu(x) \bigg]^{q/p}
\bigg\}^{1/q}\\
&&\quad\ls \sup_{\gfz{k\in\zz}{\az\in I_k}}
\frac1{|Q_\az^k|^\tau}
\bigg\{
\sum_{j=k\vee0}^\infty
\bigg[ \dz^{(j-k)\kz (\tau p-1)}\int_{Q_\az^k}
|B(x,\dz^k)|^{\tau p-1} \,d\mu(x) \bigg]^{q/p}
\bigg\}^{1/q}\ls1,
\end{eqnarray*}
where the last step follows from the fact $|B(x,\dz^k)|\sim|Q_\az^k|$
for all $x\in Q_\az^k$. This proves
 $$ \wz B_{\fz,\fz}^{s+d\tau-d/p}(M)\hookrightarrow \wz B_{p,q}^{s,\tau}(M)$$
 and hence finishes the proof of Proposition \ref{prop4.9x}.
\end{proof}

The proofs of the following two propositions are similar to those
of \cite[Section 3.2]{s012}.

\begin{prop}\label{prop4.10x}
Let $\tau\in[1/p,\infty)$ and all other notation be as in Definition \ref{def3.2x}.
Then the supremum $\sup_{{k\in\zz},\,{\az\in I_k}}$
in Definition \ref{def3.2x} can be equivalently replaced by
$\sup_{{k\in\zz_+},\,{\az\in I_k}}$.
\end{prop}

\begin{proof}
For any $k\in\zz$, let
$g_k:= [ \sum_{j=k\vee 0}^{\infty} |B(\cdot, \delta^{j})|^{-sq/d}
|\Phi_j(\sqrt L) f|^q]^{p/q}$.
When $k<0$, by (ii) and (iv) of Lemma \ref{lem3.1x}, we have
\begin{eqnarray*}
\frac1{|Q^k_\az|^{\tau p}}
\dint_{Q_\az^k} |g_k(x)|\,d\mu(x)
&& \le \lf[\sum_{\{\bz\in I_0:\, Q_\bz^0\cap Q_\az^k\neq\emptyset\}}
\frac{|Q_\bz^0|}{|Q_\az^k|}\r]^{\tau p}
\sup_{k\in\zz_+,\,\az\in I_k}\frac1{|Q_\az^k|^{\tau p}}
\dint_{Q_\az^k} |g_k(x)|\,d\mu(x)\\
&&\ls \sup_{k\in\zz_+,\,\az\in I_k}\frac1{|Q_\az^k|^{\tau p}}
\dint_{Q_\az^k} |g_k(x)|\,d\mu(x),
\end{eqnarray*}
where, in the second step, we used the fact $\tau p\ge1$
and the well-known inequality: for all $\{a_j\}_{j\in\nn}\subset \cc$
and $r\in(0,1]$,
$(\sum_{j\in\nn}|a_j|)^r\le \sum_{j\in\nn}|a_j|^r.$
This implies that
\begin{eqnarray*}
\|f\|_{\wz F^{s,\tau}_{p, q}(M)}
\ls\sup_{\gfz{k\in\zz_+}{\az\in I_k}} \frac1{|Q^k_\az|^{\tau}}
\bigg\{ \dint_{Q^k_\az}\bigg[ \sum_{j=k}^{\infty} |B(x, \delta^{j})|^{-sq/d}
|\Phi_j(\sqrt L) f(x)|^q\bigg]^{p/q}\,d\mu(x)\bigg\}^{1/p}.
\end{eqnarray*}
The converse of this inequality is obvious.
The  equivalence for the spaces $F^{s,\tau}_{p, q}(M)$,
$\wz B^{s,\tau}_{p, q}(M)$ and $\wz B^{s,\tau}_{p, q}(M)$ follow from
the same manner. This finishes the proof of Proposition \ref{prop4.10x}.
\end{proof}

\begin{prop}\label{prop4.11x}
Let $\tau\in[0,1/p)$ and all other notation be as in Definition \ref{def3.2x}.
Then the summation $\sum_{j=k\vee 0}^\infty$
in Definition \ref{def3.2x} can be equivalently replaced by
$\sum_{j=0}^\infty$.%
\end{prop}

\begin{proof}
By similarity, we only consider $\wz F^{s,\tau}_{p, q}(M)$.
It is enough to prove that, when $k\in\nn$,
\begin{eqnarray*}
{\rm J}:= \frac1{|Q^k_\az|^\tau}
\bigg\{ \dint_{Q^k_\az}\bigg[ \sum_{j=0}^{k-1} |B(x, \delta^{j})|^{-sq/d}
|\Phi_j(\sqrt L) f(x)|^q\bigg]^{p/q}\,d\mu(x)\bigg\}^{1/p}
\ls \|f\|_{\wz F^{s,\tau}_{p, q}(M)}.
\end{eqnarray*}
By Proposition \ref{prop4.7x}, we know that
$\wz F_{p,q}^{s,\tau}(M)\hookrightarrow \wz B_{\fz,\fz}^{s+d\tau-d/p}(M)$.
Thus, to show the above desired estimate, it suffices to prove that
${\rm J}\ls\|f\|_{\wz B_{\fz,\fz}^{s+d\tau-d/p}(M)}.$
Indeed,
using  \eqref{rdoubling2}, we see that, for all $x\in Q_\az^k$
and $0\le j\le k-1$,
$|B(x, \delta^{j})| \gs \dz^{(j-k)\kz}|B(x, \delta^{k})|
\sim \dz^{(j-k)\kz}|Q_\az^k|,$
which, combined with the fact $\tau<1/p$, further implies that
\begin{eqnarray*}
{\rm J}
\le \|f\|_{\wz B_{\fz,\fz}^{s+d\tau-d/p}(M)}
\frac1{|Q^k_\az|^\tau}
\bigg\{ \dint_{Q^k_\az}\bigg[ \sum_{j=0}^{k-1} |B(x, \delta^{j})|^{q(\tau -1/p)}
\bigg]^{p/q}\,d\mu(x)\bigg\}^{1/p}\ls \|f\|_{\wz B_{\fz,\fz}^{s+d\tau-d/p}(M)}.
\end{eqnarray*}
This finishes the proof of Proposition \ref{prop4.11x}.
\end{proof}

%%%%%%%%%%%%%%%%%%%%%%%%%%%%%%%%%%%%%%%%%%%%%%%%%%%%%%%%%%%%%%%%%%%%%%%%%

\section{Equivalent characterizations}\label{sec-6}

\hskip\parindent
In this section, we characterize the Besov-type and the Triebel--Lizorkin-type spaces
via the Peetre maximal functions and the heat semigroups; see Theorems \ref{thm6.2x},
 \ref{thm6.7x} and \ref{thm6.8x}  below, respectively.

%%%%%%%%%%%%%%%%%%%%%%%%%%%%%%%%%%%%%%%%%%%%%%%

\subsection{Peetre maximal function characterizations}\label{sec-6.1}
\hskip\parindent  For notational convenience, we introduce the following
(quasi-)norms.

\begin{defn}\label{def6.1x}
Let $p,\,q\in(0,\fz]$ and $\tau\in[0,\fz)$. The {\it space}
$\ell^q(L^p_\tau)$ is defined to be the
set of all sequences $\{g_j\}_{j\in\zz_+}$ of measurable
functions on $M$ such that
$$\|\{g_j\}_{j\in\zz_+}\|_{\ell^q(L^p_\tau)}
:= \sup_{\gfz{k\in\zz}{\az\in I_k}} \frac1{|Q^k_\az|^\tau}
\bigg\{\sum_{j=k\vee0}^\fz \bigg[\int_{Q^k_\az}
|g_j(x)|^p\,d\mu(x)\bigg]^{q/p}\bigg\}^{1/q}<\fz$$
with a suitable modification made when $p=\infty$ or $q=\infty$.
The {\it space} $L^p_\tau(\ell^q)$ (only for $p\in(0,\infty)$)
 is defined to be the
set of all sequences $\{g_j\}_{j\in\zz_+}$ of
measurable functions on $M$ such that
$$\|\{g_j\}_{j\in\zz_+}\|_{L^p_\tau(\ell^q)}
:= \sup_{\gfz{k\in\zz}{\az\in I_k}} \frac1{|Q^k_\az|^\tau}
\bigg\{\int_{Q^k_\az}\bigg[\sum_{j=k\vee0}^\fz
|g_j(x)|^q\bigg]^{p/q}\,d\mu(x)\bigg\}^{1/p}<\fz$$
with a suitable modification made when  $q=\infty$.
\end{defn}

With the notation in Definition \ref{def6.1x}, we rewrite
the (quasi-)norms in Definition \ref{def3.2x} as follows:
$$\|f\|_{B_{p,q}^{s,\tau}(M)}
=\|\{\delta^{-js}\Phi_j(\sqrt L) f\}_{j\in\zz_+}\|_{\ell^q(L^p_\tau)},
\quad
\|f\|_{\wz B_{p,q}^{s,\tau}(M)}
=\|\{|B(\cdot, \dz^j)|^{-s/d}\Phi_j(\sqrt L) f\}_{j\in\zz_+}\|_{\ell^q(L^p_\tau)},
$$
and
$$\|f\|_{F_{p,q}^{s,\tau}(M)}
=\|\{\delta^{-js}\Phi_j(\sqrt L)f\}_{j\in\zz_+}\|_{L^p_\tau(\ell^q)},
\quad
\|f\|_{\wz F_{p,q}^{s,\tau}(M)}
=\|\{|B(\cdot, \dz^j)|^{-s/d}\Phi_j(\sqrt L)f\}_{j\in\zz_+}\|_{L^p_\tau(\ell^q)}.
$$

\begin{thm}\label{thm6.2x}
Let $(\Psi_0, \Psi)$ satisfy \eqref{eq:2.14xx} and \eqref{eq:2.15xx}.
For $j\in\nn$, define $\Psi_j$ as in \eqref{eq:2.16xx}. Let $s\in\rr$
and $[\Psi_j(\sqrt{L})]_a^*f$ with $a\in(0,\fz)$ and $j\in\zz$
be defined as in Definition \ref{def4.1x}.
\begin{enumerate}
\item[\rm (i)] Let $\tau\in[0,\infty)$ and $q\in(0,\infty]$.
If $p\in(0,\infty]$ and $a>d(\tau+1/p)$, then there exists a constant $C\in[1,\fz)$ such that, for all $f\in\cd'(M)$,
\begin{eqnarray}\label{eq:6.1x}
C^{-1}\|f\|_{B_{p,q}^{s,\tau}(M)}
\le
\|\{\delta^{-js}[\Psi_j(\sqrt L)]_a^* f\}_{j\in\zz_+}\|_{\ell^q(L^p_\tau)}
\le C\|f\|_{B_{p,q}^{s,\tau}(M)}
\end{eqnarray}
and
\begin{eqnarray}\label{eq:6.2x}
C^{-1}\|f\|_{\wz B_{p,q}^{s,\tau}(M)}
\le\|\{[\Psi_j(\sqrt L)]_{a,-s/d}^* f\}_{j\in\zz_+}\|_{\ell^q(L^p_\tau)}
\le C\|f\|_{\wz B_{p,q}^{s,\tau}(M)}.
\end{eqnarray}

\item[\rm (ii)] Let $\tau\in[0,\infty)$ and $q\in(0,\infty]$.
If $p\in(0,\infty)$ and $a>d[\tau+1/(p\wedge q)]$, then there exists a constant $C\in[1,\fz)$ such that, for all $f\in\cd'(M)$,
\begin{eqnarray}\label{eq:6.3x}
C^{-1}\|f\|_{F_{p,q}^{s,\tau}(M)}
\le
\|\{\delta^{-js}[\Psi_j(\sqrt L)]_a^* f\}_{j\in\zz_+}\|_{L^p_\tau(\ell^q)}
\le C\|f\|_{F_{p,q}^{s,\tau}(M)}
\end{eqnarray}
and
\begin{eqnarray}\label{eq:6.4x}
C^{-1}\|f\|_{\wz F_{p,q}^{s,\tau}(M)}
\le\|\{[\Psi_j(\sqrt L)]_{a,-s/d}^* f\}_{j\in\zz_+}\|_{L^p_\tau(\ell^q)}
\le C\|f\|_{\wz F_{p,q}^{s,\tau}(M)}.
\end{eqnarray}

\item[\rm (iii)] The spaces $B_{p,q}^{s,\tau}(M)$, $\wz B_{p,q}^{s,\tau}(M)$,
$F_{p,q}^{s,\tau}(M)$ and $\wz F_{p,q}^{s,\tau}(M)$
are independent of the choices of the functions $(\Phi_0,\Phi)$
satisfying \eqref{eq:2.14xx} and \eqref{eq:2.15xx}.
\end{enumerate}
\end{thm}

The Peetre maximal function characterizations for Besov-type and
Triebel--Lizorkin-type
spaces on $\rn$ can be found in \cite{YY1,lsuyy}.
To prove Theorem \ref{thm6.2x}, we need
the following estimates.

\begin{lem}\label{lem6.3x}
Let $\varepsilon\in(0,1)$, $\sigma\in(0,\infty)$  and $\{\varepsilon_j\}_{j=0}^\infty\subset[0,1]$.
When
$\sz\in(1,\infty)$, assume that
$$\bigg(\sum_{j=0}^\infty |\varepsilon_j|^{\varepsilon \sz'}\bigg)^{1/\sz'}=:B<\infty,$$
where $\sz'$ denotes the conjugate index of $\sz$.
Then, for any sequence $\{a_j\}_{j=0}^\infty\subset\cc$,
\begin{eqnarray}\label{eq:6.5x}
\bigg(\sum_{j=0}^\infty |\varepsilon_j a_j| \bigg)^\sigma
\le \max\{B^{\sz},1\}\,\sum_{j=0}^\infty |\varepsilon_j|^{\sigma (1-\varepsilon)} |a_j|^\sz.
\end{eqnarray}
\end{lem}

\begin{proof}
If $\sz\in(0,1]$,
 then \eqref{eq:6.5x} follows from the facts that every $|\varepsilon_j|\le1$ and
$$\bigg(\sum_{j=0}^\infty |\varepsilon_j||a_j|\bigg)^\sz\le \sum _{j=0}^\infty |\varepsilon_j|^{\sz}|a_j|^\sz
\le \sum_{j=0}^\infty |\varepsilon_j|^{\sz(1-\varepsilon)}|a_j|^\sz.$$
If $\sz\in(1,\infty)$, then H\"older's inequality implies that
 \begin{eqnarray*}
\bigg(\sum_{j=0}^\infty |\varepsilon_j a_j| \bigg)^\sigma
\le \bigg(\sum_{j=0}^\infty |\varepsilon_j|^{\varepsilon \sz'} \bigg)^{\sz/\sz'}
\sum_{j=0}^\infty |\varepsilon_j|^{\sigma (1-\varepsilon)} |a_j|^\sz
\le B^{\sz}\sum_{j=0}^\infty |\varepsilon_j|^{\sigma (1-\varepsilon)} |a_j|^\sz.
 \end{eqnarray*}
 This finishes the proof of Lemma \ref{lem6.3x}.
\end{proof}

The following estimates were essentially established in \cite[Lemma~2.3]{YY1};
however we give a much simpler proof here by using Lemma \ref{lem6.3x}.

\begin{lem}\label{lem6.4x}
Let $b\in(0,1)$, $q\in(0,\fz]$, $\tau\in[0,\fz)$ and $\theta\in(d\tau,\fz)$. Suppose
that $\{g_m\}_{m\in\zz_+}$ are measurable functions on $M$. For
all $j\in\zz_+$, let
$G_j:=\sum_{m\in\zz_+}b^{|m-j|\theta}g_m.$
\begin{enumerate}
\item[{\rm (i)}] If $p\in(0,\fz]$, then there exists a positive constant $C$,
independent of $\{g_m\}_{m\in\zz_+}$, such that
$$\|\{G_j\}_{j\in\zz_+}\|_{\ell^q(L^p_\tau)}\le C
\|\{g_m\}_{m\in\zz_+}\|_{\ell^q(L^p_\tau)}.$$

\item[{\rm (ii)}] If $p\in(0,\fz)$, then there exists a positive constant $C$,
independent of $\{g_m\}_{m\in\zz_+}$, such that
$$\|\{G_j\}_{j\in\zz_+}\|_{L^p_\tau(\ell^q)}\le C
\|\{g_m\}_{m\in\zz_+}\|_{L^p_\tau(\ell^q)}.$$
\end{enumerate}
\end{lem}

\begin{proof}
By similarity, we only prove (i).
Let $k\in\zz$ and $\az\in I_k$. Fix $\ez\in(0,1)$.
Applying Lemma \ref{lem6.3x} twice, we find that
\begin{eqnarray}\label{eq:6.6x}
{\rm J}^k_\az:=&&\frac1{|Q^k_\az|^\tau}
\bigg\{\sum_{j=k\vee0}^\fz \bigg[\int_{Q^k_\az}
\bigg|\sum_{m=0}^\fz
b^{|m-j|\theta} g_m(x)\bigg|^p\,d\mu(x)\bigg]^{q/p}\bigg\}^{1/q}\notag\\
\ls&&\frac1{|Q^k_\az|^\tau}
\bigg\{\sum_{j=k\vee0}^\fz \sum_{m=0}^\fz b^{|m-j|\theta q(1-\ez)^2}\bigg[ \int_{Q^k_\az}
\lf|g_m(x)\r|^p\,d\mu(x)\bigg]^{q/p}\bigg\}^{1/q}.
\end{eqnarray}
Split the summation $\sum_{m=0}^\fz$ in the last formula of \eqref{eq:6.6x} into
$\sum_{m=k\vee 0}^\infty$ and $\sum_{m=0}^{(k\vee 0)-1}$, so
the last formula of \eqref{eq:6.6x} is
controlled by the sum of the corresponding two terms, denoted by
${\rm J}^{k,1}_\az$ and ${\rm J}^{k,2}_\az$.
Clearly, by $b\in(0,1)$, we have
$$
{\rm J}^{k,1}_\az
:= \frac1{|Q^k_\az|^\tau}
\bigg\{\sum_{j=k\vee0}^\fz \sum_{m=k\vee 0}^\fz b^{|m-j|\theta q(1-\ez)^2}\bigg[ \int_{Q^k_\az}
\lf|g_m(x)\r|^p\,d\mu(x)\bigg]^{q/p}\bigg\}^{1/q}
\ls \|\{g_m\}_{m\in\zz_+}\|_{\ell^q(L^p_\tau)}.
$$
Notice that ${\rm J}^{k,2}_\az$ is void when $k\le 0$.
If $k>0$ and $0\le m\le k-1$,
then $Q^k_\az$ is covered by a finite number of cubes,
$\{Q^{m}_{\beta}\}_{\beta\in I}$, with $\#I\ls 1$ uniformly in $m$.
Any $Q_\bz^m$ must intersect $Q_\az^k$,
so  $|Q_\bz^m|\ls \dz^{(m-k)d}|Q_\az^k|.$
Since $b,\,\dz\in(0,1)$, there exists a unique $k_0\in\nn$ such that
$\dz^{k_0}\le b <\dz^{k_0-1}.$
Since $\tz>d\tau$, we choose $\ez\in(0,1)$ such that $k_0\tz(1-\ez)^2>d\tau$. Thus,
\begin{eqnarray*}
{\rm J}^{k,2}_\az&&\ls \sum_{\beta\in I}
\frac1{|Q^k_\az|^\tau}\bigg\{\sum_{j=k}^\fz
\sum_{m=0}^{k-1}b^{(j-m)\theta q(1-\ez)^2}\bigg[\int_{Q^m_\beta}
\lf|g_m(x)\r|^p\,d\mu(x)\bigg]^{q/p}\bigg\}^{\frac 1q}\\
&&\ls \sum_{\beta\in I}\bigg\{\sum_{j=k}^\fz
\sum_{m=0}^{k-1}b^{(j-m)\theta q(1-\ez)^2}
\delta^{(m-k)d\tau q}
\bigg\}^{\frac 1q}  \|\{g_m\}_{m\in\zz_+}\|_{\ell^q(L^p_\tau)}
\ls \|\{g_m\}_{m\in\zz_+}\|_{\ell^q(L^p_\tau)}.
\end{eqnarray*}
This finishes the proof of Lemma \ref{lem6.4x}.
\end{proof}

\begin{proof}[Proof of Theorem \ref{thm6.2x}]
Observe that (iii) follows directly from (i), (ii) and
the fact that, for all $j\in\zz_+$, $x\in M$, $a\in(0,\infty)$ and $\gz\in\rr$,
$$|B(x, \dz^{j})|^{\gz} |\Psi_j(\sqrt L) f(x)|
\le [\Psi_j(\sqrt L)]_{a,\gz}^* f(x).$$
To obtain (i) and (ii),
we only prove \eqref{eq:6.2x} and \eqref{eq:6.4x}
since
the proofs for \eqref{eq:6.1x} and \eqref{eq:6.3x}
are similar and easier.

Let  $k\in\zz$ and  $j\ge (k\vee 0)$.
Choose smooth even functions $(\Phi_0,\Phi)$ satisfying \eqref{eq:2.14xx} and \eqref{eq:2.15xx}.
For $j\in\nn$, define $\Phi_j$ as in \eqref{eq:2.16xx}.
With $(\Psi_0, \Psi)$ as in Theorem \ref{thm6.2x},
there exist $(\wz \Psi_0, \wz\Psi)$ such that the Calder\'on
reproducing formula \eqref{c-crf} holds true. Thus, for all $f\in \cd'(M)$
and $x\in M$,
\begin{eqnarray}\label{eq:6.7x}
\Phi_j(\sqrt L) f(x)
&& = \sum_{\ell=0}^\infty \int_M \Phi_j(\sqrt L) \wz\Psi_\ell(\sqrt L) (x,y)
\Psi_\ell(\sqrt L)f(y)\,d\mu(y)
\end{eqnarray}
in $\cd'(M)$,
where $\wz\Psi_\ell(\cdot):=\wz\Psi(\dz^{\ell \bz_0/2} \cdot)$ for all $\ell\in\nn$.
Since the kernel
$\Phi_j(\sqrt L) \wz\Psi_\ell(\sqrt L)  $ is non-zero
only when $\supp\Phi_j\cap\supp\wz\Psi_\ell\neq\emptyset$,
the summation in \eqref{eq:6.7x} is  only valid for
 $\ell$
satisfying $|\ell-j|\le 2$.
For such $\ell$, applying Proposition \ref{prop2.10x}(i)
and Lemma \ref{lem2.1x}(ii) implies that, for any $x,\,y\in M$,
\begin{eqnarray*}
|\Phi_j(\sqrt L) \wz\Phi_\ell(\sqrt L) (x,y)|
&&\ls \int_M D_{\delta^{j},\, \sz}(x,z)  D_{\delta^{\ell},\, \sz} (z,y)\,d\mu(z)
\ls D_{\delta^\ell,\, \sz}(x,y),
\end{eqnarray*}
where we choose $\sz>|s|+a+d$.
Then, Lemma \ref{lem2.1x}(i) tells us  that, for all $x\in M$,
\begin{eqnarray*}
&&|B(x, \delta^{j})|^{-s/d}|\Phi_j(\sqrt L) f(x)|\\
&&\ \ls |B(x, \delta^{j})|^{-s/d}\sum_{{\ell\ge0},\,{|\ell-j|\le2}}\,
\int_M D_{\delta^\ell,\, \sz}(x,y)
|\Psi_\ell(\sqrt L)f(y)|\,d\mu(y)\notag\ls \sum_{{\ell\ge0},\,{|\ell-j|\le2}}\,
[\Psi_\ell(\sqrt{L})]^*_{a,-s/d}f(x). \notag
\end{eqnarray*}
Notice that the sum in $\ell$ has at most $5$ terms.
Thus, by Lemma \ref{lem6.4x}, we conclude that
\begin{equation}\label{eq:6.8x}
\|f\|_{\wz B_{p,q}^{s,\tau}(M)}=
\|\{|B(\cdot, \delta^{j})|^{-s/d}\Phi_j(\sqrt L) f\}_{j\in\zz_+}\|_{\ell^q(L^p_\tau)}
\ls\|\{[\Psi_j(\sqrt L)]_{a,-s/d}^* f\}_{j\in\zz_+}\|_{\ell^q(L^p_\tau)}
\end{equation}
and
\begin{equation}\label{eq:6.9x}
\|f\|_{\wz F_{p,q}^{s,\tau}(M)}=
\|\{|B(\cdot, \delta^{j})|^{-s/d}\Phi_j(\sqrt L) f\}_{j\in\zz_+}\|_{L^p_\tau(\ell^q)}
\ls\|\{[\Psi_j(\sqrt L)]_{a,-s/d}^* f\}_{j\in\zz_+}\|_{L^p_\tau(\ell^q)}.
\end{equation}
This proves the left-hand side inequalities in
\eqref{eq:6.2x} and \eqref{eq:6.4x}, respectively.

To obtain \eqref{eq:6.2x},
we need to prove the converse of \eqref{eq:6.8x}.
Assuming, for the moment, that
\begin{eqnarray}\label{eq:6.10x}
\|\{[\Psi_j(\sqrt L)]_{a,-s/d}^* f\}_{j\in\zz_+}\|_{\ell^q(L^p_\tau)}\ls
\|\{|B(\cdot, \delta^{j})|^{-s/d}\Psi_j(\sqrt L) f\}_{j\in\zz_+}\|_{\ell^q(L^p_\tau)}
\end{eqnarray}
holds true, we prove the converse of \eqref{eq:6.8x}.
Indeed, notice that \eqref{eq:6.8x} remains true if
 we reverse  the roles of $\Phi_j$ and $\Psi_j$ there,
 and \eqref{eq:6.10x} is also true if $\Psi_j$ there is replaced by $\Phi_j$.
Then
\begin{eqnarray*}
\|\{[\Psi_j(\sqrt L)]_{a,-s/d}^* f\}_{j\in\zz_+}\|_{\ell^q(L^p_\tau)}
&&\ls
\|\{|B(\cdot, \delta^{j})|^{-s/d}\Psi_j(\sqrt L) f\}_{j\in\zz_+}\|_{\ell^q(L^p_\tau)}\\
&&\ls
\|\{[\Phi_j(\sqrt L)]_{a,-s/d}^* f\}_{j\in\zz_+}\|_{\ell^q(L^p_\tau)}\\
&&\ls
\|\{|B(\cdot, \delta^{j})|^{-s/d}\Phi_j(\sqrt L) f\}_{j\in\zz_+}\|_{\ell^q(L^p_\tau)},
\end{eqnarray*}
which proves the converse of \eqref{eq:6.8x}.
We still need to prove \eqref{eq:6.10x}.
Fix  $k\in\zz$ and $\az\in I_k$.
Since $a>d(\tau+1/p)$,
there exist $\nu\in(0,\fz)$ and $r\in(0, p)$ such that
$a\ge\nu+d/r$ and $\nu>d\tau$.
By Proposition \ref{prop4.2x}, we know that, for any $x\in Q^k_\az$
 and $\ell \ge (k\vee0)$,
\begin{eqnarray*}
[\Psi_\ell(\sqrt{L})]^*_{a, -s/d}f(x)
\ls  \sum_{j=\ell}^\fz
\sum_{i=0}^\fz \delta^{(j-\ell)\nu} \delta^{i\nu }\mathcal{M}_r
\left(|B(\cdot,\delta^j)|^{-s/d} |\Psi_j(\sqrt L)f|
\chi_{B(z_\az^k, \delta^{\ell-i}+C_\natural \delta^k)}\right)(x). \notag
\end{eqnarray*}
For notational simplicity, we let
$g_j:= |B(\cdot,\delta^j)|^{-s/d} |\Psi_j(\sqrt L)f|$ for any $j\in\zz_+$.
Fix $\epsilon\in(0,1)$ small enough such that
$\nu (1-\epsilon)> d\tau$. Then Lemma \ref{lem6.3x} further implies that
\begin{eqnarray*}
\lf([\Psi_\ell(\sqrt{L})]^*_{a, -s/d}f(x)\r)^p\ls  \sum_{j=\ell}^\fz
\sum_{i=0}^\fz [\delta^{(j-\ell)\nu} \delta^{i\nu }]^{p(1-\epsilon)}
\lf[\mathcal{M}_r
\left(g_j
\chi_{B(z_\az^k, \delta^{k-i}+C_\natural \delta^k)}\right)(x)\r]^p.
\end{eqnarray*}
Next, we integrate both sides of the above inequality in $Q_\az^k$.
The well-known boundedness of  $\cm$ on $L^{p/r}(M)$ further implies that
\begin{eqnarray*}
\int_{Q^k_\az}
\lf\{[\Psi_\ell(\sqrt{L})]^*_{a, -s/d}f(x)\r\}^p\,d\mu(x)
\ls\sum_{j=\ell}^\fz
\sum_{i=0}^\fz [\delta^{(j-\ell)\nu} \delta^{i\nu }]^{p(1-\epsilon)}
\int_{B(z_\az^k, \delta^{k-i}+C_\natural \delta^k)}
|g_j(x)|^p\,d\mu(x).
\end{eqnarray*}
For every $i\in\zz_+$,  define
\begin{equation}\label{eq:6.11x}
\cj_{k,i}:=\lf\{\beta\in I_{k-i}:\, Q_\bz^{k-i} \cap B(z_\az^k, \delta^{k-i}+C_\natural \delta^k)\neq\emptyset\r\}.
\end{equation}
Clearly, the union of all $Q_\bz^{k-i}$ with $\beta\in \cj_{k,i}$
covers the ball
 $B(z_\az^k, \delta^{k-i}+C_\natural \delta^k)$.
Since $\{Q_\bz^{k-i}\}_{\beta\in \cj_{k,i}}$ are mutually disjoint,
it follows that
\begin{equation}\label{eq:6.12x}
\sharp \cj_{k,i} \le C,
\end{equation}
where $C$ is a positive constant depending only on $\delta$ and $K$.
Also,  for any $\bz\in\cj_{k,i}$,
\begin{equation}\label{eq:6.13x}
|Q_\bz^{k-i}|\ls |B(z_\az^k, (C_\natural+1)\delta^{k-i}+ C_\natural \delta^k)|
\ls \dz^{-id} |Q_\az^k|.
\end{equation}
Therefore,
\begin{eqnarray*}
&&\frac1{|Q^k_\az|^{\tau p}}\int_{Q^k_\az}
\lf\{[\Psi_\ell(\sqrt{L})]^*_{a, -s/d}f(x)\r\}^p\,d\mu(x)\\
&&\quad\ls
\sum_{j=\ell}^\fz
\sum_{i=0}^\fz
\sum_{\bz\in\cj_{k,i}}
 \frac{[\delta^{(j-\ell)\nu} \delta^{i\nu }]^{p(1-\epsilon)}\dz^{-id\tau p}}{|Q^{k-i}_\bz|^{\tau p}} \int_{Q_\bz^{k-i}}
|g_j(x)|^p\,d\mu(x).
\end{eqnarray*}
Again, using Lemma \ref{lem6.3x}, we see that
\begin{eqnarray*}
&&\frac1{|Q^k_\az|^{\tau q}}\sum_{\ell=k \vee 0}^\infty \lf[\int_{Q^k_\az}
\lf\{[\Psi_\ell(\sqrt{L})]^*_{a, -s/d}f(x)\r\}^p\,d\mu(x)\r]^{q/p}\\
&&\quad\ls\sum_{\ell=k \vee 0}^\infty\bigg\{
\sum_{j=\ell}^\fz
\sum_{i=0}^\fz
\sum_{\bz\in\cj_{k,i}}
 \frac{[\delta^{(j-\ell)\nu} \delta^{i\nu }]^{p(1-\epsilon)}
\dz^{-id\tau p}}{|Q^{k-i}_\bz|^{\tau p}} \int_{Q_\bz^{k-i}}
|g_j(x)|^p\,d\mu(x)
\bigg\}^{q/p}
\ls\|\{g_j\}_{j\in\zz_+}\|_{\ell^q(L^p_\tau)}^q,
\end{eqnarray*}
where the last step follows from interchanging the summations in $j$ and $\ell$
and then using
$$ \sum_{\ell=k \vee 0}^j
\sum_{i=0}^\fz
\sum_{\bz\in\cj_{k,i}}[\delta^{(j-\ell)\nu} \delta^{i\nu }]^{q(1-\epsilon)^2}
{\dz^{-id\tau q(1-\epsilon)} }
\ls 1.$$
Thus,  \eqref{eq:6.10x} holds true. This finishes the proof of \eqref{eq:6.2x}.

To obtain \eqref{eq:6.4x}, it suffices to prove
the converse of \eqref{eq:6.9x}, which  follows from
\begin{eqnarray}\label{eq:6.14x}
\|\{[\Psi_j(\sqrt L)]_{a,-s/d}^* f\}_{j\in\zz_+}\|_{L^p_\tau(\ell^q)}
\ls\|\{|B(\cdot, \delta^{j})|^{-s/d}\Psi_j(\sqrt L) f\}_{j\in\zz_+}\|_{L^p_\tau(\ell^q)}.
\end{eqnarray}
To see  \eqref{eq:6.14x},
we again use the notation
$g_j:= |B(\cdot,\delta^j)|^{-s/d} |\Psi_j(\sqrt L)f|$, $j\in\zz_+$.
Since  $a>d[\tau+1/(p\wedge q)]$,
we choose $r,\,\nu,\,\ez$ such that $r\in(0, p\wedge q)$,
$d\tau<\nu<a-d/r$ and $\nu (1-\epsilon)^3>d\tau$.
Fix  $k\in\zz$ and $\az\in I_k$.
For any $x\in Q^k_\az$
 and  $\ell \ge (k\vee0)$, by Proposition \ref{prop4.2x} and
 Lemma \ref{lem6.3x},
we find that
\begin{eqnarray*}
&&\lf\{[\Psi_\ell(\sqrt{L})]^*_{a, -s/d}f(x)\r\}^q
\ls  \sum_{j=\ell}^\fz
\sum_{i=0}^\fz [\delta^{(j-\ell)\nu} \delta^{i\nu } ]^{q(1-\epsilon)}
\lf[\mathcal{M}_r
\left(g_j
\chi_{B(z_\az^k, \delta^{k-i}+C_\natural \delta^k)}\right)(x)\r]^q. \notag
\end{eqnarray*}
Applying the vector-valued Fefferman-Stein maximal inequality
for spaces of homogeneous type (see \cite[Theorem 1.3]{s05} and \cite{GLY-ms}), we obtain
\begin{eqnarray*}
\mathrm{Z}_\az^k:=&&\frac1{|Q_\az^k|^\tau}
\bigg\{ \dint_{Q_\alpha^k}\bigg( \sum_{\ell= k \vee 0}^\infty
\lf\{[\Psi_\ell(\sqrt{L})]^*_{a, -s/d}f(x)\r\}^q\bigg)^{p/q}\,d\mu(x)\bigg\}^{1/p}\\
\ls&&
\sum_{i=0}^\infty \delta^{i\nu (1-\epsilon)^3}
\frac1{|Q_\az^k|^\tau}\bigg\{ \dint_{B(z_\az^k, \delta^{k-i}+C_\natural \delta^k)}\Bigg[ \sum_{j= k \vee 0}^\infty
|g_j(x)|^q\Bigg]^{p/q}\,d\mu(x)\bigg\}^{1/p},
\end{eqnarray*}
where, in the third step, we used Lemma \ref{lem6.3x} twice and,
in the fourth step, we interchanged the summations in $j$ and $\ell$.
For $i\in\zz_+$,  define $\cj_{k,i}$ as in \eqref{eq:6.11x}.
By  \eqref{eq:6.12x} and \eqref{eq:6.13x}, together with the fact
that $\sum_{i=0}^\infty \delta^{i\nu (1-\epsilon)^3} \delta^{-id\tau}
\ls1$, we have
\begin{eqnarray*}
\mathrm{Z}_\az^k&& \ls
\sum_{i=0}^\infty \delta^{i\nu (1-\epsilon)^3}
\bigg\{ \sum_{\beta\in\cj_{k,i}}
\frac{\delta^{-id\tau p}}{|Q_\bz^{k-i}|^{\tau p}}
\dint_{Q_\bz^{k-i} }\bigg[ \sum_{j= k \vee 0}^\infty
|g_j(x)|^q\bigg]^{p/q}\,d\mu(x)\bigg\}^{1/p}
\ls \|\{g_j\}_{j\in\zz_+}\|_{L^p_\tau(\ell^q)}.
\end{eqnarray*}
Thus, \eqref{eq:6.14x} holds true. This proves
\eqref{eq:6.4x} and hence Theorem \ref{thm6.2x}.
\end{proof}

%%%%%%%%%%%%%%%%%%%%%%%%%%%%%%%%%%%%%%%%%%%%%%%%%%%%%%%%%%%%%%%%%%%%%

\subsection{Heat kernel characterizations}\label{sec-6.2}

\hskip\parindent
Applying the Peetre maximal function characterizations in Theorem \ref{thm6.2x},
we now characterize Besov-type and Triebel--Lizorkin-type spaces via the heat semigroup.

\begin{defn}\label{def6.5x}
Let  $\tau\in[0,\fz)$, $s\in\rr$ and $m\in\nn$ such that $m>s/\bz_0$.
For any $j\in\nn$ and $\lz\in(0,\infty)$, define
\begin{equation}\label{eq:6.15x}
h_0(\lz):= e^{-\lz^2},\qquad h(\lz):=
\lz^{2m} e^{-\lz^2}\quad
\textup{and}\quad
h_j(\lz):= h(\dz^{j\bz_0/2}\lz)=
(\dz^{j\bz_0/2}\lz)^{2m} e^{-\dz^{j\bz_0}\lz^2}.
\end{equation}
If $p,\,q\in(0,\infty]$, then define $_h\!B^{s,\tau}_{p, q}(M)$ to be the collection of all
 $f\in\cd'(M)$ such that
\begin{eqnarray*}
\|f\|_{_h\!B^{s,\tau}_{p, q}(M)}
&&:=\|\{\delta^{-js}
h_j(\sqrt L) f\}_{j\in\zz_+}\|_{\ell^q(L^p_\tau)}<\infty,
\end{eqnarray*}
and  define $_h\!\wz B^{s,\tau}_{p, q}(M)$ to be the collection of all
 $f\in\cd'(M)$ such that
\begin{eqnarray*}
\|f\|_{_h\!\wz B^{s,\tau}_{p, q}(M)}
&&:=\|\{|B(\cdot, \delta^{j})|^{-s/d}h_j(\sqrt L)  f\}_{j\in\zz_+}\|_{\ell^q(L^p_\tau)}<\infty.
\end{eqnarray*}
If $p\in(0,\infty)$ and $q\in(0,\infty]$, then define $_h\!F^{s,\tau}_{p, q}(M)$ to be the collection of all
 $f\in\cd'(M)$ such that
$$\|f\|_{_h\!F^{s,\tau}_{p, q}(M)}
:=\|\{\delta^{-js}h_j(\sqrt L) f\}_{j\in\zz_+}\|_{L^p_\tau(\ell^q)}<\infty,$$
and define $_h\!\wz F^{s,\tau}_{p, q}(M)$ to be the collection of all
 $f\in\cd'(M)$ such that
$$\|f\|_{_h\!\wz F^{s,\tau}_{p, q}(M)}:=
\|\{|B(\cdot, \delta^{j})|^{-s/d}h_j(\sqrt L) f\}_{j\in\zz_+}
\|_{L^p_\tau(\ell^q)}<\infty.$$
\end{defn}

\begin{rem}\label{rem6.6x}
Let $h_j$ for $j\in\zz_+$ be as in \eqref{eq:6.15x}.
For any  $\sz>0$ and $k\in \zz_+$,
it follows, from  Proposition \ref{prop2.10x}(i), that
$$|L^k h_j(\sqrt L)(x,y)| \ls \dz^{-2 k}D_{\dz^j,\sz}(x,y),\qquad j\in\zz_+,\,\  x,\,y\in M.$$
This implies that $h_j(\sqrt L)(x,\cdot)\in\cd(M)$, so that $h_j(\sqrt L)f$ makes sense for any $f\in\cd'(M)$.
\end{rem}

Now we prove the following discrete  heat kernel characterizations for the full range of $p$.

\begin{thm}\label{thm6.7x}
Let $q\in(0,\infty]$, $\tau\in[0,\fz)$, $s\in\rr$ and $m\in\nn$
such that $m>s/\bz_0$.
\begin{enumerate}

\item[\rm(i)] If $p\in(0,\infty]$, then  $B^{s,\tau}_{p, q}(M)=\, _h\!B^{s,\tau}_{p, q}(M)$ and $\wz B^{s,\tau}_{p, q}(M)={_h\!\wz B^{s,\tau}_{p, q}(M)}$
    with equivalent (quasi-) norms.

\item[\rm(ii)] If $p\in(0,\infty)$, then  $F^{s,\tau}_{p, q}(M)=\, _h\!F^{s,\tau}_{p, q}(M)$ and $\wz F^{s,\tau}_{p, q}(M)={_h\!\wz F^{s,\tau}_{p, q}(M)}$
    with equivalent (quasi-) norms.
\end{enumerate}
\end{thm}

\begin{proof}
Due to similarity, we only show $\wz B^{s,\tau}_{p, q}(M)={_h\!\wz B^{s,\tau}_{p, q}(M)}$
and $\wz F^{s,\tau}_{p, q}(M)={_h\!\wz F^{s,\tau}_{p, q}(M)}$.
With $(\Phi_0, \Phi)$ as in Definition \ref{def3.2x},
there exist $(\wz \Phi_0, \wz\Phi)$ satisfying \eqref{eq:2.14xx} and \eqref{eq:2.15xx}
such that  \eqref{c-crf} holds true. Let $h_j$ for $j\in\zz_+$ be as in \eqref{eq:6.15x}.
Hence, for all $j\in\zz_+$, $f\in\cd'(M)$ and $x\in M$,
\begin{eqnarray}\label{eq:6.16x}
h_j(\sqrt L)  f(x)
&& = \sum_{\ell=0}^\infty
h_j(\sqrt L)\wz\Phi_\ell(\sqrt L)\Phi_\ell(\sqrt L)f(x)
\end{eqnarray}
in $\cd'(M)$.

Given any $\sz>0$, we claim that, for all $j,\,\ell\in\zz_+$
and $x,\,y\in M$,
\begin{eqnarray}\label{eq:6.17x}
|h_j(\sqrt L)\wz\Phi_\ell(\sqrt L)(x,y) |
\ls  \dz^{(j-\ell)\bz_0 m} e^{-\dz^{(j-\ell+1)\bz_0/2} }
D_{\dz^\ell, \sz}(x,y).
\end{eqnarray}
Indeed, if $j=\ell=0$, then \eqref{eq:6.17x} follows directly from Corollary \ref{cor2.6x},
Proposition \ref{prop2.10x}(ii) and Lemma \ref{lem2.1x}(ii).
If $j\in\nn$ and $\ell=0$, we let
$\Psi(\lz):= \dz^{j\bz_0 m}\lz^{2m} e^{-\dz^{j\bz_0}\lz^2}
\wz\Phi_0(\lz)$ for all $\lz\in\rr_+.$
Then $h_j(\sqrt L)\wz\Phi_0(\sqrt L)=\Psi(\sqrt L)$.
For any $\nu,\,i\in\zz_+$, we observe that $\Psi$ is compactly supported and
$$\Psi^{(2\nu+1)}(0)=0,\qquad
|\Psi^{(\nu)}(\lz)| \le  C_{i,\nu} \dz^{j\bz_0 m} (1+\lz)^{-i},$$
where $C_{i,\nu}\in(0,\fz)$ is a positive constant
independent of $\lz$ and $j$. Hence,
Proposition \ref{prop2.10x}(i) implies that, for all $x,\,y\in M$,
\begin{eqnarray*}
|h_j(\sqrt L)\wz\Phi_0(\sqrt L)(x,y) |=|\Psi(\sqrt L)(x,y)|
\ls \dz^{j\bz_0 m} D_{1, \sz}(x,y),
\end{eqnarray*}
which proves \eqref{eq:6.17x} for the case $j\in\nn$ and $\ell=0$.
If $j=0$ and $\ell\in\nn$, we let
$$
\omega(\lz):= \dz^{(j-\ell)m\bz_0}\lz^{2m}
e^{-\dz^{(j-\ell)\bz_0}\lz^2} \wz\Phi(\lz),\qquad \lz\in\rr_+.
$$
Then $h_0(\sqrt L)\wz\Phi_\ell(\sqrt L)=\omega(\dz^{\ell \bz_0/2}\sqrt L)$.
 By the properties of $\wz\Phi$, we have
 $\supp \omega\subset [\delta^{\bz_0/2}, \delta^{-\bz_0/2}]$
 and,  for any $\nu,\,i\in\zz_+$,
$$
|\omega^{(\nu)}(\lz)| \le  C_{i,\nu}
\dz^{(j-\ell)\bz_0 m} e^{-\dz^{(j-\ell+1)\bz_0/2} }(1+\lz)^{-i},$$
where $C_{i,\nu}\in(0,\fz)$ is a positive constant
independent of $\lz$, $j$ and $\ell$.
Again, applying Proposition \ref{prop2.10x}(i), we find that
\begin{eqnarray*}
|h_j(\sqrt L)\wz\Phi_\ell(\sqrt L)(x,y) |=|\omega(\sqrt L)(x,y)|
\ls  \dz^{(j-\ell)\bz_0 m} e^{-\dz^{(j-\ell+1)\bz_0/2}}
D_{\dz^\ell, \sz}(x,y),
\end{eqnarray*}
which proves  \eqref{eq:6.17x} for the case $j=0$ and $\ell\in\nn$.
Altogether, we obtain \eqref{eq:6.17x}.

Fix $k\in\zz$, $j\ge (k\vee0)$, and $x\in Q_\az^k$ for some $\az\in I_k$.
Let $a$ be a sufficiently large number satisfying
the condition of Theorem \ref{thm6.2x}.
By \eqref{eq:6.16x} and \eqref{eq:6.17x}, we see that
\begin{eqnarray*}
|B(x, \dz^j)|^{-s/d} |h_j(\sqrt L)  f(x)|
\ls \sum_{\ell=0}^\infty  \dz^{(j-\ell)\bz_0 m} e^{- \dz^{(j-\ell+1)\bz_0/2} }
\max\{1,\, \dz^{(\ell-j)s}\}
[\Phi_\ell(\sqrt L)]_{a,-s/d}^* f(x).
\end{eqnarray*}
Observe that $\sum_{\ell=0}^\infty  \dz^{(j-\ell)\bz_0 m} e^{-\dz^{(j-\ell+1)\bz_0/2} }
\max\{1,\, \dz^{(\ell-j)s}\}<\infty$ when $m>s/\bz_0$.
From this, Lemma \ref{lem6.4x} and Theorem \ref{thm6.2x}, we deduce that
$$
\|f\|_{_h\!\wz B_{p,q}^{s,\tau}(M)} \ls
\|\{[\Phi_j(\sqrt L)]_{a,-s/d}^* f\}_{j\in\zz_+}\|_{\ell^q(L^p_\tau)}
\sim
\|f\|_{\wz B_{p,q}^{s,\tau}(M)}
$$
and
$$
\|f\|_{_h\!\wz F_{p,q}^{s,\tau}(M)} \ls
\|\{[\Phi_j(\sqrt L)]_{a,-s/d}^* f\}_{j\in\zz_+}\|_{L^p_\tau(\ell^q)}
\sim
\|f\|_{\wz F_{p,q}^{s,\tau}(M)}.
$$
It remains to prove the converse of these two inequalities.
Let  $r\in(0, \min\{1,p,q\})$ and $a$ be sufficiently large.
Define $\phi_0(\lz):= e^{\lz^2}\Phi_0(\lz)$
and $\phi(\lz):= \lz^{-m}e^{\lz^2}\Phi(\lz)$, where $\lz\in\rr_+$.
For any $k\in\zz$, $\ell\in\nn$ such that $\ell\ge (k\vee 0)$, and $x\in Q_\az^k$
for some $\az\in I_k$, the functional calculus gives us that
\begin{eqnarray}\label{eq:6.18x}
|B(x, \delta^{\ell})|^{-s/d}\Phi_\ell(\sqrt L)  f(x)
=|B(x, \delta^{\ell})|^{-s/d}
\phi_\ell(\sqrt L)
h_\ell(\sqrt L) f(x).
\end{eqnarray}
Given any $\sz> a+ 2d+|s|$, Proposition \ref{prop2.10x}(i) implies  that, for all $\ell\in\zz_+$ and $x,\,y\in M$,
\begin{eqnarray}\label{eq:6.19x}
|\phi_\ell(\sqrt L)(x, y)|
\ls D_{\delta^\ell,\sigma}(x,y).
\end{eqnarray}
From \eqref{eq:6.18x} and \eqref{eq:6.19x},
it follows that
\begin{eqnarray}\label{eq:6.20x}
|B(x, \delta^{\ell})|^{-s/d}|\Phi_\ell(\sqrt L)  f(x)|
\ls   [h_\ell(\sqrt L)]_{a, -s/d}^\ast f(x).
\end{eqnarray}
Furthermore,  Proposition \ref{prop4.2x} implies that,
for all $\ell\in\zz_+$ and $x\in M$,
\begin{eqnarray}\label{eq:6.21x}
[h_\ell(\sqrt L)]_{a, -s/d}^\ast f(x)
\ls  \sum_{j=\ell}^\fz
\sum_{i=0}^\fz \delta^{(j-\ell)\nu} \delta^{i\nu r}\mathcal{M}_r
\left(|B(\cdot,\delta^j)|^\gamma |h_j(\sqrt L)f|
\chi_{B(z_\az^k, \delta^{j-i}+C_\natural \delta^k)}\right)(x).\quad
\end{eqnarray}
Invoking this and \eqref{eq:6.20x}, we proceed
the same lines as in the proof  of \eqref{eq:6.10x} to obtain
$$
\|f\|_{\wz B_{p,q}^{s,\tau}(M)}
= \|\{|B(\cdot, \delta^{\ell})|^{-s/d}|\Phi_\ell(\sqrt L)  f|\}_{\ell\in\zz_+}\|_{\ell^q(L^p_\tau)}
\ls \|f\|_{_h\!\wz B_{p,q}^{s,\tau}(M)}.
$$
Likewise, by \eqref{eq:6.20x} and \eqref{eq:6.21x},
we follow the same procedure as that used in the proof of
 \eqref{eq:6.14x} to deduce that
\begin{eqnarray*}
\|f\|_{\wz F_{p,q}^{s,\tau}(M)}
= \|\{|B(\cdot, \delta^{\ell})|^{-s/d}|\Phi_\ell(\sqrt L)  f|\}_{\ell\in\zz_+}\|_{L^p_\tau(\ell^q)}
\ls \|f\|_{_h\!\wz F_{p,q}^{s,\tau}(M)}.
\end{eqnarray*}
This finishes the proof of Theorem \ref{thm6.7x}.
\end{proof}

For all $p\in(0,\fz]$, $\tau\in[0,\fz)$ and $f\in \cd'(M)$, we let
$$\|f\|_{p,\tau}:= \sup_{k\le 0,\ \az\in I_k} \lf[ \frac1{|Q_\az^k|^\tau} \int_{Q_\az^k} |e^{-L}f(x)|^p\,d\mu(x)\r]^{1/p}$$
and
$$\wz{\|f\|}_{p,\tau, s}:= \sup_{k\le 0,\ \az\in I_k} \lf[ \frac1{|Q_\az^k|^\tau} \int_{Q_\az^k} |B(x,1)|^{-sp/d}|e^{-L}f(x)|^p\,d\mu(x)\r]^{1/p}$$
with the usual modifications made when $p=\infty$.
If $\tau=0$, then $\|f\|_{p,\tau}=\|e^{-L}f\|_{L^p(M)}$ and $\wz{\|f\|}_{p,\tau,s}= \||B(\cdot,1)|^{-s/d}e^{-L}f\|_{L^p(M)}$.
Analogous to \cite[Theorems~6.7 and 7.5]{KP}, for $p\ge1$, we derive
the following
continuous versions of the heat semigroup characterizations of the Besov-type and the
Triebel-Lizorkin-type spaces.

\begin{thm}\label{thm6.8x}
Let  $\tau\in[0,\fz)$, $s\in\rr$ and $m\in\nn$
such that $m>s/\bz_0$.
\begin{enumerate}
\item[\rm(i)] If $p\in[1,\infty]$ and $q\in(0,\fz]$, then, for all $f\in \cd'(M)$,
the (quasi-)norm
$\|f\|_{B^{s,\tau}_{p, q}(M)}$
is equivalent to
\begin{equation}\label{eq:6.22x}
\|f\|_{p,\tau} + \sup_{\gfz{k\in\zz}{\az\in I_k}} \frac1{|Q_\az^k|^\tau}
\lf\{\int_0^{\min\{1,\dz^k\}}
\lf[\int_{Q_\az^k} t^{-sp} |(t^{\bz_0}L)^m e^{-t^{\bz_0}L}f(x)|^p\,d\mu(x)\r]^{q/p}\,\frac{dt}{t}\r\}^{1/q},
\end{equation}
and a similar result also holds true
for  $\|\cdot\|_{\wz B^{s,\tau}_{p, q}(M)}$, but with $\|f\|_{p,\tau,s}$ and
$t^{-sp}$ in \eqref{eq:6.22x} replaced by $\wz{\|f\|}_{p,\tau}$
and $|B(x,t)|^{-sp/d}$, respectively.

\item[\rm(ii)] If $p\in[1,\infty)$ and $q\in[1,\fz]$, then,
for all $f\in \cd'(M)$,
the (quasi-)norm
$\|f\|_{F^{s,\tau}_{p, q}(M)}$
is equivalent to
\begin{equation}\label{eq:6.23x}
\|f\|_{p,\tau} + \sup_{\gfz{k\in\zz}{\az\in I_k}} \frac1{|Q_\az^k|^\tau}
\lf\{\int_{Q_\az^k}
\lf[ \int_0^{\min\{1,\dz^k\}} t^{-sq} |(t^{\bz_0}L)^m e^{-t^{\bz_0}L}f(x)|^q\,\frac{dt}{t}\r]^{p/q}\,d\mu(x)\r\}^{1/p},
\end{equation}
and a similar result also holds true for
$\|\cdot\|_{\wz F^{s,\tau}_{p, q}(M)}$, but with $\|f\|_{p,\tau}$ and
$t^{-sq}$ in \eqref{eq:6.23x} replaced by $\wz{\|f\|}_{p,\tau,s}$
and $|B(x,t)|^{-sq/d}$, respectively.
\end{enumerate}
\end{thm}

\begin{proof}
 By similarity, we only consider $\|\cdot\|_{\wz B^{s,\tau}_{p, q}(M)}$
in (i).
Write the first integral in \eqref{eq:6.22x} as
 $$\int_0^{\min\{1,\dz^k\}}\cdots\,\frac{dt}{t} =\sum_{j=k\vee0}^\infty \int_{\dz^{j+1}}^{\dz^j}\cdots\,\frac{dt}{t}.$$
If we observe that $t\sim \dz^j$ when $t\in [\dz^{j+1},\, \dz^j ]$,
then, following the same
procedure as the first part of the proof of Theorem \ref{thm6.7x}, we conclude that,
for any  $p\in(0,\fz]$,
$$\wz{\|f\|}_{p,\tau, s} + \sup_{\gfz{k\in\zz}{\az\in I_k}} \frac1{|Q_\az^k|^\tau}
\lf\{\int_0^{\min\{1,\dz^k\}}
\lf[\int_{Q_\az^k} |B(x,t)|^{-sp/d} |(t^{\bz_0}L)^m e^{-t^{\bz_0}L}f(x)|^p\,d\mu(x)\r]^{q/p}\,\frac{dt}{t}\r\}^{1/q}$$
 is dominated by $\|f\|_{\wz B^{s,\tau}_{p,q}(M)}.$

To prove the converse direction of the above inequality, we assume $p\in[1,\fz]$
and let $\{\Phi_j\}_{j\in\zz_+}$
be as in Definition \ref{def3.2x}.
Fix $x\in Q_\az^k$, with $k\in\zz$ and $\az\in I_k$,
and let  $\sz>d\tau q+|s|+d+1$.
For $\ell\ge 1$ and
$\dz^{\ell+1}\le t< \dz^\ell$
, or for $\ell=0$ and $t=1$,  by an argument similar to that used in
the proof of \eqref{eq:6.20x}, we see that
\begin{eqnarray*}
|B(x, \delta^{\ell})|^{-s/d}|\Phi_\ell(\sqrt L)  f(x)|
\ls \lf[\int_M
D_{\delta^\ell,\sigma-|s|}(x,y)
|B(y, t)|^{-sp/d}
|(t^{\bz_0}L)^m e^{-t^{\bz_0}L}
 f(y)|^p\,d\mu(y)\r]^{1/p},
\end{eqnarray*}
where the second inequality is due to $p\in[1,\fz]$, H\"older's inequality and
Lemma \ref{lem2.1x}(i). Let
$$g_t(\cdot):= \begin{cases}
|B(\cdot, t)|^{-s/d}
|(t^{\bz_0}L)^m e^{-t^{\bz_0}L}
 f(\cdot)|,&\,\qquad t\in(0,1);\\
|B(\cdot, 1)|^{-s/d}
| e^{-L}
 f(\cdot)|,&\,\qquad t=1.\\
 \end{cases}$$
Then we have
\begin{eqnarray*}
|B(x, \delta^{\ell})|^{-sp/d}|\Phi_\ell(\sqrt L)  f(x)|^p \ls
\int_M
D_{\delta^\ell,\sigma-|s|}(x,y)
[g_t(y)]^p\,d\mu(y).
\end{eqnarray*}
Splitting the integral over $M$ into annuals, we obtain
\begin{eqnarray}\label{eq:6.24x}
&&|B(x, \delta^{\ell})|^{-sp/d}|\Phi_\ell(\sqrt L)  f(x)|^p\notag\\
&&\hs\ls
\lf[\int_{B(z_\az^k, C_\natural \dz^k)}
+\sum_{j\in\nn}\int_{
C_\natural \dz^{k-j+1}\le \rho(y, z_\az^k)< C_\natural \dz^{k-j}}\r]D_{\delta^\ell,\sigma-|s|}(x,y)
[g_t(y)]^p\,d\mu(y),
\end{eqnarray}
where $C_\natural $ is as in Lemma \ref{lem3.1x} and  $z_\az^k$ is the
``center" of $Q_\az^k$.
For any $x\in Q_\az^k$ with $\ell\ge (k \vee1)$ and
$C_\natural \dz^{k-j+1}\le \rho(y, z_\az^k)< C_\natural \dz^{k-j}$,  we have
$D_{\delta^\ell,\sigma-|s|}(x,y)\ls \delta^{j(\sigma-|s|-d-1)}D_{\delta^\ell,d+1}(x,y).$
Therefore,
Fubini's theorem implies that
\begin{eqnarray}\label{eq:6.25x}
&&\sum_{\ell=k\vee 1}^\fz \lf[\int_{Q_\az^k} |B(x, \delta^{\ell})|^{-sp/d}|\Phi_\ell(\sqrt L)  f(x)|^p\,d\mu(x)\r]^{q/p}\notag\\
&&\quad\ls \int_{0}^{\min\{1,\delta^k\}} \lf\{\int_{B(z_\az^k, C_\natural \dz^k)}[g_t(y)]^p\,d\mu(y)\r\}^{q/p}\,\frac{dt}t \notag\\
&&\quad\quad+\sum_{j\in\nn}\delta^{j(\sigma-|s|-d-1)[1-\ez]}
\int_{0}^{\min\{1,\delta^k\}}\lf\{\int_{B(z_\az^k, C_\natural \dz^{k-j})}[g_t(y)]^p\,d\mu(y)
\r\}^{q/p}\,\frac{dt}t,
\end{eqnarray}
where, in the last inequality, we used
Lemma \ref{lem6.4x} with $\ez\in(0,1)$.
Observe that  $B(z_\az^k, C_\natural \dz^{k-j})$
can be covered with $N$ Christ cubes $Q_\bz^{k-j}$,
where the integer $N$ is independent of $k,\,j$ and $\az$.
Also, $|Q_\bz^{k-j}|\ls \dz^{-jd} |Q_\az^k|$.
Choose $\ez\in(0,1)$ such that $(\sz-|s|-d-1)(1-\ez)>d\tau q$.
Then
\begin{eqnarray}\label{eq:6.26x}
&&\sup_{\gfz{k\in\zz}{\az\in I_k}}\frac1{|Q_\az^k|^\tau}\Bigg\{\sum_{\ell=k\vee 1}^\fz \bigg[\int_{Q_\az^k} |B(x, \delta^{\ell})|^{-sp/d}|\Phi_\ell(\sqrt L)  f(x)|^p\,d\mu(x)\bigg]^{q/p}\Bigg\}^{1/q}\notag\\
&&\quad\ls\sup_{\gfz{k\in\zz}{\az\in I_k}} \frac1{|Q_\az^k|^\tau}
\lf\{\int_0^{\min\{1,\dz^k\}}
\lf[\int_{Q_\az^k} |B(x,t)|^{-sp/d} |(t^{\bz_0}L)^m e^{-t^{\bz_0}L}f(x)|^p\,d\mu(x)\r]^{q/p}
\,\frac{dt}{t}\r\}^{1/q}.\ \qquad
\end{eqnarray}
By \eqref{eq:6.24x}, together with  an argument similar
to that used in the estimates for
\eqref{eq:6.25x} and \eqref{eq:6.26x}, we have
\begin{eqnarray*}
&&\sup_{\gfz{k\in\zz}{\az\in I_k}}\frac1{|Q_\az^k|^\tau}\Bigg\{\sum_{\ell=k\vee 0}^{(k\vee 1)-1} \bigg[\int_{Q_\az^k} |B(x, \delta^{\ell})|^{-sp/d}|\Phi_\ell(\sqrt L)  f(x)|^p\,d\mu(x)\bigg]^{q/p}\Bigg\}^{1/q}\\
&&\quad\ls\sup_{k\le0,\ \az\in I_k} \frac1{|Q_\az^k|^\tau}
\lf\{\int_0^1\lf[\int_{Q_\az^k}|B(x,1)|^{-sp/d} |e^{-L}f(x)|^p\,d\mu(x)\r]^{q/p}\,\frac{dt}{t}\r\}^{1/q}\ls\wz{\|f\|}_{p,\tau,s},
\end{eqnarray*}
as desired. This proves (i) for $\|\cdot\|_{\wz B^{s,\tau}_{p, q}(M)}$.
The proofs for (ii) and the remainder
of (i)  follow from a similar method, the details being omitted.
This finishes the proof of Theorem \ref{thm6.8x}.
\end{proof}

Taking $\tau=0$ in Theorem \ref{thm6.8x}, we easily obtain the following corollary,
the details being omitted; see \cite[Theorems~6.7 and 7.5]{KP} for the case $\bz_0=2$.

\begin{cor}\label{cor6.9x}
Let   $s\in\rr$ and $m\in\nn$
such that $m>s/\bz_0$.
\begin{enumerate}
\item[\rm(i)] If $p\in[1,\infty]$ and $q\in(0,\fz]$, then, for all $f\in \cd'(M)$,
\begin{equation*}
\|f\|_{B^{s}_{p, q}(M)}\sim \|e^{-L}f\|_{L^p(M)} +
\lf\{\int_0^1
 t^{-sp} \|(t^{\bz_0}L)^m e^{-t^{\bz_0}L}f\|_{L^p(M)}^q\,\frac{dt}{t}\r\}^{1/q},
\end{equation*}
and a similar result also holds true
for  $\|\cdot\|_{\wz B^{s}_{p, q}(M)}$, but with $\|e^{-L}f\|_{L^p(M)} $ and
$t^{-sp}$ in the above formula replaced by $\||B(\cdot,1)|e^{-L}f\|_{L^p(M)} $
and $|B(\cdot,t)|^{-sp/d}$, respectively.

\item[\rm(ii)] If $p\in[1,\infty)$ and $q\in[1,\fz]$, then, for all $f\in \cd'(M)$,
\begin{equation*}
\|f\|_{F^{s}_{p, q}(M)}
\sim
\|e^{-L}f\|_{L^p(M)} +
\lf\|
\lf[ \int_0^1 t^{-sq} |(t^{\bz_0}L)^m e^{-t^{\bz_0}L}f|^q\,\frac{dt}{t}\r]^{1/q}\r\|_{L^p(M)},
\end{equation*}
and
a similar result also holds true
for  $\|\cdot\|_{\wz F^{s}_{p, q}(M)}$, but with $\|e^{-L}f\|_{L^p(M)} $ and
$t^{-sq}$ in the above formula replaced by $\||B(\cdot,1)|e^{-L}f\|_{L^p(M)} $
and $|B(\cdot,t)|^{-sq/d}$, respectively.
\end{enumerate}
\end{cor}

%%%%%%%%%%%%%%%%%%%%%%%%%%%%%%%%%%%%%%%%%%%%%%%%%%%%%%%%%%%%%%%%%%%%%

\section{Frame characterizations}\label{sec-7}

\hskip\parindent
The main aim of this section is to establish  frame characterizations of the
Besov-type and the Triebel--Lizorkin-type spaces.
As an application,
we prove that $F_{p,q}^{s,1/p}(M)$ and $\wz F_{p,q}^{s,1/p}(M)$
are indeed the endpoint case of the Triebel-Lizorkin spaces $F_{\infty,q}^s(M)$ and $\wz F_{\infty,q}^s(M)$, respectively,
where $p\in(0,\infty)$, $q\in(0,\infty]$ and $s\in\rr$.

%%%%%%%%%%%%%%%%%%%%%%%%%%%%%%%%%%%%%%%%%%%%%%%%%%%%%%%%%%%%%%%%%%%%%%%%

\subsection{Frame decompositions}\label{sec-7.1}

\hskip\parindent
Due to the definitions
of the Besov-type and the Triebel--Lizorkin-type spaces,
the frame structure we considered in this section is more specific
and relies on Christ's dyadic cubes in $M$.
Thus, we need to establish a new discrete Calder\'on
reproducing formula adapted to  Christ's dyadic cubes, which is different
from the one  in \cite[Theorem~4.3 and Proposition~5.5]{KP}.

The  discrete Calder\'on reproducing formula is as follows.
For the completeness of the paper,  we present its proof in Section \ref{sec-appendix} below, though a majority of the skills used  comes from \cite{CKP, KP}.

\begin{thm}\label{thm-CRF}
Let $\delta\in(0,1)$ be as in Lemma \ref{lem3.1x}.
Suppose that  $(\Phi_0, \Phi)\in C^\infty(\rr_+)$
satisfy \eqref{eq:2.14xx} and \eqref{eq:2.15xx}.
For any $j\in\nn$, let $\Phi_j(\cdot):= \Phi(\delta^{j\bz_0/2}\cdot)$.
Then there exist a small number $\ez_0\in(0,1)$ and
a sequence $\{\Psi_j(\sqrt L)\}_{j=0}^\infty$ of operators such that the following hold true:
\begin{enumerate}
\item[\rm(a)] For $j\in\zz$ and $\tau\in I_j$, denote by $\{Q_\tau^{j,\nu}:\,
\nu\in\{1,\ldots, N_{\tau}^{j}\}\,\}$ the set
of Christ's dyadic cubes $Q_{\tau'}^{j+j_0}\subset Q_\tau^j$, where $j_0\in\nn$ is some large number determined by $\ez_0$.
Then, for any  $f\in\cd'(M)$
and all $\xi_{\tau}^{j,\nu}\in Q_{\tau}^{j,\nu}$
with $\tau\in I_j$ and $\nu\in\{1,\dots, N_\tau^{j}\}$,
\begin{equation}\label{DCRF}
f(\cdot) =  \sum_{j=0}^\infty\sum_{\tau\in I_j} \sum_{\nu=1}^{N_\tau^{j}}
|Q_{\tau}^{j,\nu}| (\Phi_j(\sqrt L)f )(\xi_\tau^{j,\nu})\,
\Psi_j(\sqrt L)(\xi_\tau^{j,\nu}, \cdot),
\end{equation}
where the series converge in $\cd'(M)$.

\item[{\rm(b)}] For any given $m\in\zz_+$ and $\sz>2d$, there exists a positive constant $C:=C(m,\sz)$ such that, for all $j\in\zz_+$ and $x,\,y\in M$,
$|L^m \Psi_j( \sqrt L)(x, y)|
\le C \delta^{-2mj} D_{\delta^j, \sigma}(x,y).$

\item[{\rm(c)}] For any given $m\in\zz_+$ and $\sz>2d$, there exists a positive constant $C:=C(m,\sz)$ such that, for all $j\in\zz_+$ and $x,\,y,\,y'\in M$ satisfying $\rho(y,y')\le \delta^j$,
\begin{eqnarray*}
&&|L^m \Psi_j( \sqrt L)(x, y)-L^m\Psi_j(\sqrt L)(x, y')|
+| L^m\Psi_j( \sqrt L)(y,x)-L^m\Psi_j(\sqrt L)(y',x)|\\
&&\quad
\le C \delta^{-2mj}\lf[\delta^{-j}\rho(y,y')\r]^{\alpha_0}
D_{\delta^j, \sigma}(x,y).
\end{eqnarray*}

\item[{\rm(d)}] If the smooth functions $(\wz\Psi_0, \wz\Psi)$
satisfy \eqref{eq:2.14xx} and \eqref{eq:2.15xx},
then, for any  $m\in\nn$ and $\sz>2d$,
there exists a positive constant $C:=C(m,\sz)$
such that, for all $j,\,k\in\zz_+$  and $x,\,y\in M$,
\begin{eqnarray*}
\left|(\wz\Psi_k(\sqrt L)\Psi_j(\sqrt L))(x,y)\right|
\le C \delta^{|k-j|(m\bz_0-2d)} D_{\delta^{k\wedge j}, \sigma}(x,y).
\end{eqnarray*}
\end{enumerate}
\end{thm}

\begin{rem}\label{rem7.2x}
The subcubes $\{Q_\tau^{j,\nu}:\, \tau\in I_j,\, 1\le \nu\le N_\tau^j\}$ with $j\in\zz$ are given as follows: we first find some small $\ez_0$ (see \eqref{eq:9.19x}) so that, associated to this specific $\ez_0$, we find a $j_0$ (see \eqref{eq:9.1x} and \eqref{eq:9.2x}) and then, with this $j_0$, we choose all the subcubes in each level $j+j_0$ and denote them by
$\{Q_\tau^{j,\nu}:\, \tau\in I_j,\, 1\le \nu\le N_\tau^j\}$.
In what follows, we fix such notation.
\end{rem}

Now we introduce the related sequence spaces.

\begin{defn}\label{b-type}
Let $s\in\rr$, $\tau\in[0,\infty)$ and $p,\, q\in(0,\infty]$. The \emph{sequence space}
$b^{s,\tau}_{p, q}(M)$ is defined to be the collection of all
sequences $a:=\{a_t^{j,\nu}\}_{j\in\zz_+, t\in I_j, 1\le \nu\le N_t^j}\subset \cc$ such that
\begin{equation*}
\|a\|_{b^{s,\tau}_{p, q}(M)}
:= \sup_{\gfz{k\in\zz}{ \az\in I_k}} \frac1{|Q^k_\az|^\tau}
\lf[
\sum_{j=k\vee 0}^{\infty}
\bigg\{\int_{Q_\az^k} \bigg[\sum_{t\in I_j}\sum_{\nu=1}^{N_t^j}
\dz^{-js} |a_t^{j,\nu}|\chi_{Q_t^{j,\nu}}(x)\bigg]^p\, d \mu(x)\bigg\}^{q/p}
\r]^{1/q}
<\infty.
\end{equation*}
 The \emph{sequence space}
$\wz b^{s,\tau}_{p, q}(M)$ is defined to be the collection of all
 $a:=\{a_t^{j,\nu}\}_{j\in\zz_+, t\in I_j, 1\le \nu\le N_t^j}\subset\cc$ such that
\begin{equation*}
\|a\|_{\wz b^{s,\tau}_{p, q}(M)}
:= \sup_{\gfz{k\in\zz}{ \az\in I_k} } \frac1{|Q^k_\az|^\tau}
\lf[
\sum_{j=k\vee 0}^{\infty}
\bigg\{\int_{Q_\az^k} \bigg[\sum_{t\in I_j}\sum_{\nu=1}^{N_t^j}
|Q_t^{j,\nu}|^{-s/d} |a_t^{j,\nu}|\chi_{Q_t^{j,\nu}}(x)\bigg]^p\, d \mu(x)\bigg\}^{q/p}
\r]^{1/q}
<\infty.
\end{equation*}
\end{defn}

\begin{defn}\label{f-type}
Let $s\in\rr$, $\tau\in[0,\infty)$, $p\in(0,\infty)$ and  $q\in(0,\infty]$. The \emph{sequence space}
$f^{s,\tau}_{p, q}(M)$ is defined to be the collection of all
sequences $a:=\{a_t^{j,\nu}\}_{j\in\zz_+, t\in I_j, 1\le \nu\le N_t^j}\subset\cc$ such that
\begin{equation*}
\|a\|_{f^{s,\tau}_{p, q}(M)}
:= \sup_{\gfz{k\in\zz}{ \az\in I_k}} \frac1{|Q^k_\az|^\tau}
\lf[
\int_{Q_\az^k}
\bigg\{\sum_{j=k\vee 0}^{\infty} \bigg[\sum_{t\in I_j}\sum_{\nu=1}^{N_t^j}
 \dz^{-js}|a_t^{j,\nu}|\chi_{Q_t^{j,\nu}}(x)\bigg]^q\, d \mu(x)\bigg\}^{p/q}
\r]^{1/p}
<\infty.
\end{equation*}
 The \emph{sequence space}
$\wz f^{s,\tau}_{p, q}(M)$ is defined to be the collection of all
 $a:=\{a_t^{j,\nu}\}_{j\in\zz_+, t\in I_j, 1\le \nu\le N_t^j}\subset\cc$ such that
\begin{equation*}
\|a\|_{\wz f^{s,\tau}_{p, q}(M)}
:= \sup_{\gfz{k\in\zz}{ \az\in I_k} } \frac1{|Q^k_\az|^\tau}
\lf[
\int_{Q_\az^k}
\bigg\{\sum_{j=k\vee 0}^{\infty}\bigg[\sum_{t\in I_j}\sum_{\nu=1}^{N_t^j}
|Q_t^{j,\nu}|^{-s/d} |a_t^{j,\nu}|\chi_{Q_t^{j,\nu}}(x)\bigg]^q\, d \mu(x)\bigg\}^{p/q}
\r]^{1/p}
<\infty.
\end{equation*}
\end{defn}

%%%%%%%%%%%%%%%%%%%%%%%%%%%%%%%%%%%%%%%%%%%%%%%%%%%%%%%%%%%%%%%%%%%%

Let the smooth functions $(\Phi_0,\Phi)$
satisfy \eqref{eq:2.14xx} and \eqref{eq:2.15xx}.
Define $\Phi_j$ with $j\in\nn$  as in \eqref{eq:2.16xx}.
By \eqref{DCRF},
there exist $\{\Psi_j\}_{j=0}^\infty$
satisfying (i) through (iii) of Theorem \ref{thm-CRF}
such that, for any  $f\in\cd'(M)$,
$$ f(\cdot)=  \sum_{j=0}^\infty\sum_{t\in I_j} \sum_{\nu=1}^{N_t^{j}}
|Q_{t}^{j,\nu}| (\Phi_j(\sqrt L)f )(\xi_t^{j,\nu})\,
\Psi_j(\sqrt L)(\xi_t^{j,\nu}, \cdot) \qquad \textup{in}\,\, \cd'(M),$$
where $\xi_{t}^{j,\nu}\in Q_{t}^{j,\nu}$, $t\in I_j$ and $\nu\in\{1,\dots, N_t^{j}\}$.
Define the ``analysis" and ``synthesis" operators, respectively, as follows:
\begin{equation*}
S_{\Phi} :\, \, f\to \{(\Phi_j(\sqrt L)f )(\xi_t^{j,\nu})\}_{j\in\zz_+, t\in I_j, 1\le \nu\le N_t^j}
\end{equation*}
and
\begin{equation*}
T_{\Psi} :\, \, \{a_t^{j,\nu}\}_{j\in\zz_+, t\in I_j, 1\le \nu\le N_t^j}
\to  \sum_{j=0}^\infty\sum_{t\in I_j} \sum_{\nu=1}^{N_t^{j}}
|Q_{t}^{j,\nu}| \Psi_j(\sqrt L)(\xi_t^{j,\nu}, \cdot) a_t^{j,\nu}.
\end{equation*}
These operators $S_\Phi$ and $T_{\Psi}$ are generalizations of the $\vz$-transform
and the inverse $\vz$-transform of Frazier and Jawerth \cite{FJ90}.
Notice that  \eqref{DCRF} implies that
$T_\Psi \circ S_\Phi ={\rm Id}$ on $\cd'(M)$, here and hereafter, we use ${\rm Id}$
to denote the \emph{identity operator}. Then we have the following frame characterizations.

\begin{thm}\label{thm7.5x}
Let $\tau\in[0,\infty)$, $s\in\rr$ and $q\in(0,\infty]$.
\begin{enumerate}
\item[\rm(i)] Let $p\in (0,\infty]$. Then the operators
$S_{\Phi}:\, \wz B^{s,\tau}_{p,q}(M)\to \wz b^{s, \tau}_{p, q}(M)$
and  $T_{\Psi}:\, \wz b^{s,\tau}_{p,q}(M)\to \wz B^{s, \tau}_{p, q}(M)$ are bounded,
and
$T_\Psi \circ S_{\Phi} = {\rm Id}$ on $\wz B^{s,\tau}_{p, q}(M).$
Moreover, $f\in \wz B^{s, \tau}_{p, q}(M)$ if and only if
$S_\Phi f \in \wz b^{s, \tau}_{p, q}(M)$,
and  there exists a constant $C\in[1,\fz)$ such that, for all $f\in \wz B^{s, \tau}_{p, q}(M)$,
\begin{equation*}
\frac1C\|f\|_{\wz B^{s,\tau}_{p, q}(M)}
\le \| S_\Phi f\|_{\wz b^{s, \tau}_{p, q}(M)} \le C\|f\|_{\wz B^{s,\tau}_{p, q}(M)}.
\end{equation*}

\item[\rm(ii)] Item ${\rm(i)}$ keeps valid if $\wz B^{s,\tau}_{p,q}(M)$ and $ \wz b^{s, \tau}_{p, q}(M)$ therein
are replaced by $ B^{s,\tau}_{p,q}(M)$ and $  b^{s, \tau}_{p, q}(M)$,
respectively.

\item[\rm(iii)] Let $p\in(0,\infty)$. Then the operators $S_{\Phi}:\, \wz F^{s,\tau}_{p,q}(M)\to \wz  f^{s, \tau}_{p, q}(M)$
and  $T_{\Psi}:\, \wz  f^{s,\tau}_{p,q}(M)\to \wz F^{s, \tau}_{p, q}(M)$ are bounded,
and
$T_\Psi \circ S_{\Phi} = {\rm Id}$ on $\wz F^{s,\tau}_{p, q}(M).$
Moreover, $f\in \wz B^{s, \tau}_{p, q}(M)$ if and only if
$S_\Phi f \in \wz b^{s, \tau}_{p, q}(M)$,
and  there exists a constant $C\in[1,\fz)$ such that, for all $f\in \wz F^{s, \tau}_{p, q}(M)$,
\begin{equation*}
\frac1C\|f\|_{\wz F^{s,\tau}_{p, q}(M)} \le \| S_\Phi f\|_{\wz f^{s, \tau}_{p, q}(M)}
\le C\|f\|_{\wz F^{s,\tau}_{p, q}(M)}.
\end{equation*}
\item[\rm(iv)] Item ${\rm (iii)}$ keeps valid if $\wz F^{s,\tau}_{p,q}(M)$ and $ \wz f^{s, \tau}_{p, q}(M)$ therein
are replaced by $ F^{s,\tau}_{p,q}(M)$ and $f^{s, \tau}_{p, q}(M)$,
respectively.
\end{enumerate}
\end{thm}

Theorem \ref{thm7.5x} generalizes the $\vz$-transform characterization for
Besov-type and Triebel--Lizorkin-type spaces on $\rn$ in \cite{YSY}.
To prove Theorem \ref{thm7.5x},
we need the following two technical lemmas.
In what follows, for any $s\in\rr$, we write
 $b_{\infty,\infty}^s(M):= b_{\infty,\infty}^{s,0}(M)$
and $\wz b_{\infty,\infty}^s(M):= \wz b_{\infty,\infty}^{s,0}(M)$. Then,
for any sequence
$a:=\{a_t^{j,\nu}\}_{j\in\zz_+, t\in I_j, 1\le \nu\le N_t^j}\subset\cc$,
let
\begin{equation}\label{eq:7.2x}
\|a\|_{\wz b_{\infty,\infty}^s(M)}
:=\|a\|_{\wz b_{\infty,\infty}^{s,0}(M)}
:=\sup_{j\in\zz_+}\, \sup_{x\in M}
\sum_{t\in I_j}\sum_{\nu=1}^{N_t^j}
|Q_t^{j,\nu}|^{-s/d} |a_t^{j,\nu}|\chi_{Q_t^{j,\nu}}(x).
\end{equation}
Similarly, we define the norm $\|a\|_{ b_{\infty,\infty}^s(M)}$
via replacing $|Q_t^{j,\nu}|^{-s/d}$ in \eqref{eq:7.2x} by $\dz^{-js}$.

\begin{lem}\label{lem7.6x}
Let $s\in\rr$, $\tau\in[0,\infty)$ and $q\in(0,\infty]$.
\begin{enumerate}
\item[\rm(i)] If $p\in(0,\infty]$, then $b_{p,q}^{s,\tau}(M)
\hookrightarrow b_{\infty,\infty}^{s+d\tau-d/p}(M)$
and $\wz b_{p,q}^{s,\tau}(M)
\hookrightarrow \wz b_{\infty,\infty}^{s+d\tau-d/p}(M)$.
\item[\rm(ii)] If $p\in(0,\infty)$, then $f_{p,q}^{s,\tau}(M)
\hookrightarrow b_{\infty,\infty}^{s+d\tau-d/p}(M)$
and $\wz f_{p,q}^{s,\tau}(M)
\hookrightarrow \wz b_{\infty,\infty}^{s+d\tau-d/p}(M)$.
\end{enumerate}
\end{lem}

\begin{proof}
By Minkowski's inequality, we see that
$$b_{p,\min(p,q)}^{s,\tau}(M) \hookrightarrow f_{p,q}^{s,\tau}(M) \hookrightarrow b_{p, \max(p,q)}^{s,\tau}(M)
\quad\textup{and}\quad
\wz b_{p,\min(p,q)}^{s,\tau}(M)\hookrightarrow \wz f_{p,q}^{s,\tau}(M)
\hookrightarrow \wz b_{p, \max(p,q)}^{s,\tau}(M).$$
Hence, it suffices to show (i).
Notice that  $\{Q_t^{j,\nu}:\ t\in I_j,\ \nu\in\{1,\ldots,N_t^j\}\}$ are
mutually disjoint.
For any fixed  $j\in\zz_+$ and $x\in M$, there exist unique $t_x\in I_j$ and $\nu_x\in \{1,\ldots,N_{t(x)}^j\}$ such that
$x\in Q_{t_x}^{j,\nu_x}$. Hence,
\begin{eqnarray*}
\|a\|_{\wz b_{\infty,\infty}^{s+d\tau-d/p}(M)}
&&=\sup_{j\in\zz_+} \sup_{x\in M}
|Q_{t_x}^{j,\nu_x}|^{-s/d-\tau+1/p} |a_{t_x}^{j,\nu_x}|\\
&&\ls \sup_{k\in\zz_+}\sup_{ \az\in I_k}  \sup_{x\in Q^k_\az}\frac1{|Q^k_\az|^\tau}
\bigg\{
\sum_{j=k}^{\infty}
\bigg[
|Q_{t_x}^{j,\nu_x}|^{1-sp/d} |a_{t_x}^{j,\nu_x}|^p\bigg]^{q/p}
\bigg\}^{1/q}\ls\|a\|_{\wz b_{p,q}^{s,\tau}(M)}.
\end{eqnarray*}
Thus, $\wz b_{p,q}^{s,\tau}(M)
\hookrightarrow \wz b_{\infty,\infty}^{s+d\tau-d/p}(M)$.
Similarly, we have $ b_{p,q}^{s,\tau}(M)
\hookrightarrow  b_{\infty,\infty}^{s+d\tau-d/p}(M)$,
the details being omitted.
This finishes the proof of Lemma \ref{lem7.6x}.
\end{proof}

\begin{lem}\label{lem7.7x}
Let $s\in\rr$, $\tau\in[0,\infty)$ and $p,\, q\in(0,\infty]$.
If the sequence $a:=\{a_t^{j,\nu}\}_{j\in\zz_+, t\in I_j, 1\le \nu\le N_t^j}$ belongs to any of the sequence spaces $b_{p,q}^{s,\tau}(M)$,
$\wz b_{p,q}^{s,\tau}(M)$,  $f_{p,q}^{s,\tau}(M)$
or $\wz f_{p,q}^{s,\tau}(M)$, then the series
\begin{equation}\label{eq:7.3x}
T_{\Psi} (a)(\cdot)= \sum_{j=0}^\infty\sum_{t\in I_j} \sum_{\nu=1}^{N_t^{j}}
|Q_{t}^{j,\nu}| \Psi_j(\sqrt L)(\xi_t^{j,\nu}, \cdot) a_t^{j,\nu}
\end{equation}
converges in $\cd'(M)$.
\end{lem}

\begin{proof}
Due to Lemma \ref{lem7.6x}, it suffices to show that the series in \eqref{eq:7.3x}
converges in $\cd'(M)$ when $a\in  b_{\infty,\infty}^s(M)$
or $a\in \wz b_{\infty,\infty}^s(M)$.
We only show the conclusion for $a\in \wz b_{\infty,\infty}^s(M)$
and $\mu(M)=\infty$, since the proofs for the remaining cases are similar.
Without loss of generality, we may assume that $a\in \wz b_{\infty,\infty}^s(M)$
with norm $1$.
In this case,
$|a_t^{j,\nu}|\le |Q_t^{j,\nu}|^{s/d}$
for all $j\in\zz_+$, $t\in I_j$ and $1\le \nu\le N_t^j$.
Thus, for any $\phi\in\cd(M)$, we have
\begin{eqnarray}\label{eq:7.4x}
\sum_{j=0}^\infty\sum_{t\in I_j} \sum_{\nu=1}^{N_t^{j}}
|Q_{t}^{j,\nu}| |\laz\Psi_j(\sqrt L)(\xi_t^{j,\nu}, \cdot),\,\phi\raz| |a_t^{j,\nu}|
\le \sum_{j=0}^\infty\sum_{t\in I_j} \sum_{\nu=1}^{N_t^{j}}
|Q_{t}^{j,\nu}|^{1+s/d} |\laz\Psi_j(\sqrt L)(\xi_t^{j,\nu}, \cdot),\,\phi\raz|.\quad
\end{eqnarray}
Let $(\Phi_0,\Phi)\in C^\infty(\rr_+)$ satisfy \eqref{eq:2.14xx} and \eqref{eq:2.15xx}.
Define $\{\Phi_j\}_{j\in\nn}$ as in \eqref{eq:2.16xx}.
By \eqref{c-crf},
there exist  $\wz\Phi_0$ and $\wz\Phi$  satisfying \eqref{eq:2.14xx} and \eqref{eq:2.15xx} such that
$\phi= \sum_{\ell=0}^\infty \wz\Phi_\ell(\sqrt L)\Phi_\ell (\sqrt L)\phi,$
where the series converges in both $\cd(M)$ and  $L^2(M)$.
Notice that $\Psi_j(\sqrt L)(\xi_t^{j,\nu}, \cdot)\in L^2(M)$. Hence,
\begin{eqnarray}\label{eq:7.5x}
\laz\Psi_j(\sqrt L)(\xi_t^{j,\nu}, \cdot),\,\phi\raz
=\sum_{\ell=0}^\infty \lf\laz (\overline{\wz\Phi}_\ell(\sqrt L)\Psi_j(\sqrt L))(\xi_t^{j,\nu}, \cdot),\,\, \Phi_\ell (\sqrt L)\phi\r\raz.
\end{eqnarray}
Choose $\eta, N, m, \sz$ such that $\eta>|s|+d$, $\bz_0N>2d+\sz+|s|$,
$\bz_0m>2d+|s|$ and $\sz>3d/2+\eta$.
According to the proof of  \eqref{eq:4.9x},
we know that, for all $\ell\in\zz_+$ and $z\in M$,
\begin{eqnarray*}
|\Phi_\ell(\sqrt L)\phi(z)|
\ls \dz^{\ell(\bz_0N-d)} [1+\rho(z,x_0)]^{-\eta} [\cp_{N,\eta}(\phi)+\cp_{0,\eta}(\phi)],
\end{eqnarray*}
where $x_0$ is some fixed point of $M$. In the sequel, we let
$\mathbb A:= [\cp_{N,\eta}(\phi)+\cp_{0,\eta}(\phi)]$.
By Theorem \ref{thm-CRF}(d),
we see that, for all $j,\,\ell\in\zz_+$, $z\in M$ and
$\xi,\, \xi_t^{j,\nu}\in Q_t^{j,\nu}$,
\begin{eqnarray*}
\left|(\overline{\wz\Phi}_\ell(\sqrt L)\Psi_j(\sqrt L))(\xi_t^{j,\nu}, z)\right|
\ls\delta^{|\ell-j|(m\bz_0-2d)} \dz^{-(\ell\wedge j)(d+\sz)}
D_{1,\sigma}(\xi,z).
\end{eqnarray*}
From these two estimates and \eqref{eq:7.5x}, we deduce that, for any $\xi,\, \xi_t^{j,\nu}\in Q_t^{j,\nu}$,
\begin{eqnarray}\label{eq:7.6x}
|\laz\Psi_j(\sqrt L)(\xi_t^{j,\nu}, \cdot),\,\phi\raz|
\ls  \mathbb A\, \sum_{\ell=0}^\infty
\frac{\delta^{|\ell-j|(m\bz_0-2d)}\dz^{(\bz_0N-2d-\sz)\ell}}{[1+ \rho(\xi,x_0)]^{\eta}}.
\end{eqnarray}
For any $\xi\in Q_t^{j,\nu}$, we have
$$ |Q_t^{j,\nu}|^{s/d}
\ls \dz^{-j |s|} [1+\rho(\xi, x_0)]^{|s|}\ls \dz^{-|\ell-j||s|}\dz^{-\ell |s|} [1+\rho(\xi, x_0)]^{|s|}.
$$
Inserting this and \eqref{eq:7.6x} into \eqref{eq:7.4x}, by $\bz_0N>2d+\sz+|s|$,
$\bz_0m>2d+|s|$, $\eta>|s|+d$ and Lemma \ref{lem2.1x}(i),  we obtain
\begin{eqnarray*}
&&\sum_{j=0}^\infty\sum_{t\in I_j} \sum_{\nu=1}^{N_t^{j}}
|Q_{t}^{j,\nu}| |\laz\Psi_j(\sqrt L)(\xi_t^{j,\nu}, \cdot),\,\phi\raz| |a_t^{j,\nu}|\\
&&\quad\ls
 \mathbb A\,
\sum_{j=0}^\infty\sum_{\ell=0}^\infty
\int_{M}
\frac{{\delta^{|\ell-j|(m\bz_0-2d-|s|)}\dz^{(\bz_0N-2d-\sz-|s|)\ell}}}{[1+ \rho(\xi,x_0)]^{\eta-|s|}}\,d\mu(\xi)\ls
 \mathbb A\,,
\end{eqnarray*}
which implies that \eqref{eq:7.3x} converges in $\cd'(M)$, and hence
completes the proof of Lemma \ref{lem7.7x}.
\end{proof}

\begin{proof}[Proof of Theorem \ref{thm7.5x}]
We only show (i) and (iii),  the proofs for (ii) and (iv) being similar.
Let $f\in\cd'(M)$ and $a>d[\tau+1/(p\wedge q)]$. By \eqref{c-crf},
there exist  $(\wz\Phi_0,\wz\Phi)\in C^\infty(\rr_+)$ satisfying \eqref{eq:2.14xx} and \eqref{eq:2.15xx} such that
$f= \sum_{\ell=0}^\infty \wz\Phi_\ell(\sqrt L)\Phi_\ell (\sqrt L)f$
in $\cd'(M)$.
Given any   $m>(d+|s|+a)/\bz_0$ and $\sz>a+|s|+d$,
applying Proposition \ref{prop2.14x}, we conclude that, for any $j\in\zz_+,$ $t\in I_j$ and  $\nu\in\{1,\dots, N_t^j\}$,
\begin{eqnarray*}
|Q_t^{j,\nu}|^{-s/d} |\Phi_j(\sqrt L)f(\xi_t^{j,\nu})|
&&\ls\sum_{\ell=0}^\infty \dz^{|j-\ell|(m\bz_0-d)}
|Q_t^{j,\nu}|^{-s/d}
\int_M D_{\dz^{j\wedge \ell},\sz}(\xi_t^{j,\nu},\,y)
|\Phi_\ell (\sqrt L)f(y)|\,d\mu(y).
\end{eqnarray*}
For any $x,\,\xi_t^{j,\nu}\in Q_t^{j,\nu}$ and $y\in M$,
we see that
$D_{\dz^{j\wedge \ell},\sz}(\xi_t^{j,\nu},\,y)
\sim D_{\dz^{j\wedge \ell},\sz}(x,\,y)$
and
$$|Q_t^{j,\nu}|^{-s/d} |B(y,\dz^\ell)|^{s/d}
\sim |B(x, \dz^j)|^{-s/d} |B(y,\dz^\ell)|^{s/d}
\ls \dz^{-|j-\ell| |s|}[1+\dz^{-(j\wedge \ell)}\rho(x,y)]^{|s|}.$$
Hence, for any $x\in Q_t^{j,\nu}$,
\begin{eqnarray*}
|Q_t^{j,\nu}|^{-s/d} |\Phi_j(\sqrt L)f(\xi_t^{j,\nu})|
&& \ls
\sum_{\ell=0}^\infty \dz^{|j-\ell|(m\bz_0-d/2-|s|)}
[\Phi_\ell(\sqrt L)f]_{a,-s/d}^\ast(x)\\
&&\quad\times
\int_M
\frac1{\sqrt{|B(x, \dz^{j\wedge \ell})||B(y, \dz^{j\wedge \ell})|}}
\frac{[1+\dz^{- \ell}\rho(x,y)]^a}{[1+\dz^{-(j\wedge \ell)}\rho(x,y)]^{\sz-|s|}}\,d\mu(y)\\
&& \ls
\sum_{\ell=0}^\infty \dz^{|j-\ell|(m\bz_0-d-|s|-a)} [\Phi_\ell(\sqrt L)f]_{a,-s/d}^\ast(x).
\end{eqnarray*}
Consequently, for any $x\in Q_\az^k$, with $k\in\zz$ and $\az\in I_k$,
and any $j\ge (k\vee 0)$,
\begin{eqnarray*}%\label{thm-frame-e1}
g_j(x)&&:=\sum_{t\in I_j}\sum_{\nu=1}^{N_t^j}
|Q_t^{j,\nu}|^{-s/d} |\Phi_j(\sqrt L)f(\xi_t^{j,\nu})|\chi_{Q_t^{j,\nu}}(x)\\
&&\ls\sum_{\ell=0}^\infty \dz^{|j-\ell|(m\bz_0-d-|s|-a)}
[\Phi_\ell(\sqrt L)f]_{a,-s/d}^\ast(x).
\end{eqnarray*}
By this and the definition of $\wz b_{p,q}^{s,\tau}(M)$,
applying Lemma \ref{lem6.4x} and Theorem \ref{thm6.2x}, we know that
\begin{eqnarray*}
\| S_\Phi f\|_{\wz b^{s, \tau}_{p, q}(M)}
=\|\{g_j\}_{j=0}^\infty\|_{\ell^q(L_\tau^p)}
\ls \|\{[\Phi_\ell(\sqrt L)f]_{a,-s/d}^\ast\}_{\ell=0}^\infty\|_{\ell^q(L_\tau^p)}
\ls \|f\|_{\wz B_{p,q}^{s,\tau}(M)},
\end{eqnarray*}
which implies the boundedness of  $S_{\Phi}$
from $\wz B^{s,\tau}_{p,q}(M)$ to $\wz b^{s, \tau}_{p, q}(M)$.
Due to the same reasons,  $S_{\Phi}$
is also bounded from $\wz F^{s,\tau}_{p,q}(M)$ to $\wz f^{s, \tau}_{p, q}(M)$.%

Now we consider the boundedness of $T_\Psi$.
Fix $r\in(0, \min\{1,p,q\})$.
For any $a\in \wz b_{p,q}^{s,\tau}(M)$ or $\wz f_{p, q}^{s,\tau}(M)$,
let
$$
f(\cdot):= T_\Psi a(\cdot)=\sum_{j=0}^\infty\sum_{t\in I_j} \sum_{\nu=1}^{N_t^{j}}
|Q_{t}^{j,\nu}| \Psi_j(\sqrt L)(\xi_t^{j,\nu}, \cdot) a_t^{j,\nu}.
$$
Then $f\in\cd'(M)$ by Lemma \ref{lem7.7x}. Hence,
for any $x\in Q_\az^k$, with $k\in\zz$ and $\az\in I_k$, and  any $l\in\zz_+$,
\begin{eqnarray*}
\Phi_\ell(\sqrt L) f(x)
=\sum_{j=0}^\infty\sum_{t\in I_j} \sum_{\nu=1}^{N_t^{j}}
|Q_{t}^{j,\nu}| (\Phi_\ell(\sqrt L)\Psi_j(\sqrt L))(\xi_t^{j,\nu}, x) a_t^{j,\nu}.
\end{eqnarray*}
From this and Theorem \ref{thm-CRF}(iii),
it follows that
\begin{eqnarray*}%\label{fram-B-e3}
|B(x,\dz^\ell)|^{-s/d}|\Phi_\ell(\sqrt L) f(x)|
\ls \sum_{j=0}^\infty\sum_{t\in I_j} \sum_{\nu=1}^{N_t^{j}}
\delta^{|\ell-j|(m\bz_0-2d)}
|a_t^{j,\nu}||B(x,\dz^\ell)|^{-s/d}
|Q_{t}^{j,\nu}|   D_{\dz^{j\wedge \ell},\sz}(\xi_t^{j,\nu}, x) ,
\end{eqnarray*}
where we chose  $\sz\ge 2d/r+|s|+1$ and $m>(3d+|s|+\sz)/2$.
Notice that
\begin{eqnarray*}%\label{fram-B-e4}
|Q_{t}^{j,\nu}| D_{\dz^{j\wedge \ell},\sz}(\xi_t^{j,\nu}, x)
\ls \dz^{-|j-\ell|d}[1+\dz^{-(j\wedge \ell)}\rho(\xi_t^{j,\nu}, x)]^{-\sz}
\ls \dz^{-|j-\ell|(d+\sz)}[1+\dz^{-j}\rho(\xi_t^{j,\nu}, x)]^{-\sz}
\end{eqnarray*}
and
\begin{eqnarray*}%\label{fram-B-e5}
|B(x,\dz^\ell)|^{-s/d}
\ls |Q_t^{j,\nu}|^{-s/d}
\dz^{-|j-\ell||s|}[1+\dz^{-j}\rho(\xi_t^{j,\nu}, x)]^{|s|}.
\end{eqnarray*}
Therefore,
\begin{eqnarray}\label{eq:7.7x}
|B(x,\dz^\ell)|^{-s/d}|\Phi_\ell(\sqrt L) f(x)|
\ls \sum_{j=0}^\infty
\delta^{|\ell-j|(m\bz_0-3d-\sz-|s|)}
\sum_{t\in I_j} \sum_{\nu=1}^{N_t^{j}}
\frac{|Q_t^{j,\nu}|^{-s/d}|a_t^{j,\nu}|}
{[1+\dz^{-j}\rho(\xi_t^{j,\nu}, x)]^{\sz-|s|}}.
\end{eqnarray}
Since $r\in(0,1)$, it follows that
\begin{eqnarray}\label{eq:7.8x}
\bigg(\sum_{t\in I_j} \sum_{\nu=1}^{N_t^{j}}
\frac{|Q_t^{j,\nu}|^{-s/d}|a_t^{j,\nu}|}
{[1+\dz^{-j}\rho(\xi_t^{j,\nu}, x)]^{\sz-|s|}}\bigg)^r
&&\ls \int_M
\frac{[\sum_{t\in I_j} \sum_{\nu=1}^{N_t^{j}}|Q_t^{j,\nu}|^{-s /d}
|a_t^{j,\nu}|\chi_{Q_t^{j,\nu}}(z)]^r}
{|B(z, \dz^j)|[1+\dz^{-j}\rho(z, x)]^{(\sz-|s|)r}}\,d\mu(z).
\end{eqnarray}
For notational simplicity,  let
$f_j:=\sum_{t\in I_j} \sum_{\nu=1}^{N_t^{j}}|Q_t^{j,\nu}|^{-s /d}|a_t^{j,\nu}|\chi_{Q_t^{j,\nu}}.$
Applying Lemma \ref{lem4.3x} to \eqref{eq:7.8x} with $g$ therein replaced by $|f_j|^r$,
we further find that
\begin{eqnarray*}
\sum_{t\in I_j} \sum_{\nu=1}^{N_t^{j}}
\frac{|Q_t^{j,\nu}|^{-s/d}|a_t^{j,\nu}|}
{[1+\dz^{-j}\rho(\xi_t^{j,\nu}, x)]^{\sz-|s|}}
&&\ls
\bigg\{\sum_{i=0}^\infty \delta^{i(\sz r-|s|r-d)}
\cm(|f_j|^r\chi_{B(z_\az^k, \delta^{k-i}+C_\natural \dz^k)})(x)\bigg\}^{1/r}\\
&&\ls \sum_{i=0}^\infty \delta^{i(\sz r-|s|r-d)}
\cm_r(|f_j|\chi_{B(z_\az^k, \delta^{k-i}+C_\natural \dz^k)})(x),
\end{eqnarray*}
where the second inequality is due to H\"older's inequality.
Combining this with \eqref{eq:7.7x} implies that
\begin{eqnarray}\label{eq:7.9x}
&&|B(x,\dz^\ell)|^{-s/d}|\Phi_\ell(\sqrt L) f(x)|\notag\\
&&\quad\ls \sum_{j=0}^\infty \sum_{i=0}^\infty
\delta^{|\ell-j|(m\bz_0-3d-|s|-\sz)}
 \delta^{i(\sz r-|s|r-d)}
\cm_r(|f_j|\chi_{B(z_\az^k, \delta^{k-i}+C_\natural \dz^k)})(x).
\end{eqnarray}
Then, repeating the proof for \eqref{eq:6.10x}, we  obtain
$$
\|f\|_{\wz B_{p,q}^{s,\tau}(M)}
\ls \|\{f_j\}_{j\in\zz_+}\|_{\ell^q(L_\tau^p)}
\sim \|a\|_{\wz b_{p,q}^{s,\tau}(M)},
$$
which proves  that $T_{\Psi}:\, \wz b^{s,\tau}_{p,q}(M)\to \wz B^{s, \tau}_{p, q}(M)$
is bounded.
Again, applying \eqref{eq:7.9x} and repeating the proof for
 \eqref{eq:6.14x}, we find that
$$
\|f\|_{\wz F_{p,q}^{s,\tau}(M)}
\ls \|\{f_j\}_{j\in\zz_+}\|_{L_\tau^p(\ell^q)}
\sim \|a\|_{\wz f_{p,q}^{s,\tau}(M)},
$$
 which implies  that $T_{\Psi}:\, \wz f^{s,\tau}_{p,q}(M)\to \wz F^{s, \tau}_{p, q}(M)$
is bounded.

By \eqref{DCRF}, we have
$T_\Psi \circ S_{\Phi} = {\rm Id}$ on  $\wz B^{s,\tau}_{p, q}(M)$ or  $\wz F^{s,\tau}_{p, q}(M)$.
If $S_\Phi f\in \wz b_{p,q}^{s,\tau}(M)$, then the boundedness of $T_\Psi$
implies that
$f= T_\Psi ( S_{\Phi} f)\in \wz B_{p,q}^{s,\tau}(M)$
and
$\|f\|_{\wz B_{p,q}^{s,\tau}(M)}=\|T_\Psi ( S_{\Phi} f)\|_{\wz B_{p,q}^{s,\tau}(M)}
\ls \| S_{\Phi} f\|_{\wz b_{p,q}^{s,\tau}(M)}.$
Likewise,
$\|f\|_{\wz F_{p,q}^{s,\tau}(M)}\ls \| S_{\Phi} f\|_{\wz f_{p,q}^{s,\tau}(M)}$ holds true.
This finishes the proof of Theorem \ref{thm7.5x}.
\end{proof}

%%%%%%%%%%%%%%%%%%%%%%%%%%%%%%%%%%%%%%%%%%%%%%%%%%%%%%%%%%%%%

\subsection{The endpoint Triebel-Lizorkin  spaces $F_{\infty,q}^s(M)$ and $\wz F_{\infty,q}^s(M)$}\label{sec-7.2}

\hskip\parindent
For $s\in\rr$ and $q\in(0,\infty]$, inspired by the definition of $F_{\infty, q}^s(\rn)$
on $\rn$ in \cite{FJ90},
we define the \emph{endpoint Triebel-Lizorkin spaces} $F_{\infty, q}^s(M) := F_{q,q}^{s,1/q}(M)$
and $\wz F_{\infty, q}^s(M) := \wz F_{q,q}^{s,1/q}(M)$.
Then we have the following coincidence.

\begin{thm}\label{thm7.8x}
Let $s\in\rr$, $p\in(0,\infty)$ and $q\in(0,\infty]$. Then
$F_{\infty, q}^s(M)=F_{p,q}^{s,1/p}(M)$ and $\wz F_{\infty, q}^s(M)=\wz F_{p,q}^{s,1/p}(M)$ with equivalent (quasi-)norms.
\end{thm}

For the corresponding result on Triebel-Lizorkin spaces on $\rn$, we refer to
\cite[Corollary 5.7]{FJ90}.
Due to Theorem \ref{thm7.5x}, to show Theorem \ref{thm7.8x},
it suffices to prove the following fact.

\begin{prop}\label{prop7.9x}
Let $s\in\rr$, $p\in(0,\infty)$ and $q\in(0,\infty]$. Then
$f_{q, q}^{s,1/q}(M)=f_{p,q}^{s,1/p}(M)$ and $\wz f_{q, q}^{s,1/q}(M)=\wz f_{p,q}^{s,1/p}(M)$ with equivalent (quasi-)norms.
\end{prop}

To prove Proposition \ref{prop7.9x}, we follow the
proof of \cite[Corollary 5.7]{FJ90} and need the following lemmas.
Lemma \ref{lem7.10x} follows from an argument similar to that used in
the proof of Proposition \ref{prop4.10x} (see also \cite[Lemma 2.2]{YSY}),
the details being omitted.

\begin{lem}\label{lem7.10x}
Let $s\in\rr$, $p\in(0,\infty)$ and $q\in(0,\infty]$. If $\tau\in[1/p,\fz)$,
then  $\sup_{{k\in\zz},\,{ \az\in I_k}}$ in the definitions
of $f^{s,\tau}_{p, q}(M)$ and $\wz f^{s,\tau}_{p, q}(M)$
can be equivalently replaced by $\sup_{{k\in\zz_+},\,{ \az\in I_k}}$.
\end{lem}

\begin{lem}\label{lem7.11x}
Let $\varepsilon\in(0,1]$, $s\in\rr$, $p\in(0,
\infty)$, $q\in(0,\infty]$ and $\tau\in[0,\fz)$.
For all $j\in\zz_+$, $t\in I_j$, $1\le \nu\le N_t^j$, let
$S_t^{j,\nu}$ be a set contained in $Q_t^{j,\nu}$
such that $|S_t^{j,\nu}|/|Q_t^{j,\nu}|\ge \varepsilon$.
Then there exists a constant $C\in[1,\fz)$, depending on $\epsilon,\, s,\, p,\, q,\, \tau$,
such that, for all $a:=\{a_t^{j,\nu}\}_{j\in\zz_+, t\in I_j, 1\le \nu\le N_t^j}\subset\cc$,
\begin{eqnarray*}
\frac1{C}\|a\|_{f^{s,\tau}_{p, q}(M)}
&&\le\sup_{\gfz{k\in\zz}{ \az\in I_k}} \frac1{|Q^k_\az|^\tau}
\lf[\int_{Q_\az^k}
\bigg\{\sum_{j=k\vee0}^{\infty} \dz^{-jsq}\bigg[\sum_{t\in I_j}\sum_{\nu=1}^{N_t^j}
 |a_t^{j,\nu}|\chi_{S_t^{j,\nu}}(x)\bigg]^q\bigg\}^{p/q}\, d \mu(x)
\r]^{1/p}\\
&&\le \|a\|_{f^{s,\tau}_{p, q}(M)}
\end{eqnarray*}
and
\begin{eqnarray*}
\frac1{C}\|a\|_{\wz f^{s,\tau}_{p, q}(M)}
&&\le\sup_{\gfz{k\in\zz}{ \az\in I_k} } \frac1{|Q^k_\az|^\tau}
\lf[
\int_{Q_\az^k}
\bigg\{\sum_{j=k\vee0}^{\infty}\bigg[\sum_{t\in I_j}\sum_{\nu=1}^{N_t^j}
|Q_t^{j,\nu}|^{-s/d} |a_t^{j,\nu}|\chi_{S_t^{j,\nu}}(x)\bigg]^q\bigg\}^{p/q}\, d \mu(x)
\r]^{1/p}\\
&&\le \|a\|_{\wz f^{s,\tau}_{p, q}(M)}.
\end{eqnarray*}
\end{lem}

\begin{proof}
By similarity, we only consider the second formula regarding $\wz f^{s,\tau}_{p, q}(M)$.
Obviously, the right-hand side is dominated by the left-hand side.
To see the inverse, notice that, for all $j\in\zz_+$, $t\in I_j$,
$1\le \nu\le N_t^j$,
 $A\in(0,\fz)$ and $x\in M$,
$\chi_{{Q}_t^{j,\nu}}(x)\le \varepsilon^{-1/A}
[\mathcal{M}(\chi_{{S}_t^{j,\nu}})(x)]^{1/A}$.
We choose $A$ such that $0<A<\min\{p,q\}$.
Then, by the Fefferman-Stein vector-valued inequality (see \cite{s05,GLY-ms}), we see that
\begin{eqnarray*}
\|a\|_{\wz f^{s,\tau}_{p, q}(M)}\ls \sup_{\gfz{k\in\zz}{ \az\in I_k} } \frac1{|Q^k_\az|^\tau}
\lf[\int_{Q_\az^k}
\bigg\{ \sum_{j=k\vee0}^{\infty}\bigg[\sum_{t\in I_j}\sum_{\nu=1}^{N_t^j}
|Q_t^{j,\nu}|^{-s/d} |a_t^{j,\nu}|\chi_{S_t^{j,\nu}}(x)\bigg]^q\bigg\}^{p/q}\, d \mu(x)
\r]^{1/p}.
\end{eqnarray*}
This finishes the proof of Lemma \ref{lem7.11x}.
\end{proof}

Combining Lemmas \ref{lem7.10x} and \ref{lem7.11x}, we
can show that,
if $\tau\in[\frac1p,\fz)$, then the supremum $\sup_{{k\in\zz},\,{\az\in I_k}}$  and
${Q}_t^{j,\nu}$
in the definitions
of $f^{s,\tau}_{p, q}(M)$ and $\wz f^{s,\tau}_{p, q}(M)$
can be equivalently replaced by $\sup_{{k\in\zz_+},\,{\az\in I_k}}$
and $S_t^{j,\nu}$, respectively, the details being omitted.

For any sequence $a:=\{a_t^{j,\nu}\}_{j\in\zz_+, t\in I_j, 1\le \nu\le N_t^j}\subset\cc$,
define $$\wz G^{s,q}_{k,\az,u}(a)(x):= \lf\{\sum_{j=k}^{\infty}
\bigg[\sum_{t\in I_j}\sum_{\nu=1}^{N_t^j}
|Q_t^{j,\nu}|^{-s/d} |a_t^{j,\nu}|\chi_{{Q}_t^{j,\nu}}(x)
\chi_{{Q}_\az^{k,u}}(x)\bigg]^q\r\}^{1/q},\qquad x\in M,$$
$$\wz m^{s,q}_{k,\az,u}(a):= \inf\lf\{\lz>0:
\ |\{x\in {Q}_\az^{k,u}:\ \widetilde{G}^{s,q}_{k,\az,u}(a)(x)>\lz\}|
<|{Q}_\az^{k,u}|/4\r\}$$
and
$$\wz m^{s,q}(a)(x):=\sup_{k\in\zz_+, \az\in I_k, 1\le u\le N_\az^k} \widetilde{m}^{s,q}_{k,\az,u}(a)
\chi_{{Q}_\az^{k,u}}(x),\qquad x\in M.$$
In the definitions of $\wz G^{s,q}_{k,\az,u}$, $\wz m^{s,q}_{k,\az,u}$
and  $\wz m^{s,q}$,
if we replace
$|Q_t^{j,\nu}|^{-s/d}$ by $\delta^{-jsq}$, then we denote the corresponding
definitions, respectively, by
$G^{s,q}_{k,\az,u}$, $m^{s,q}_{k,\az,u}$ and $m^{s,q}$.

\begin{lem}\label{lem7.12x} Let $s\in\rr$ and  $q\in(0,\infty]$.
Then there exists a constant $C\in[1,\fz)$ such that, for all sequences $a:=\{a_t^{j,\nu}\}_{j\in\zz_+, t\in I_j, 1\le \nu\le N_t^j}\subset\cc$,
$$ \frac1C\|a\|_{f^{s,1/q}_{q,q}(M)}\le \|{m}^{s,q}(a)\|_{L^\fz(M)}\le C\|a\|_{f^{s,1/q}_{q,q}(M)}
$$
and
$$ \frac1C\|a\|_{\wz f^{s,1/q}_{q,q}(M)}
\le \|\widetilde{m}^{s,q}(a)\|_{L^\fz(M)}\le C\|a\|_{\wz f^{s,1/q}_{q,q}(M)}.$$
\end{lem}

\begin{proof}
Choose  $\lz\in(4^{1/q}\|a\|_{\wz f^{s,1/q}_{q,q}(M)},\fz)$. Then, by the Chebyshev inequality, we have
$$\lf|\lf\{x\in {Q}_\az^{k,u}:\ \widetilde{G}^{s,q}_{k,\az,u}(a)(x)>\lz\r\}\r|
\le \frac1{\lz^q}\int_{{Q}_\az^{k,u}} [\widetilde{G}^{s,q}_{k,\az,u}(a)(x)]^q\,d\mu(x)\le
\frac{|{Q}_\az^{k,u}|\|a\|^q_{\wz f^{s,1/q}_{q,q}(M)}}{\lz^q}<\frac14|{Q}_\az^{k,u}|.$$
Thus, $\|\widetilde{m}^{s,q}(a)\|_{L^\fz(M)}\ls\|a\|_{\wz  f^{s,1/q}_{q,q}(M)}$.
Conversely, for any $x\in M$,
we define a stopping time function
$$v(x):= \inf\lf\{v\in\zz_+:\ \lf\{\sum_{j=v}^{\infty}\bigg[\sum_{t\in I_j}\sum_{\nu=1}^{N_t^j}
|Q_t^{j,\nu}|^{-s/d} |a_t^{j,\nu}|\chi_{{Q}_t^{j,\nu}}(x)
\bigg]^q\r\}^{1/q}\le \wz m^{s,q}(a)(x)\r\},$$
and let
$$S_t^{j,\nu}:=\lf\{x\in {Q}_t^{j,\nu}:\ v(x)\le j\r\}
=\lf\{x\in {Q}_t^{j,\nu}:\
\widetilde{G}^{s,q}_{j,t,\nu}(a)(x)\le \wz m^{s,q}(a)(x)\r\}.$$
Then $|S_t^{j,\nu}|/|{ Q}_t^{j,\nu}|\ge 3/4$ and
$$\lf\{\sum_{j\in\zz_+}\bigg[\sum_{t\in I_j}\sum_{\nu=1}^{N_t^j}
|{Q}_t^{j,\nu}|^{-s/d} |a_t^{j,\nu}|\chi_{S_t^{j,\nu}}(x)
\bigg]^q\r\}^{1/q}\le \wz m^{s,q}(a)(x),\qquad x\in M.$$
By this and Lemma \ref{lem7.11x}, we obtain
$\|a\|_{\wz  f^{s,1/q}_{q,q}(M)}\ls \|\widetilde{m}^{s,q}(a)\|_{L^\fz(M)}$,
which completes the proof of Lemma \ref{lem7.12x}.
\end{proof}

\begin{proof}[Proof of Proposition \ref{prop7.9x}]
By similarity, we only prove $\wz f^{s,1/p}_{p, q}(M)=\wz f^{s,1/q}_{q,q}(M)$.
If $p\ge q$, then $\|a\|_{\wz  f^{s,1/q}_{q,q}(M)}\ls \|a\|_{\wz f^{s,1/p}_{p, q}(M)}$
follows immediately from H\"older's inequality.
Conversely, let $S_t^{j,\nu}$ be as in the proof of Lemma \ref{lem7.12x}.
Then, by Lemmas \ref{lem7.10x}, \ref{lem7.11x} and \ref{lem7.12x}, we know that
\begin{eqnarray*}
\|a\|_{\wz f^{s,1/p}_{p, q}(M)}
\ls \lf\|\bigg\{ \sum_{j=0}^{\infty}\bigg[\sum_{t\in I_j}\sum_{\nu=1}^{N_t^j}
|Q_t^{j,\nu}|^{-s/d} |a_t^{j,\nu}|
\chi_{S_t^{j,\nu}}\bigg]^q \bigg\}^{1/q}\r\|_{L^\fz(M)}\ls \|a\|_{\wz f^{s,1/q}_{q,q}(M)}.
\end{eqnarray*}
Now we consider $p<q$. By H\"older's inequality, we have
$\|a\|_{\wz f^{s,1/p}_{p, q}(M)}\ls \|a\|_{\wz f^{s,1/q}_{q,q}(M)}.$
To prove the converse inequality, we only need to repeat the proof of Lemma \ref{lem7.12x}
involving the Chebyshev inequality. The only difference is that we need to
replace $q$ therein  by $p$ now. This finishes the proof of Proposition \ref{prop7.9x}.
\end{proof}

%%%%%%%%%%%%%%%%%%%%%%%%%%%%%%%%%%%%%%%%%%%%%%%%%%%%%%%%%%%%%%%%%%%%%

\section{Further remarks}\label{sec-8}

\hskip\parindent
In this section, we first prove that,
on the Euclidean spaces, when $L$ is the Laplacian operator, the Besov-type
and the Triebel--Lizorkin-type spaces coincide with those
introduced in \cite{YSY}.
Furthermore, when $\tau=0$ and $\bz_0=2$ in ${\bf (UE)}$ and ${\bf (HE)}$,
we show that the Besov and the Triebel--Lizorkin spaces
on RD-spaces satisfying \eqref{non-collapsing}
defined in \cite{HMY2} coincide with
 the Besov and the Triebel--Lizorkin spaces  in \cite{KP},
 which gives a positive answer to a question
presented in \cite{KP}.

\subsection{Go back to Euclidean spaces}\label{sec-8.1}

\hskip\parindent
Consider now $M=\rn$, $\rho$ is the Euclidean distance,
and the measure $\mu$ is the $n$-dimensional Lebesgue measure.
In this case, the  Christ cubes
turn out to be the classical  dyadic cubes
$$\cq:= \lf\{ Q_m^j:=2^{-j}([0,1)^n+m):\, j\in\zz,\, m\in\zz^n\r\},$$
and the constant $\delta$ in Lemma \ref{lem3.1x} is exactly $1/2$.
Assume that $L$ is the Laplacian operator $\Delta:=-\sum^n_{j=1}\frac{\partial^2}{\partial x_j^2}$. It is a nonnegative self-adjoint operator
on $L^2(\rn)$. The associated heat semigroup $\{e^{-tL}\}_{t>0}$
is a set of integral operators whose heat kernels
$\{p_t\}_{t>0}$ are the Gauss-Weierstrass kernel:
$$p_t(x,y)=\frac{1}{(4\pi t)^{n/2}}\exp\lf(-\frac{|x-y|^2}{4t}\r),\,\qquad x,\,y\in\rn,\hs t\in(0,\fz).$$
Then $\{p_t\}_{t>0}$
satisfies the  conditions
\eqref{GUB} with $\bz_0=2$, \eqref{HC} with $\az_0=1$, and
\eqref{Markov}.
Let $\cs(\rn)$ denote the class of all Schwartz functions on $\rn$.
By the Newton-Leibniz formula and the mathematical induction, an easy calculation
leads to  that
the test function space  $\cd(\rn)$ defined in the beginning of Section \ref{sec-3} is exactly the Schwartz class $\cs(\rn)$.
As a consequence, the distribution space $\cd'(\rn)$ coincides with the  space $\cs'(\rn)$ of Schwartz  distributions.
Thus, the previous discussed Besov-type
and Triebel--Lizorkin-type spaces on $\rn$ read as follows.

\begin{defn}\label{def8.1x}
Let $\Phi_0, \Phi\in C^\infty(\rr_+)$ such that
\begin{equation*}
\supp \Phi_0\subset [0, 2],\quad
\Phi_0^{(2\nu+1)}(0)=0\,\ \textup{for \ all }\ \nu\in\nn,\,\, \qquad
|\Phi_0(\lz)|\ge c \,\,\textup{for}\,\,\ \lz\in[0, 2^{3/4}],
\end{equation*}
and
\begin{equation*}
\supp \Phi\subset [2^{-1}, 2],\quad \, \quad
|\Phi(\lz)|\ge c \,\,\textup{for}\,\,\,\lz\in[2^{-3/4}, 2^{3/4}],
\end{equation*}
where $c$ is a positive constant.
Let $\Phi_j(\cdot):=\Phi(2^{-j}\cdot)$ for all $j\in\nn$.
Let $\tau\in[0,\infty),\,s\in\rr$ and $q\in(0,\infty]$.
For $p\in(0,\infty]$, the \emph{Besov-type space} $B_{p,q}^{s,\tau}(\rn)$ is
defined to be the set of all $f\in\cd'(\rn)$  (see Section \ref{sec-2.2}) such that
\begin{eqnarray*}
\|f\|_{B_{p,q}^{s,\tau}(\rn)}
:=\sup_{k\in\zz, m\in\zz^n}
\frac1{|Q_m^k|^\tau}
\bigg\{\sum_{j=k\vee0}^\infty
\bigg[\int_{Q_m^k} 2^{jsp} |\Phi_j(\sqrt {\Delta})f(x)|^p\,dx\bigg]^{q/p}\bigg\}^{1/q}
<\infty.
\end{eqnarray*}
For $p\in(0,\infty)$,
the \emph{Triebel--Lizorkin-type space} $F_{p,q}^{s,\tau}(\rn)$ is
defined to be the set of all
$f\in\cd'(\rn)$ (see Section \ref{sec-2.2}) such that
\begin{eqnarray*}
\|f\|_{F_{p,q}^{s,\tau}(\rn)}
:=\sup_{k\in\zz, m\in\zz^n}
\frac1{|Q_m^k|^\tau}
\bigg\{\int_{Q_m^k}
\bigg[\sum_{j=k\vee0}^\infty 2^{jsq} |\Phi_j(\sqrt {\Delta})f(x)|^q\bigg]^{p/q}\,dx\bigg\}^{1/p}
<\infty.
\end{eqnarray*}
\end{defn}

Recall that, in \cite{YSY}, the Besov-type
and the Triebel--Lizorkin-type spaces were introduced as follows.

\begin{defn}\label{def8.2x}
Let $\phi_0$, $\phi\in\cs(\rn)$
such that
\begin{eqnarray*}%\label{eq:8.1x}
\supp \widehat \phi_0 \subset \{\xi\in\rn:\,|\xi|\le 2\},\qquad
|\widehat \phi_0(\xi)|\ge c \quad\textup{if}\quad |\xi|\le 5/3,
\end{eqnarray*}
and
\begin{eqnarray*}%\label{eq:8.2x}
\supp \widehat \phi \subset \{\xi\in\rn:\,1/2\le|\xi|\le 2\},\qquad
|\widehat \phi(\xi)|\ge c \quad\textup{if}\quad 3/5\le |\xi|\le 5/3,
\end{eqnarray*}
where $c$ is a positive constant.
For $j\in\nn$, define
$\phi_j(\cdot):= 2^{jn}\phi(2^j\cdot).$
Let $\tau\in[0,\infty),\,s\in\rr$ and $q\in(0,\infty]$.
For $p\in(0,\infty]$, the
\emph{Besov-type space} $\mathcal B_{p,q}^{s,\tau}(\rn)$ is defined to be the set of all $f\in\cs'(\rn)$ such that
\begin{eqnarray*}
\|f\|_{\mathcal B_{p,q}^{s,\tau}(\rn)}
:=\sup_{k\in\zz, m\in\zz^n}
\frac1{|Q_m^k|^\tau}
\bigg\{\sum_{j=k\vee0}^\infty
\bigg[\int_{Q_m^k} 2^{jsp} |\phi_j\ast f(x)|^p\,dx\bigg]^{q/p}\bigg\}^{1/q}
<\infty.
\end{eqnarray*}
For $p\in(0,\infty)$, the \emph{Triebel--Lizorkin-type space}
$\mathcal F_{p,q}^{s,\tau}(\rn)$ is defined to be the set of
all $f\in\cs'(\rn)$ such that
\begin{eqnarray*}
\|f\|_{\mathcal F_{p,q}^{s,\tau}(\rn)}
:=\sup_{k\in\zz, m\in\zz^n}
\frac1{|Q_m^k|^\tau}
\bigg\{\int_{Q_m^k}
\bigg[\sum_{j=k\vee0}^\infty 2^{jsq} |\phi_j\ast f(x)|^q\bigg]^{p/q}\,dx\bigg\}^{1/p}
<\infty.
\end{eqnarray*}
\end{defn}

\begin{thm}\label{thm8.3x}
Let all the notation be as in Definitions \ref{def8.1x} and \ref{def8.2x}. Then
$B_{p,q}^{s,\tau}(\rn) =\mathcal B_{p,q}^{s,\tau}(\rn)$
and $
{F_{p,q}^{s,\tau}}(\rn) ={\mathcal F_{p,q}^{s,\tau}(\rn)}
$
with equivalent (quasi-)norms.
\end{thm}

\begin{proof}
Let $m>s/2$ and $h_0$, $h$  and $h_j$ be as in \eqref{eq:6.15x}.
That is, $h_0(\lz):= e^{-\lz^2}$, $h(\lz):=
\lz^{2m} e^{-\lz^2}$ and $
h_j(\lz):= h(2^{j}\lz)=
(2^{j}\lz)^{2m} e^{-2^{2j}\lz^2}$ for all $j\in\nn$ and $\lz\in(0,\fz).$
By Theorem \ref{thm6.7x}, we know that
$f\in B_{p,q}^{s,\tau}(\rn)$ if and only if
$f\in \cs'(\rn)$ and
\begin{eqnarray*}
\|f\|_{_h\!B^{s,\tau}_{p, q}(\rn)}
&&:=
\|\{2^{-js}
h_j(\sqrt \Delta) f\}_{j\in\zz_+}\|_{\ell^q(L^p_\tau)}<\fz.
\end{eqnarray*}
Moreover, $\|\cdot\|_{_h\!B^{s,\tau}_{p, q}(\rn)}\sim \|\cdot\|_{B^{s,\tau}_{p, q}(\rn)}$.
On the other hand, let $H_0(x):=e^{-|x|^2}$ and $H(x):=(|\cdot|^{2m}e^{-|\cdot|^2})^\vee(x)$
for all $x\in\rn$. Then it is well known that
$h_0(\sqrt{\Delta})f= e^{-\Delta}f=C_0H_0\ast f $ and
$h(\sqrt{\Delta})f=\Delta^{m} e^{-\Delta}f=C_1 H\ast f$
for all $f\in \cs'(\rn)$ and some positive constants $C_0$ and $C_1$.
Notice that $\widehat{H}_0$ and $\widehat{H}$ are positive on $B(0,2)$ and $B(0,2)\setminus B(0,1/2)$, respectively,
and $(\partial^\az\widehat{H})(0)=0$ for all $|\az|\le 2m$.
Thus, by the local means characterization of $\mathcal B^{s,\tau}_{p, q}(\rn)$
(see \cite{lsuyy}), we know that
$\mathcal B_{p,q}^{s,\tau}(\rn)=\, _h\!B^{s,\tau}_{p, q}(\rn)$ with equivalent (quasi-)norms. Thus,
$B_{p,q}^{s,\tau}(\rn) =\mathcal B_{p,q}^{s,\tau}(\rn)$ with equivalent (quasi-)norms.

The proof for the Triebel--Lizorkin-type case is similar, the details being omitted.
This finishes the proof of Theorem \ref{thm8.3x}.
\end{proof}

\subsection{Go back to RD-spaces}\label{sec-8.2}

\hskip\parindent
A systemic treatment for the theory of
(in)homogeneous Besov and Triebel--Lizorkin
spaces on RD-spaces was due to the work  \cite{HMY2}.
Here, in this section, we compare the
Besov and the Triebel--Lizorkin
spaces introduced in \cite{HMY2} with those in \cite{KP}.
Throughout this section, for all $x$, $y\in M$ and $\dz>0$,
let $V_\dz(x) := \mu(B(x, \dz))$ and $V(x, y) := \mu(B(x, d(x, y)))$.

The definitions of Besov and Triebel--Lizorkin
spaces in \cite{HMY2} rely on the existence of the following  {\em approximation
of the identity} on RD-spaces; see  \cite[Definition~2.2]{HMY2}
and \cite[Theorem 2.6]{HMY2}.

\begin{defn}\label{def8.4x}
Let $\ez_1\in(0, 1]$ and $\ez_2,\,\ez_3\in(0,\fz)$.
A sequence $\{S_k\}_{k\in\zz}$ of bounded
linear  integral operators on $L^2(M)$ is called an \emph{approximation
of the identity  of order $(\ez_1, \ez_2, \ez_3)$}
(in short, $(\ez_1, \ez_2, \ez_3)$ -$\aoti$), if there exists a
positive constant $C$ such that, for all $k\in\zz$, $x$, $x'$,
$y$ and $y'\in M$, $S_k(x, y)$, the integral kernel of $S_k$, is
a measurable function from $M\times M$ into $\cc$ satisfying
\begin{enumerate}
\item[ \rm (i)] $|S_k(x, y)|
\le C\frac{1}{V_{2^{-k}}(x)+V_{2^{-k}}(y)+V(x, y)}
[\frac{2^{-k}}{2^{-k}+\rho(x, y)}]^{\ez_2};$

\item[ \rm (ii)] for
$\rho(x, x')\le[2^{-k}+\rho(x, y)]/2$,\\
$|S_k(x, y)-S_k(x', y)| \le C\lf[\frac{\rho(x,
x')}{2^{-k}+\rho(x, y)}\r]^{\ez_1}
\frac{1}{V_{2^{-k}}(x)+V_{2^{-k}}(y)+V(x, y)}
\lf[\frac{2^{-k}}{2^{-k}+\rho(x, y)}\r]^{\ez_2};$

\item[ \rm (iii)] $S_k$ satisfies {\rm (ii)} with $x$ and $y$ interchanged;

\item[ \rm (iv)] for $\rho(x, x')\le
[2^{-k}+\rho(x, y)]/3$ and $\rho(y, y')\le [2^{-k}+\rho(x, y)]/3$,
\begin{eqnarray*}
&&|[S_k(x, y)-S_k(x, y')]-[S_k(x', y)-S_k(x', y')]|\\
&&\quad\le C\lf[\frac{\rho(x, x')}{2^{-k}+\rho(x, y)}\r]^{\ez_1}
\lf[\frac{\rho(y, y')}{2^{-k}+\rho(x, y)}\r]^{\ez_1}\frac{1}{V_{2^{-k}}(x)+V_{2^{-k}}(y)+V(x, y)}
\lf[\frac{2^{-k}}{2^{-k}+\rho(x, y)}\r]^{\ez_3};
\end{eqnarray*}

\item[ \rm (v)] $\int_M S_k(x, w)\,
d\mu(w)=1=\int_M S_k(w, y) \, d\mu(w)$.
\end{enumerate}
\end{defn}

The following version of test functions on
RD-spaces was originally introduced in \cite{HMY1} (see also \cite{HMY2}).

\begin{defn}\label{def8.5x}
Let $x_1\in M$, $r\in(0, \fz)$, $\bz\in(0, 1]$ and $\gz\in(0,
\fz)$. A function $\vz$ on $M$ is said to belong to the \emph{space
$\cg(x_1, r, \bz, \gz)$ of test functions},
if there exists a positive constant $C$ such
that, for all $x,\,y\in M$,
\begin{enumerate}
\item[ \rm (i)] $|\vz(x)|\le C\frac{1}{V_r(x_1)+V_r(x)+V(x_1, x)}
[\frac{r}{r+\rho(x_1, x)}]^\gz$;

\item[ \rm (ii)] $|\vz(x)-\vz(y)| \le C[\frac{\rho(x, y)}{r+\rho(x_1, x)}]^\bz
\frac{1}{V_r(x_1)+V_r(x)+V(x_1, x)}[\frac{r}{r+\rho(x_1, x)}]^\gz$ when $\rho(x, y)\le[r+\rho(x_1, x)]/2$.
\end{enumerate}
If $\vz\in\cg(x_1, r, \bz,\gz)$,
then its \emph{norm} is defined by $\|\vz\|_{\cg(x_1,\, r,\, \bz,\,\gz)}  :=
\inf\{C:\, \rm {(i)} \mbox{ and } {\rm (ii)} \mbox{ hold}\}$.
\end{defn}

Fix $x_1\in M$ and let $\cg(\bz, \gz) : = \cg(x_1, 1, \bz, \gz)$.
For any
$x_2\in M$ and $r\in(0,\fz)$,  it is easy to see that $\cg(x_2, r, \bz, \gz)=\cg(\bz, \gz)$ with equivalent norms. Also, the space  $\cg(\bz, \gz)$ is a Banach space.
Let $\ez\in(0, 1]$ and $\bz$, $\gz\in(0, \ez]$.
Denote by $\cg_0^\ez(\bz, \gz)$ the
completion of  $\cg(\ez, \ez)$ in $\cg(\bz, \gz)$. Then
$\vz\in\cg_0^\ez(\bz, \gz)$ if and
only if $\vz\in\cg(\bz, \gz)$ and there exist
functions
$\{\phi_j\}_{j\in\nn}$
converging to $\vz$ in $\cg(\ez, \ez)$. For any
$\vz\in\cg_0^\ez(\bz, \gz)$,  define
$\|\vz\|_{\cg_0^\ez(\bz,
\gz)} := \|\vz\|_{\cg(\bz, \gz)}.$
For the above chosen $\{\phi_j\}_{j\in\nn}$, we have
$\|\vz\|_{\cg_0^\ez(\bz, \gz)}=\lim_{j\to\fz}\|\phi_j\|_{\cg(\bz, \gz)}.$
Notice that  $\cg_0^\ez(\bz, \gz)$ is also a Banach space.
Denote by $(\cg_0^\ez(\bz, \gz))'$ the \emph{set of all bounded linear functionals} on
$\cg_0^\ez(\bz, \gz)$. Define $\laz
f, \vz\raz$ to be the natural pairing of elements $f\in
(\cg_0^\ez(\bz, \gz))'$ and $\vz\in\cg_0^\ez(\bz, \gz)$.

We now recall the  definitions of Besov and Triebel--Lizorkin
spaces on RD-spaces in \cite[Definition 5.29]{HMY2}.
In what follows, for all $\ez\in(0,1)$
and $|s|<\ez$, let
$p(s,\ez) := \max\{d/(d+\ez), d/(d+\ez+s)\}.$
For all $g\in L^1_\loc(M)$ and Christ dyadic cubes $Q$,
we write $m_Q(g):=\frac1{|Q|}\int_{Q}g(y)\,d\mu(y)$.

\begin{defn}\label{def8.6x}
Let
$\ez_1\in(0, 1]$, $\ez_2,\,\ez_3\in(0,\fz)$,
$\ez\in(0, \ez_1\wg\ez_2)$, $\bz,\,\gz\in(0,\ez)$,
$|s|<\ez$ and $\{S_k\}_{k\in\zz}$ be an $(\ez_1, \ez_2, \ez_3)$-$\aoti$.
Define $D_0:=S_0$, and $D_k:=S_k-S_{k-1}$ for
all $k\in\nn$. Let $\{Q^{0,v}_\tau\}_{\tau\in I_0, v\in\{1,\ldots,N^0_\tau\}}$
be dyadic cubes as in Remark \ref{rem7.2x}.

{\rm(i)} Let $p\in(p(s,\ez),\fz]$ and $q\in(0,\fz]$. The {\it Besov space
$\mathcal B^s_{p, q}(M)$} is defined to be the set of all
$f\in(\cg^\ez_0(\bz,\gz))'$, for some $\bz$, $\gz$ satisfying
\begin{equation}\label{eq:8.1x}
\max\lf\{s,0,-s+d(1/p-1)_+\r\}<\bz<\ez,\quad d(1/p-1)_+<\gz<\ez
\end{equation}
such that
$$\|f\|_{ \mathcal B^s_{p, q}(M)}
:=\lf\{\sum_{\tau\in I_0}\sum_{v=1}^{N_\tau^0}|Q^{0,v}_\tau|[m_{Q^{0,v}_\tau}(|D_0f|)]^p\r\}^{1/p}
+\bigg[\sum^\fz_{k=1}2^{ksq}\|D_kf\|_{L^p(M)}^q\bigg]^{1/q}<\fz$$
with the usual modifications made when $p=\fz$ or $q=\fz$.

{\rm(ii)} Let $p\in(p(s,\ez),\fz)$ and $q\in(p(s,\ez),\fz].$ The {\it Triebel--Lizorkin space
$ \mathcal F^s_{p, q}(M)$} is defined to be the set of all
$f\in(\cg^\ez_0(\bz,\gz))'$ for some $\bz$, $\gz$ satisfying \eqref{eq:8.1x} such that
$$\|f\|_{ \mathcal F^s_{p, q}(M)}:=\lf\{\sum_{\tau\in I_0}\sum_{v=1}^{N_\tau^0}|Q^{0,v}_\tau|[m_{Q^{0,v}_\tau}(|D_0f|)]^p\r\}^{1/p}
+\bigg\|\bigg[\sum^\fz_{k=1}
2^{ksq}|D_kf|^q\bigg]^{1/q}\bigg\|_{L^p(M)}<\fz$$ with the usual
modification made when $q=\fz$.
\end{defn}

For all $\ez_1\in(0, 1]$, $\ez_2,\,\ez_3\in(0,\fz)$,
$\ez\in(0, \ez_1\wg\ez_2)$, $\bz,\,\gz\in(0,\ez)$ and $|s|<\ez$,
it was proved in \cite[Proposition 5.32]{HMY2} that,
if $p\in[1,\fz]$ and $q\in(0,\fz]$, then, for all $f\in(\cg^\ez_0(\bz,\gz))'$,
$$\|f\|_{ \mathcal B^s_{p, q}(M)}
\sim\bigg[\sum^\fz_{k=0}2^{ksq}\|D_kf\|_{L^p(M)}^q\bigg]^{1/q}$$
and, if $p\in[1,\fz)$ and $q\in(p(s,\ez),\fz]$, then, for all $f\in(\cg^\ez_0(\bz,\gz))'$,
$$\|f\|_{ \mathcal F^s_{p, q}(M)}\sim \bigg\|\bigg[\sum^\fz_{k=0}
2^{ksq}|D_kf|^q\bigg]^{1/q}\bigg\|_{L^p(M)}$$
with implicit positive constants independent of $f$.

Applying the discrete  Calder\'on
reproducing formula in Theorem \ref{thm-CRF}, by
an argument similar to that used in the proof of \cite[Proposition 5.32]{HMY2},
we obtain the following (quasi-)norm equivalences between two types of function spaces,
which answers a question presented in \cite{KP}.

\begin{thm} \label{thm8.7x}
 Let  $\ez_1\in(0, 1]$, $\ez_2,\,\ez_3\in(0,\fz)$,
 $\ez\in(0, \az_0\wg \ez_1\wg\ez_2)$, $\bz,\,\gz\in(0,\ez)$ and
$|s|<\ez$, where  $\az_0$ is as in ${\bf (HE)}$.
\begin{enumerate}
\item[{\rm(i)}] If $p\in(p(s,\ez),\fz]$, $q\in(0,\fz]$ and $\bz$, $\gz$ satisfy \eqref{eq:8.1x},
then
$\|f\|_{ \mathcal  B^s_{p, q}(M)}\sim \|f\|_{ B^s_{p, q}(M)}$ for all $f\in (\cg^\ez_0(\bz,\gz))'\cap\cd'(M)$.

\item[{\rm(ii)}]  If $p\in(p(s,\ez),\fz)$, $q\in(p(s,\ez),\fz]$ and $\bz$, $\gz$ satisfy \eqref{eq:8.1x},
then
$\|f\|_{ \mathcal  F^s_{p, q}(M)}\sim \|f\|_{ F^s_{p, q}(M)}$ for all $f\in (\cg^\ez_0(\bz,\gz))'\cap\cd'(M)$.
\end{enumerate}
Here, in (i) and (ii), the implicit  equivalent positive
constants are independent of $f$.
\end{thm}

\begin{proof}
First, we  show (i).  Let $\{S_k\}_{k\in\zz}$ be an
$(\ez_1,\ez_2,\ez_3)$-${\rm ATI}$. Set $D_0:=S_0$, and $D_k:=S_k-S_{k-1}$ for $k\in\nn$. It was proved in \cite[Lemma~3.2,\,(3.2)]{HMY2} that,  for any $\ez_1'\in(0, \ez_1\wedge\ez_2)$ and all $x,\,y\in M$,
\begin{eqnarray}\label{eq:8.2x}
|{D}_kD_j(x,y)|\ls 2^{-|j-k|\ez_1'}
\frac{1}{V_{2^{-(k\wedge j)}}(x)+V_{2^{-(k\wedge j)}}(y)+V(x, y)}
\frac{2^{-(k\wedge j)\ez_2}}{[2^{-(k\wedge j)}+\rho(x, y)]^{\ez_2}}.
\end{eqnarray}

For $\bz,\,\gz\in(0,\epsilon)$, we recall that the discrete inhomogeneous Calder\'on reproducing formula in \cite[Theorems 4.14 and 4.16]{HMY2}:
for all $f\in (\cg^\ez_0(\bz,\gz))'$,
\begin{eqnarray}\label{eq:8.3x}
f(\cdot)= \sum_{\tau\in I_0}\sum_{v=1}^{N_\tau^0}\int_{Q^{0,v}_\tau} \wz D_0(\cdot,y)\,d\mu(y)
D^{0,v}_{\tau,1}f +\sum_{k=1}^\infty\sum_{\tau\in I_k}\sum_{v=1}^{N_\tau^k}
|Q^{k,v}_\tau|\widetilde D_k(\cdot, y^{k,v}_\tau) D_k f(y^{k,v}_\tau)
\end{eqnarray}
converges in $(\cg^\ez_0(\bz,\gz))'$, where $y^{k,v}_\tau$
 is an arbitrary point in $Q^{k,v}_\tau$,  $D^{0,v}_{\tau,1}$ denotes the integral
operator with kernel $D^{0,v}_{\tau,1}(z):=
\frac1{|Q^{0,v}_\tau|}\int_{Q^{0,v}_\tau}D_0(z,u)\,d\mu(u)$,
and the kernels of the operators
$\{\widetilde D_k\}_{k\in\zz_+}$ satisfy the conditions (i) and (iii) of Definition \ref{def8.4x} with $\ez_1$
and $\ez_2$ replaced by $\ez'\in(\ez, \ez_1\wedge\ez_2)$, and $\int_M \widetilde D_k(x,y)\,d\mu(y)=0$ when $k\in\nn$ and $=1$ when $k=0$.

Let $\Phi_0\in C_c^\infty(\rr_+)$ be such that $\supp \Phi_0\subset [0,2^{\bz_0/2}]$
and $\Phi_0\equiv 1$ on $[0,1]$. Define
$\Phi(\lz):= \Phi_0(\lz)-\Phi_0(2^{\bz_0/2}\lz)$ for $\lz\in\rr_+$.
For $j\in\nn$, define $\Phi_j(\cdot):= \Phi(2^{-j\bz_0/2}\cdot)$.
By (i) and (ii) of Proposition \ref{prop2.10x}, we see that $\Phi_j(L)(x,\cdot)\in \cg_0^\ez(\az_0,\gz)\subset \cg_0^\ez(\bz,\gz)$. Moreover, by
Proposition \ref{prop2.10x}(iii), we see that
$\int_M \Phi_j(\sqrt L)(x,y)\,d\mu(y)=0$ for $j\in\nn$ and $=1$ for $j=0$.
From these and the proof of \cite[Lemma~3.2,\,(3.2)]{HMY2},
it follows that an orthogonal estimate similar to \eqref{eq:8.2x} holds true,
namely, for any given $\ez_1''\in(0, \az_0\wedge\ez_1\wedge\ez_2)$,
\begin{eqnarray}\label{eq:8.4x}
|\Phi_j(\sqrt L){\widetilde D_k}(x,y)|\ls 2^{-|j-k|\ez_1''}
\frac{1}{V_{2^{-(k\wedge j)}}(x)+V_{2^{-(k\wedge j)}}(y)+V(x, y)}
\frac{1}{[1+2^{k\wedge j}\rho(x, y)]^{\ez_2}}
\end{eqnarray}
holds true
for all $j,\,k\in\zz_+$ and $x,\,y\in M$.
Further, applying \eqref{eq:8.3x}, we know that, for all $j\in\zz_+$ and $x\in M$,
\begin{eqnarray}\label{eq:8.5x}
&&|\Phi_j(\sqrt L) f(x)|\notag\\
&&\quad\ls \sum_{\tau\in I_0}\sum_{v=1}^{N_\tau^0}
\int_{Q^{0,v}_\tau}
\frac{2^{-j\ez_1''}}{V_{1}(x)+V_{1}(y)+V(x, y)}
\frac{1}{[1+\rho(x, y)]^{\ez_2}}\,d\mu(y) |D^{0,v}_{\tau,1}f| \notag \\
&&\quad\quad+\sum_{k=1}^\infty\sum_{\tau\in I_k}\sum_{v=1}^{N_\tau^k}
|Q^{k,v}_\tau|
\frac{2^{-|k-j|\ez_1''}}{V_{2^{-(k\wedge j)}}(x)+V_{2^{-(k\wedge j)}}(y^{k,v}_\tau)+V(x, y^{k,v}_\tau)}
\frac{|D_k f(y^{k,v}_\tau)|}{[1+2^{k\wedge j}\rho(x, y^{k,v}_\tau)]^{\ez_2}}.\qquad
\end{eqnarray}

Since $\ez,\, \ez_1''\in(0, \az_0\wg\ez_1\wg\ez_2)$,
$|s|<\ez$ and $\bz,\,\gz\in(0,\ez)$, we may choose
$\ez_1''$ such that $|s|<\ez_1''$, $p(s,\ez_1'')<p$
and $\bz,\,\gz\in(0,\ez_1'')$ satisfy \eqref{eq:8.1x}
with $\ez$ replaced by $\ez_1''$. On the other hand,
since $V_{1}(x)+V_{1}(y)+V(x, y)\sim V_{1}(x)+V_{1}(y^{0,v}_\tau)+V(x, y^{0,v}_\tau)$
and $1+\rho(x, y)\sim 1+\rho(x,y^{0,v}_\tau)$ for any $y\in Q^{0,v}_\tau$,
applying \cite[Lemma 5.3]{HMY2}, we see that, for all $\ell\in\zz_+$ and $x\in M$,
\begin{eqnarray}\label{eq:8.6x}
&&|2^{js}\Phi_j(\sqrt L) f(x)|\notag\\
&&\hs\ls 2^{-j(\ez_1''-s)} \Bigg[\cm\Bigg(\sum_{\tau\in I_0}\sum_{v=1}^{N_\tau^0}|D^{0,v}_{\tau,1}f|^r\chi_{Q^{0,v}_\tau}\Bigg)(x)\Bigg]^{1/r}\notag \\
&&\hs\quad+\sum_{k=1}^\infty2^{-|k-j|\ez_1''}2^{(j-k)s}2^{[(k\wedge j)-k]d(1-1/r)}\Bigg[\cm\Bigg(\sum_{\tau\in I_k}\sum_{v=1}^{N_\tau^k}2^{ksr}|D_k f(y^{k,v}_\tau)|^r\chi_{Q^{k,v}_\tau}\Bigg)(x)\Bigg]^{1/r},\qquad
\end{eqnarray}
where we chose $r\in (p(s,\ez_1''),p)$ if $p\le 1$ or $r=1$ if $p\in(1,\fz]$.
Therefore, letting
$$\sigma\in(0,\ez_1''-\max\{s,0,-s+d(1/r-1)\}),$$
by   Lemma \ref{lem6.3x} and the  boundedness of $\cm$ on $L^{p/r}(M)$,
we see that
\begin{eqnarray}\label{eq:8.7x}
\|f\|_{B_{p,q}^s(M)}&&\ls \Bigg\{\sum_{j=0}^\infty 2^{-j(\ez_1''-s-\sigma)q} \Bigg\|\Bigg[\cm\Bigg(\sum_{\tau\in I_0}\sum_{v=1}^{N_\tau^0}|D^{0,v}_{\tau,1}f|^r\chi_{Q^{0,v}_\tau}\Bigg)\Bigg]^{1/r}\Bigg
\|_{L^p(M)}^q\notag\\
&&\quad+\sum_{j=0}^\infty\sum_{k=1}^\infty2^{-|k-j|(\ez_1''-\sigma)q}
2^{(j-k)sq}2^{[(k\wedge j)-k]d(1-1/r)q}\notag\\
&&\quad\times\Bigg\|\Bigg[\cm\Bigg(\sum_{\tau\in I_k}\sum_{v=1}^{N_\tau^k}2^{ksr}|D_k f(y^{k,v}_\tau)|^r\chi_{Q^{k,v}_\tau}\Bigg)\Bigg]^{1/r}\Bigg\|_{L^p(M)}^q\Bigg\}^{1/q}
\ls\|f\|_{\mathcal B_{p,q}^s(M)},
\end{eqnarray}
where the last step follows from the frame characterization of $\mathcal B_{p,q}^s(M)$ in \cite[Theorem 7.4]{HMY2}.

 Now we show the converse part. For any $f\in  B_{p,q}^s(M)$,  we use  \eqref{DCRF} to write every $D_jf$ as
\begin{eqnarray*}
D_jf(\cdot)=  \sum_{k=0}^\infty\sum_{\tau\in I_k} \sum_{\nu=1}^{N_\tau^{k}}
|Q_{\tau}^{k,\nu}| (\Phi_k(\sqrt L)f )(\xi_\tau^{k,\nu})\,
D_j\Psi_k(\xi_\tau^{k,\nu},\cdot),
\end{eqnarray*}
where $\Phi_k$ and $\Psi_k$ are as in \eqref{DCRF}.
Notice that $D_j \Psi_k(\sqrt L)(x,y)$ has the same estimate as in \eqref{eq:8.4x}.  Thus, repeating the previous proofs of \eqref{eq:8.5x} through \eqref{eq:8.7x},
but with the roles of $D_j$ and $\Phi_j(\sqrt L)$ exchanged, and applying
Theorem \ref{thm7.5x}, we obtain
$\|f\|_{\mathcal B_{p,q}^s(M)}\ls \|f\|_{B_{p,q}^s(M)}$. This finishes the proof of (i).

To prove (ii), we choose
$\ez_1''$ such that $|s|<\ez_1''$, $p(s,\ez_1'')<\min\{p,q\}$
and $\bz,\,\gz\in(0,\ez_1'')$ satisfying \eqref{eq:8.1x}
with $\ez$ replaced by $\ez_1''$.
In this case, \eqref{eq:8.5x} and \eqref{eq:8.6x} keep valid,
but this time we choose
$r\in (p(s,\ez_1''),\min\{p,q\})$ if $\min\{p,q\} \le 1$ or $r=1$ if $\min\{p,q\}\in(1,\fz]$.
Then, by  Lemma \ref{lem6.3x}, the frame characterization of $\mathcal F_{p,q}^s(M)$ in \cite[Theorem 7.4]{HMY2}, and repeating the above proof for Besov spaces,
with the boundedness of $\cm$ on $L^{p/r}(M)$ replaced by
the Fefferman-Stein vector-valued maximal
inequality (see \cite{GLY-ms}), we see that
\begin{eqnarray*}
\|f\|_{F_{p,q}^s(M)}
&&\ls\Bigg\|\Bigg\{\sum_{j=0}^\fz
2^{-j(\ez_1''-s)q}
\Bigg[\cm\Bigg(\sum_{\tau\in I_0}\sum_{v=1}^{N_\tau^0}
|D^{0,v}_{\tau,1}f|^r\chi_{Q^{0,v}_\tau}\Bigg)\Bigg]^{q/r}
 \notag \\
&&\quad +\sum_{k=1}^\infty2^{-|k-j|(\ez_1''-\sigma)q}
2^{(j-k)sq}2^{[(k\wedge j)-k]d(1-1/r)q}\\
&&\quad\times\Bigg[\cm\Bigg
(\sum_{\tau\in I_k}\sum_{v=1}^{N_\tau^k}2^{ksr}|D_k f(y^{k,v}_\tau)
|^r\chi_{Q^{k,v}_\tau}\Bigg)\Bigg]^{q/r}
 \Bigg\}^{1/q}\Bigg\|_{L^p(M)}
\ls \|f\|_{\mathcal F_{p,q}^s(M)}.
\end{eqnarray*}
The converse of this inequality follows from an argument  similar to that used in
the proof of
$\|f\|_{\mathcal B_{p,q}^s(M)}\ls \|f\|_{B_{p,q}^s(M)}$, with the boundedness of $\mathcal{M}$ on $L^{p/r}(M)$ replaced by the Fefferman-Stein vector-valued inequality.
Thus, (ii) holds true, and the proof of Theorem \ref{thm8.7x} is then completed.
\end{proof}

\section{Appendix}\label{sec-appendix}

%%%%%%%%%%%%%%%%%%%%%%%%%%%%%%%%%%%%%%%%%%%%%%%%%%%%%%%%%%%%%%%%%%%%%%

\hskip\parindent
The main aim of this section is to show Theorem \ref{thm-CRF}.
Using the language of Christ cubes,
we restate the Marcinkiewicz--Zygmund inequality,
whose proof was essentially given in  \cite[Proposition~4.1]{CKP},
here we re-present its proof for convenience as the constants involved is very subtle.
Recall that, in what follows, the spectral space $\Sigma_\lz^p$ is the same as in Remark \ref{rem4.6x}.

\begin{lem}\label{lem9.1x}
For any $\ez\in(0,1)$, there exists a  large number $j_\ez\in\nn$ such that, for all $j\in\zz$ and $\tau\in I_j$, the collection $\{Q_\tau^{j,\nu}:\,
\nu\in\{1,\ldots, N_{\tau}^{j}\}\,\}$, which is the set
of all Christ cubes $Q_{\tau'}^{j+j_\ez}$ contained in $Q_\tau^j$,
satisfies the following:
\begin{enumerate}
\item[\rm (i)] for all $f\in \Sigma_\lz^p(M)$ with $\lz,\,p\in[1, \infty)$, $j\ge -\frac2{\bz_0}\log_\delta \lz$,  and   all $\xi_{\tau}^{j,\nu}\in Q_{\tau}^{j,\nu}$
with $\tau\in I_j$ and $\nu\in\{1,\dots, N_\tau^{j}\}$,
\begin{equation*}%\label{MZI}
\bigg\{\sum_{\tau\in I_j}
\sum_{\nu=1}^{N_{\tau}^{j}} \int_{Q_{\tau}^{j,\nu}}
|f(x)-f(\xi_{\tau}^{j,\nu})|^p\,d\mu(x)\bigg\}^{1/p}
\le \frac{\epsilon}8\lf(\lz^{2/\bz_0}\delta^j\r)^{\alpha_0} \|f\|_{L^p(M)};
\end{equation*}
\item[\rm (ii)] for all $f\in \Sigma_\lz^\fz(M)$ with $\lz\in[1, \infty)$, $j\ge -\frac2{\bz_0}\log_\delta \lz$,  and   all $\xi_{\tau}^{j,\nu}\in Q_{\tau}^{j,\nu}$
with $\tau\in I_j$ and $\nu\in\{1,\dots, N_\tau^{j}\}$,
\begin{equation*}%\label{MZI'}
\sup_{\tau\in I_j}\sup_{1\le\nu\le N_{\tau}^{j}}\sup_{x\in Q_{\tau}^{j,\nu}}
|f(x)-f(\xi_{\tau}^{j,\nu})|
\le \frac{\epsilon}8(\lz^{2/\bz_0}\delta^j)^{\alpha_0} \|f\|_{L^\infty(M)}.
\end{equation*}
\end{enumerate}
\end{lem}

\begin{proof}
Let $\phi\in C_c^\infty(\rr_+)$ be such that $\supp\phi\subset [0,2]$,
$0\le\phi\le1$ and $\phi\equiv1$ on $[0,1]$.
For $\sz=d+1$,
by Proposition \ref{prop2.10x}(ii) and the self-adjoint property of $L$,
 there exists a positive constant $C_{d,\phi}\,$ such that, for all $x,\,x',\,y\in M$ satisfying  $\rho(x,x')\le \lz^{-2/\bz_0}$,
\begin{equation*}
|\phi(\lz^{-1} \sqrt L)(x, y)-\phi(\lz^{-1} \sqrt L)(x', y)|
\le C_{d,\phi}\,  [\lz^{2/\bz_0}\rho(x,x')]^{\alpha_0} D_{\lz^{-2/\bz_0}, d+1}(x,y).
\end{equation*}
Also, $\phi(\lz^{-1}\sqrt L)f=f$ for all $f\in \Sigma_\lz^p(M)$.
If we choose $j_\ez\in\nn$ such that
\begin{equation}\label{eq:9.1x}
C_\natural \dz^{j_\ez}\le 1,
\end{equation}
then, for all $x\in Q_\tau^{j,\nu}$, we have $\rho(x,\xi_\tau^{j,\nu})
\le \diam Q_\tau^{j,\nu}\le C_\natural \dz^{j+j_\ez}
 \le \dz^j\le \lz^{-2/\bz_0}$, and hence
\begin{eqnarray*}
|f(x)-f(\xi_{\tau}^{j,\nu})|
&&=\lf|\int_M [\phi(\lz^{-1} \sqrt L)(x, y)
-\phi(\lz^{-1} \sqrt L)(\xi_\tau^{j,\nu}, y)]f(y)\,d\mu(y)\r|\\
&&\le   C_{d,\phi}\, (C_\natural \lz^{2/\bz_0}\dz^{j+j_\ez})^{{\alpha_0} }
 \int_M  D_{\lz^{-2/\bz_0}, d+1}(x,y)
|f(y)|\,d\mu(y).
\end{eqnarray*}
Denote by ${\rm J}$ the left-hand side of the inequality in (i).
By H\"older's inequality and the third inequality in Lemma \ref{lem2.1x}(i),
together with Fubini's theorem,
we see that
\begin{eqnarray*}
{\rm J}
&&\le  C_{d,\phi}\, (C_\natural \lz^{2/\bz_0}\dz^{j+j_\ez})^{{\alpha_0} }
\lf\{\int_M
\left|
\int_M   D_{\lz^{-2/\bz_0}, d+1}(x,y)
|f(y)|\,d\mu(y)\right|^p\,d\mu(x)\r\}^{1/p}\\
&&\le 2K8^d C_{d,\phi}\, (C_\natural \lz^{2/\bz_0}\dz^{j+j_\ez})^{{\alpha_0} }
\|f\|_{L^p(M)},
\end{eqnarray*}
which is controlled by $\frac{\epsilon}8(\lz^{2/\bz_0}\delta^j)^{\alpha_0} \|f\|_{L^p(M)}$, provided that
\begin{equation}\label{eq:9.2x}
2K8^d C_{d,\phi}\, (C_\natural \dz^{j_\ez})^{{\alpha_0} }
\le \frac{\epsilon}8.
\end{equation}
This proves (i) by choosing $j_\ez$ satisfying \eqref{eq:9.1x} and \eqref{eq:9.2x}.
 A modification of the above proof also shows (ii).
Thus, we complete the proof of Lemma \ref{lem9.1x}.
\end{proof}

Indeed, the definition of the subcubes $Q_\tau^{j,\nu}$ in
Lemma \ref{lem9.1x} depends on $j_\ez$, but
below we will not explicitly indicate this for simplicity.

Applying Lemma \ref{lem9.1x} and using the choice of $j_\ez$
(see  \eqref{eq:9.1x} and \eqref{eq:9.2x}), we argue as in the proof of \cite[Theorem~4.2]{CKP}
to obtain the following sampling theorem, the details  being omitted.

\begin{lem}\label{lem9.2x}
Let all the notation be as in Lemma \ref{lem9.1x}.
Then,
for all $f\in \Sigma_\lz^p(M)$ with $\lz\in[1,\fz)$ and $p\in[1, 2]$, for all integer
$j\ge - \frac2{\bz_0}\log_\delta \lz$,  and all $\xi_{\tau}^{j,\nu}\in Q_{\tau}^{j,\nu}$
with $\tau\in I_j$ and $\nu\in\{1,\dots, N_\tau^{j}\}$,
\begin{equation*}
(1-\epsilon)\|f\|_{L^p(M)}
\le \bigg\{\sum_{\tau\in I_j}\sum_{\nu=1}^{N_{\tau}^{j}} |Q_{\tau}^{j,\nu}|
|f(\xi_{\tau}^{j,\nu})|^p\bigg\}^{1/p}
\le (1+\epsilon)\|f\|_{L^p(M)}.
\end{equation*}
\end{lem}

The following cubature formula (see \cite[Theorem~4.4]{CKP})
follows directly from Lemma \ref{lem9.1x} and \cite[Proposition~4.6]{CKP}, the details being omitted.

\begin{cor}\label{cor9.3x}
Let all the notation be as in Lemma \ref{lem9.1x}.
For $j\in\zz$, $\tau\in I_j$ and $\nu\in\{1,\dots, N_\tau^{j}\}$, fix $\xi_{\tau}^{j,\nu}\in Q_{\tau}^{j,\nu}$.
Then there exists a sequence $\{\varepsilon_{\tau}^{j,\nu}:\,{\tau\in I_j,\,1\le\nu\le N_{\tau}^{j}}\}$ of positive constants satisfying
$\frac23\le\varepsilon_\tau^{j,\nu}\le 2$
such that for all
$f\in \Sigma_\lz^1(M)$ with $\lz\in[1,\infty)$
and
$j\ge  -\frac2{\bz_0}\log_\delta \lz$,
\begin{equation*}
\int_M f(x)\,d\mu(x)=
\sum_{\tau\in I_j}\sum_{\nu=1}^{N_{\tau}^{j}}
\varepsilon_\tau^{j,\nu} |Q_{\tau}^{j,\nu}|
f(\xi_{\tau}^{j,\nu}).
\end{equation*}
\end{cor}

For any $\delta,\, \gz\in(0,\infty)$ and $\beta\in(0,1)$, let
\begin{equation*}
E_{\delta}^{ \gz,\beta}(x,y)
:= \frac{1}
{\sqrt{|B(x, \delta)|\, |B(y,\delta)|} }\,\exp\lf\{-\gz\left[\frac{\rho(x,y)}{\delta}\right]^\beta\r\},\qquad\, x,\,y\in M.
\end{equation*}
Obviously, for any $\sz\in(0,\infty)$, there exists a positive constant $C_{\gz,\bz,\sigma}$, depending only on $\gz$, $\bz$ and $\sigma$,
such that
$$
E_{\delta}^{ \gz,\beta}(x,y) \le C_{\gz,\bz,\sigma}  D_{\delta, \sigma}(x,y)
$$
for all $\dz\in(0,\infty)$ and $x,\,y\in M$.

\begin{lem}\label{lem9.4x}
For $j\in\zz$, let $Q_\tau^{j,\nu}$ be the subcubes associated to some parameter $j_\ez$ as in Lemma \ref{lem9.1x}.
Then, for any given $\gz,\ \bz\in(0,\infty)$,
there exists a positive constant  $C_{\gz,\bz}^{\diamond}$ such that, for all $j\in\zz$,   $\xi_\tau^{j,\nu}\in Q_\tau^{j,\nu}$ with $\tau\in I_j$ and $1\le \nu\le N_\tau^j$,
and all $x,\, y\in M$,
\begin{equation}\label{eq:9.3xx}
\sum_{\tau\in I_j}\sum_{\nu=1}^{N_{\tau}^{j}}
 |Q_{\tau}^{j,\nu}|
 E_{\delta^j}^{\gz,\beta}(x, \xi_{\tau}^{j,\nu})
 E_{\delta^j}^{\gz,\beta}(\xi_{\tau}^{j,\nu}, y)
 \le C_{\gz,\bz}^\diamond E_{\delta^j}^{\gz,\beta}(x,y)
\end{equation}
and
\begin{equation}\label{eq:9.4xx}
\int_M
 E_{\delta^j}^{\gz,\beta}(x, z)
 E_{\delta^j}^{\gz,\beta}(z, y)\,d\mu(z)
 \le C_{\gz,\bz}^\diamond E_{\delta^j}^{\gz,\beta}(x,y),
\end{equation}
where $\dz\in(0,1)$ is the same as in Lemma \ref{lem3.1x} and $C_{\gz,\bz}^\diamond $ is independent of $\ez$ and $j_\ez$.
\end{lem}

\begin{proof}  To obtain \eqref{eq:9.3xx}, it is enough to prove that
\begin{eqnarray*}
&&\sum_{\tau\in I_j}\sum_{\nu=1}^{N_{\tau}^{j}}
 \frac{|Q_{\tau}^{j,\nu}| }{|B(\xi_{\tau}^{j,\nu},\delta^j)|}
 \,\exp\lf\{-\gz\left[\frac{\rho(x,\xi_{\tau}^{j,\nu})}{\delta^j}\right]^\beta\r\}\,
 \exp\lf\{-\gz\left[\frac{\rho(\xi_{\tau}^{j,\nu},y)}{\delta^j}\right]^\beta\r\}\\
&&\quad\quad\le C_{\gz,\bz}^\diamond
\,\exp\lf\{-\gz\left[\frac{\rho(x,y)}{\delta^j}\right]^\beta\r\},
\end{eqnarray*}
which was implicitly proved in
\cite[Lemma~3.10]{KP}.
To obtain \eqref{eq:9.4xx},
 we split the integral in its left-hand side into
 \begin{eqnarray*}
\sum_{\tau\in I_j}\sum_{\nu=1}^{N_{\tau}^{j}}
\int_{Q_{\tau}^{j,\nu}}
E_{\delta^j}^{\gz,\beta}(x, z)
 E_{\delta^j}^{\gz,\beta}(z, y)\,d\mu(z).
\end{eqnarray*}
For any $z\in Q_{\tau}^{j,\nu}$ and $x,\,y\in M$, we observe that
$
E_{\delta^j}^{\gz,\beta}(x, z) \sim E_{\delta^j}^{\gz,\beta}(x, \xi_{\tau}^{j,\nu})
$ and
$E_{\delta^j}^{\gz,\beta}(z, y)\sim E_{\delta^j}^{\gz,\beta}(\xi_{\tau}^{j,\nu}, y)
$
with implicit positive constants independent of $j,\,x,\,y,\,z$ and $\xi_{\tau}^{j,\nu}$.
Hence, the left-hand side of \eqref{eq:9.4xx} is comparable to that of \eqref{eq:9.3xx},
so \eqref{eq:9.4xx} follows from \eqref{eq:9.3xx}.
This finishes the proof of Lemma \ref{lem9.4x}.
\end{proof}

In what follows, we fix  functions
$(\Gamma_0, \Gamma)\in C_c^\infty(\rr_+)$
satisfying
\begin{equation}\label{eq:9.3x}
\supp \Gamma_0\subset [0, \delta^{-\bz_0}],\quad
\Gamma_0(\lz)=1\,\,\textup{for}\,\,\lz\in[0,\delta^{-\bz_0/2}],
\end{equation}
and
\begin{equation}\label{eq:9.4x}
\supp \Gamma\subset [\delta^{\bz_0}, \delta^{-\bz_0}],\quad \,
\quad
\Gamma_0(\lz)=1\,\,\textup{for}\,\,\lz\in[\delta^{\bz_0/2}, \delta^{-\bz_0/2}],
\quad
\textup{and} \,\, 0\le \Gamma_0,\,\Gamma\le1.
\end{equation}
According to Lemma \ref{lem2.7x}, we may as well assume that
there exist constants $A\in(0,\fz)$ and $\beta\in(0,1)$ such that
\begin{equation}\label{eq:9.5x}
\|\Gamma_0^{(k)}\|_{L^\infty(\rr_+)}
\le (A k^{1+\bz})^{k}\quad \textup{and}\quad
\|\Gamma^{(k)}\|_{L^\infty(\rr_+)}
\le (A k^{1+\bz})^{k},\qquad k\in\nn.
\end{equation}
For $j\in\nn$, let
\begin{equation}\label{eq:9.6x}
\Gamma_j(\lz):=\Gamma(\dz^{j\bz_0/2}\lz), \qquad \lz\in\rr.
\end{equation}
For any $m\in\nn$, applying Proposition \ref{prop2.12x},
we find positive constants $\gz:=\gz(m,A, \bz)$
and $C^{\flat}_{\gz,\bz,m}:=C(m, \gz,\bz)$
such that, for all $j\in\zz_+$, $\delta\in(0,1]$ and $x,\,y\in M$,
\begin{equation}\label{eq:9.7x}
|L^m \Gamma_j(\sqrt L)(x, y)|
\le C^{\flat}_{m,\gz,\bz}
 \delta^{-\bz_0mj}\,  E_{\dz^j}^{\gz, \frac1{10(1+\bz)}}
(x,y)
\end{equation}
and,  for all $x,\,y,\,y'\in M$ satisfying $\rho(y,y')\le \delta$,
\begin{equation}\label{eq:9.8x}
|L^m\Gamma_j(\sqrt L)(x, y)-L^m\Gamma_j(\sqrt L)(x, y')|
\le C^{\flat}_{m,\gz,\bz}
 \delta^{-\bz_0mj}
\lf[\frac{\rho(y,y')}{\delta}\r]^{\alpha_0}\, E_{\dz^j}^{\gz, \frac1{10(1+\bz)}}(x,y).
\end{equation}
With these $\Gamma_j$, we apply Lemma \ref{lem9.4x} and Corollary \ref{cor9.3x}, and
also some ideas used in the proof of
\cite[Lemma~4.2]{KP} to obtain the following conclusion.

\begin{lem} \label{lem9.5x}
Let  $\{\Gamma_j\}_{j\in\zz_+}$ be as in \eqref{eq:9.6x} so that \eqref{eq:9.7x} and \eqref{eq:9.8x} hold true.
Then there exist $\ez_0\in(0,1)$ and operators $\{\wz\varphi_j(\sqrt{L})\}_{j\in\zz_+}$ such that, if $f\in L^2(M)$ satisfies that $\Gamma_j(\sqrt L)f=f$ for some $j\in\nn$, then
\begin{equation}\label{eq:9.9x}
f(\cdot)= \sum_{\tau\in I_j} \sum_{\nu=1}^{N_\tau^{j}}
|Q_{\tau}^{j,\nu}| f(\xi_\tau^{j,\nu})
\widetilde \varphi_j(\sqrt L)(\xi_\tau^{j,\nu}, \,\cdot)
\end{equation}
in $L^2(M)$, where  $\{Q_\tau^{j,\nu}:\, \tau\in I_j,\, 1\le \nu\le N_\tau^j\}$
are  the subcubes associated to some parameter $j_{\ez_0}$ as in Lemma \ref{lem9.1x},
$\xi_{\tau}^{j,\nu}$ is any point in $Q_{\tau}^{j,\nu}$,
and $\widetilde \varphi_j(\sqrt L)(\xi_\tau^{j,\nu}, x)$
is the integral kernel of  $\widetilde \varphi_j(\sqrt L)$.
Moreover, the following hold true:
\begin{enumerate}
\item[\rm(i)]   For any $m\in\zz_+$, there exist positive constants $\gamma:=\gamma(A, m, \bz)$ and $C:=C(m,\gz,\bz)$ such that,
for all $j\in\zz_+$ and $x,\,y\in M$,
\begin{equation}\label{l8.51}
|L^m\widetilde \varphi_j(\sqrt L)(x,y)|
\le C \delta^{-m \bz_0 j}E_{\delta^j}^{\frac\gz2, \frac1{10(1+\bz)}}(x, y)
\end{equation}
and,
for all $j\in\zz_+$ and $x,\,y,\,y'\in M$ satisfying $\rho(y,y')\le \delta^{j}$,
\begin{equation}\label{l8.52}
|L^m\widetilde \varphi_j(\sqrt L)(x,y)-L^m\widetilde \varphi_j(\sqrt L)(x,y')|
\le  C \delta^{-m\bz_0 j} [\delta^{-j} \rho(y,y')]^{\alpha_0}  E_{\delta^j}^{\frac\gz2, \frac1{10(1+\bz)}}(x, y);
\end{equation}

\item[\rm(ii)] Let $(\Phi_0, \Phi)$
satisfy \eqref{eq:2.14xx} and \eqref{eq:2.15xx}.
For $j\in\nn$, define $\Phi_j$ as in \eqref{eq:2.16xx}.
Then, for any  $m\in\nn$ and $\sz\in(2d,\fz)$,
there exists a positive constant $C:=C(\sz,m)$
such that, for all $j,\,k\in\zz_+$  and $x,\,y\in M$,
\begin{eqnarray*}
|(\Phi_k( \sqrt L)
\wz\varphi_j(\sqrt L))(x, y)|
\le C \delta^{|k-j|(m\bz_0-d)} D_{\delta^{k\wedge j}, \sigma}(x,y).
\end{eqnarray*}
\end{enumerate}
\end{lem}

\begin{proof}
Let $\ez_0\in(0,1)$ which will be determined later (see \eqref{eq:9.19x} below).
We then choose a large integer $j_{\ez_0}$ as in Lemma \ref{lem9.1x} (see \eqref{eq:9.1x} and \eqref{eq:9.2x})
and  define $\{Q_\tau^{j,\nu}:\, \tau\in I_j,\, 1\le \nu\le N_\tau^j\}$
to be the subcubes of $Q_\tau^j$ associated to such a $j_{\ez_0}$
as in Remark \ref{rem7.2x}.

Let
$\Theta\in C_c^\infty(\rr_+)$
such that $0\le \Theta\le1$,
$
\supp \Theta\subset [0, \delta^{-2\bz_0}]$ and
$
\Theta\equiv1$ on $[0,\delta^{-3\bz_0/2}].
$
Define  $\Theta_j:= \Theta(\delta^{j\bz_0/2}\sqrt L)$ for $j\in\zz_+$.
For simplicity, we use
$\Gamma_j(x,y)$ and $\Theta_j(x,y)$ to denote the integral kernels of
the operators $\Gamma_j:=\Gamma_j(\sqrt L)$ and
$\Theta_j:=\Theta_j(\sqrt L)$, respectively. Define the operator $U_j$ which is associated to the kernel
$$U_j(x,y):= \sum_{\tau\in I_j}\sum_{\nu=1}^{N_{\tau}^{j}}
\frac1{(1+\epsilon_0)^2}  |Q_{\tau}^{j,\nu}| \Theta_j(x, \xi_{\tau}^{j,\nu})
\Theta_j(\xi_{\tau}^{j,\nu},y),\qquad x,\ y\in M.
$$
Observe that $\Gamma_j(L^2(M))\subset \Sigma_{\delta^{-(j+3)\bz_0/2}}^2(M)$
and  $\Theta_j g=g$ for all
$g\in \Sigma_{\delta^{-(j+3)\bz_0/2}}^2(M)$.
In particular, for any $g\in \Sigma_{\delta^{-j\bz_0/2}}^2(M)$, we have
$$
\laz U_j g, \, g\raz
=\sum_{\tau\in I_j}\sum_{\nu=1}^{N_{\tau}^{j}}
\frac1{(1+\epsilon_0)^2} |Q_{\tau}^{j,\nu}| |g(\xi_{\tau}^{j,\nu})|^2,
$$
which, combined with Lemma \ref{lem9.2x} and the fact $\frac{1-\epsilon_0}{1+\epsilon_0}\ge 1-2\epsilon_0$, implies that
\begin{eqnarray}\label{eq:9.10x}
(1-2\epsilon_0)^2\|g\|_{L^2(M)}^2
\le \laz U_j g, \, g\raz \le \|g\|_{L^2(M)}^2.
\end{eqnarray}
Define
$V_j:= \Gamma_jU_j\Gamma_j$ and $R_j:= \Gamma_j({\rm Id}-U_j)\Gamma_j$.
Since $0\le \Gamma_j\le1$,
the functional calculus implies that
$\|\Gamma_jf\|_{L^2(M)}\le \|f\|_{L^2(M)}$.
Thus, for  any $f\in L^2(M)$, by $R_j= \Gamma_j^2-\Gamma_jU_j\Gamma_j$ and \eqref{eq:9.10x}, we have
\begin{eqnarray*}
0\le \laz R_j f, \, f\raz
= \laz \Gamma_j f,\, \Gamma_j f\raz
- \laz U_j\Gamma_j f,\, \Gamma_j f\raz
\le[1-(1-2\epsilon_0)^2]\|\Gamma_jf\|_{L^2(M)}^2
\le 4\epsilon_0\|f\|_{L^2(M)}^2,
\end{eqnarray*}
which implies that $R_j$ is  bounded  on $L^2(M)$
with operator norm at most $2\sqrt\epsilon_0$.
Hence, it makes sense to define
$$T_j:= (I-R_j)^{-1}= I+\sum_{k=1}^\infty R_j^k.$$
Since  $\Theta_j\Gamma_j=\Gamma_j$,
we use the expression of $U_j$ to find that the operator $V_j$ has a kernel
\begin{eqnarray}\label{eq:9.11x}
V_j(x,y)
&&=\sum_{\tau\in I_j}\sum_{\nu=1}^{N_{\tau}^{j}}
\frac1{(1+\epsilon_0)^2} |Q_{\tau}^{j,\nu}|
\Gamma_j(x, \xi_{\tau}^{j,\nu})
\Gamma_j(\xi_{\tau}^{j,\nu}, y),\qquad x,\ y\in M.
\end{eqnarray}
If $f\in L^2(M)$ such that $\Gamma_j(\sqrt L)f=f$,
then $f
= T_j(f-R_jf)
=T_j(f- \Gamma_j^2 f +V_j f)
= T_j V_j f$,
so that
\begin{eqnarray*}
f(x)&&=\sum_{\tau\in I_j}\sum_{\nu=1}^{N_{\tau}^{j}}
\frac1{(1+\epsilon_0)^2} |Q_{\tau}^{j,\nu}|\Gamma_j f(\xi_{\tau}^{j,\nu})
T_j\left(\Gamma_j(\cdot, \xi_{\tau}^{j,\nu})\right)(x),\quad x\in M
\end{eqnarray*}
in $L^2(M)$,
which gives us the identity \eqref{eq:9.9x} by setting $\widetilde \varphi_j(\sqrt L)$ to be the operator whose integral kernel is
$$\widetilde \varphi_j(\sqrt L)(x,y)
:= \frac1{(1+\epsilon_0)^2} T_j\left(\Gamma_j(\cdot, y)\right)(x),\quad x,\,y\in M.$$

For simplicity, we write $\wz\varphi_j(x,y)$ instead of
the kernel $\widetilde \varphi_j(\sqrt L)(x,y)$.
Let $\bz':=\frac1{10(1+\bz)}$.
From \eqref{eq:9.11x}, \eqref{eq:9.7x} and \eqref{eq:9.3xx}, we deduce that
\begin{eqnarray*}
|V_j(x,y)|
\le \frac{(C^{\flat}_{0,\gz,\beta})^2}{(1+\epsilon_0)^2}
\sum_{\tau\in I_j}\sum_{\nu=1}^{N_{\tau}^{j}}
 |Q_{\tau}^{j,\nu}|
 E_{\delta^j}^{\gz,\bz'}(x, \xi_{\tau}^{j,\nu})
 E_{\delta^j}^{\gz,\bz'}(\xi_{\tau}^{j,\nu}, y)
\le \frac{(C^{\flat}_{0,\gz,\beta})^2 C_{\gz,\bz'}^\diamond}{(1+\epsilon_0)^2}
 E_{\delta^j}^{\gz,\bz'}(x, y).\notag
\end{eqnarray*}
By  \eqref{eq:9.7x} and \eqref{eq:9.4xx}, we have
\begin{eqnarray*}
|\Gamma_j^2(x,y)|
\le (C^{\flat}_{0,\gz,\beta})^2
\int_M E_{\delta^j}^{\gz,\beta'}(x, u)
E_{\delta^j}^{\gz,\beta'}(u, y)\,d\mu(u)
\le (C^{\flat}_{0,\gz,\beta})^2 C^{\ast}_{ \gz,\beta'} E_{\delta^j}^{\gz,\beta'}(x, y).
\end{eqnarray*}
Combining the last two inequalities implies that
\begin{eqnarray*}
|R_j(x,y)|
= |\Gamma_j^2(x,y)-V_j(x,y)|
&&\le  (C_{\gz,\bz'}^\diamond +  C^{\ast}_{ \gz,\beta'}) (C^{\flat}_{0,\gz,\beta})^2
E_{\delta^j}^{\gz,\beta'}(x, y).
\end{eqnarray*}
Let $\mathbb A:= (C_{\gz,\bz'}^\diamond +  C^{\ast}_{ \gz,\beta'}) (C^{\flat}_{0,\gz,\beta})^2$.
Applying the above  estimate  and \eqref{eq:9.4xx} repeatedly, we see that
\begin{eqnarray}\label{eq:9.15x}
|R_j^k(x,y)|\le \mathbb A^{k} (C^{\ast}_{ \gz,\beta'})^{k-1} E_{\delta^j}^{\gz,\beta'}(x, y),\qquad k\in\nn.
\end{eqnarray}
Consequently, for all $k\in\nn$,
\begin{eqnarray}\label{eq:9.16x}
|R_j^k \Gamma_j(x,y)|
\le (\mathbb A C^{\ast}_{ \gz,\beta'})^k C^{\flat}_{0,\gz,\beta}
 E_{\delta^j}^{\gz,\beta'}(x, y).
\end{eqnarray}
Observe that \eqref{eq:9.4xx} implies that,
for all $j\in\zz_+$ and $x,\,y\in M$,
$$\|R_j(x,\cdot)\|_{L^2(M)} \|\Gamma_j(\cdot,y)\|_{L^2(M)}
\le C_{\gz,\bz'}^\ast \mathbb A C^{\flat}_{0,\gz,\beta}.
$$
By this and Lemma \ref{lem2.2x}, we see that, for $k\ge2$,
\begin{eqnarray}\label{eq:9.17x}
|R_j^k \Gamma_j(x,y)|
= |R_j R_j^{k-1}\Gamma_j(x,y)|
&&\le \mathbb A C^{\flat}_{0,\gz,\beta} C^{\ast}_{ \gz,\beta'}  \frac{\|R_j^{k-1}\|_{L^2(M)\to L^2(M)}}
{\sqrt{|B(x, \delta^j)|\, |B(y,\delta^j)|}}\notag\\
&&\le \mathbb A C^{\flat}_{0,\gz,\beta} C^{\ast}_{ \gz,\beta'}\frac{(2\sqrt\epsilon_0)^{k-1}}
{\sqrt{|B(x, \delta^j)|\, |B(y,\delta^j)|}}.
\end{eqnarray}
By taking geometric means between \eqref{eq:9.16x} and \eqref{eq:9.17x}, we have
\begin{eqnarray}\label{eq:9.18x}
|R_j^k \Gamma_j(x,y)|
\le C^{\flat}_{0,\gz,\beta} \sqrt{(C^{\ast}_{ \gz,\beta'}\mathbb A)^{k+1}
 (2\sqrt\epsilon_0)^{k-1} } \, E_{\delta^j}^{ \gz/2,\beta'}(x, y)
\ls 2^{-k}\,
 E_{\delta^j}^{ \gz/2,\beta'}(x, y),
\end{eqnarray}
provided that we choose $\epsilon_0$ small enough satisfying
\begin{eqnarray}\label{eq:9.19x}
0<\sqrt\epsilon_0 \le \frac1{4 \mathbb A C^{\ast}_{ \gz,\beta'}}
=\frac1{4 (C_{\gz,\bz}^\diamond +  C^{\ast}_{ \gz,\beta'}) (C^{\flat}_{0,\gz,\beta})^2
 C^{\ast}_{ \gz,\beta'}}.
\end{eqnarray}
Therefore, by writing
\begin{eqnarray*}
(1+\epsilon_0)^2\widetilde \varphi_j(\sqrt L)(x,y)
= T_j\left(\Gamma_j(\cdot, y)\right)(x)
= \Gamma_j(x, y)+  R_j \left(\Gamma_j(\cdot, y)\right)(x)
+ \sum_{k=2}^\infty  R_j^k \left(\Gamma_j(\cdot, y)\right)(x),
\end{eqnarray*}
we use \eqref{eq:9.18x} to obtain that
$
|\widetilde \varphi_j(\sqrt L)(x,y)|
\ls
E_{\delta^j}^{ \gz/2,\beta'}(x, y),
$
which proves \eqref{l8.51} for $m=0$.
If $m\in\mathbb N$, then we write
\begin{eqnarray*}
(1+\epsilon_0)^2 L^m\widetilde \varphi_j(x,y)
&&= T_j\left(L^m\Gamma_j(\cdot, y)\right)(x)
= L^m\Gamma_j(x, y)
+ \sum_{k=1}^\infty  R_j^k \left(L^m\Gamma_j(\cdot, y)\right)(x).
\end{eqnarray*}
Next, we replace the operator $\Gamma_j$ in \eqref{eq:9.16x} and \eqref{eq:9.17x} by $L^m \Gamma_j$.
Following the previous argument for the case $m=0$,  we  obtain  \eqref{l8.51} for the case $m\ge1$.

To obtain \eqref{l8.52}, we write
\begin{eqnarray*}
&&(1+\epsilon_0)^2|L^m\widetilde \varphi_j(\sqrt L)(x,y)-L^m\widetilde \varphi_j(\sqrt L)(x,y')|\\
&&\quad\le |L^m\Gamma_j(x, y)-L^m\Gamma_j(x, y')|
+  \left|\int_M R_j(x, u)[L^m\Gamma_j(u, y)-L^m\Gamma_j(u,y')]\,d\mu(u)\right|\\
&&\quad
\quad+ \sum_{k=2}^\infty
 \left|\int_M R_j^k(x, u) [L^m\Gamma_j(u, y)-L^m\Gamma_j(u,y')]\,d\mu(u)\right|.
\end{eqnarray*}
Then, applying \eqref{eq:9.8x} and following the previous  argument,
we obtain \eqref{l8.52}, the details being omitted.

To show (ii), for $j\in\zz_+$, we  use $\Phi_j\Gamma_j(x,y)$ to
denote the kernel of $\Phi_j(\sqrt L)\Gamma_j(\sqrt{L})$. Notice
that Proposition \ref{prop2.14x} implies that
\begin{equation}\label{eq:9.19xx}
|\Phi_k\Gamma_j(x,y)|
\ls \dz^{|k-j|(m\bz_0-d)} D_{\dz^{k\wedge j},2\sz}(x,y),\qquad x,\ y\in M.
\end{equation}
Then, instead of \eqref{eq:9.16x}, we apply \eqref{eq:9.15x} to conclude that,
for all $i\in\nn$,
$$|R_j^i\Phi_k\Gamma_j(x,y)|
\ls (C^{\ast}_{ \gz,\beta'}\mathbb A)^i \dz^{|k-j|(m\bz_0-d)} D_{\dz^{k\wedge j},2\sz}(x,y),\qquad x,\ y\in M.$$
For all $i\ge2$, similar to the estimate  of \eqref{eq:9.17x}, we use Lemma \ref{lem2.2x}
and \eqref{eq:9.19xx} to derive that
\begin{eqnarray*}
|R_j^i \Phi_k\Gamma_j(x,y)|
&&\le {\|R_j(x,\cdot)\|_{L^2(M)}\|R_j^{i-1}\|_{L^2(M)\to L^2(M)}
\|\Phi_k\Gamma_j(\cdot, y)\|_{L^2(M)}}\\
&&
\ls  \frac{\dz^{|k-j|(m\bz_0-d)} (2\sqrt{\epsilon_0})^{i-1}}
{\sqrt{|B(x, \delta^{j\wedge k})|\, |B(y,\delta^{j\wedge k})|}}
\end{eqnarray*}
uniformly for all $x,\,y\in M$ and $k,\,j\in\zz_+$.
By taking geometric mean of the above inequalities, we see that, for all $x$, $y\in M$,
$$|R_j^i\Phi_k\Gamma_j(x,y)|
\ls (2\sqrt {\epsilon_0} C^{\ast}_{ \gz,\beta}\mathbb A)^i \dz^{|k-j|(m\bz_0-d)}
D_{\dz^{k\wedge j},\sz}(x,y)
\ls 2^{-i} \dz^{|k-j|(m\bz_0-d)} D_{\dz^{k\wedge j},\sz}(x,y),$$
if we choose $\epsilon_0$ as in \eqref{eq:9.19x}.
Then, summing over all $i\ge 2$ and using the fact that
$$(1+\epsilon_0)^2 (\Phi_k(\sqrt L)\widetilde \varphi_j(\sqrt L))(x,y)
= \Phi_k\Gamma_j(x, y)+  \sum_{i=1}^\infty  R_j^i \Phi_k\Gamma_j(x, y),
$$
we argue as before and obtain the desired estimate in (ii), which completes
the proof of Lemma \ref{lem9.5x}.
\end{proof}

\begin{proof}[Proof of Theorem \ref{thm-CRF}] By similarity, we only prove the case $f\in \cd'(M)$.
Starting from the continuous Calder\'on reproducing formula \eqref{c-crf}, we write
\begin{eqnarray*}
f&&= \sum_{j=0}^\infty \wz\Phi_j(\sqrt L)\Phi_j(\sqrt L) f
=  \sum_{j=0}^\infty \int_M \int_M
\wz\Phi_j(\sqrt L)(\cdot,y)\Phi_j(\sqrt L)(y,z) f(z)\,d\mu(y)\,d\mu(z),
\end{eqnarray*}
where the inequality holds true in $\cd'(M)$.
Applying Lemma \ref{lem9.5x} to the function $\Phi_j(\sqrt L)(\cdot,z)$,
we find that
$$
\Phi_j(\sqrt L)(y,z)=
\sum_{\tau\in I_j} \sum_{\nu=1}^{N_\tau^{j}}
|Q_{\tau}^{j,\nu}| \Phi_j(\sqrt L)(\xi_\tau^{j,\nu}, z)
\widetilde \varphi_j(\sqrt L)(\xi_\tau^{j,\nu}, y)
$$
for all $j\in\zz_+$ and $y,\,z\in M$,
where $\widetilde \varphi_j(\sqrt{L})$ satisfies (i), (ii) and (iii)
of Lemma \ref{lem9.5x}.  It is easy to show that the above equality holds true in $L^2(M)$, and hence in $\cd'(M)$. Consequently,
\begin{eqnarray*}
f(\cdot)
=  \sum_{j=0}^\infty \sum_{\tau\in I_j} \sum_{\nu=1}^{N_\tau^{j}}
|Q_{\tau}^{j,\nu}| (\Phi_j(\sqrt L)f )(\xi_\tau^{j,\nu})
\int_M
\wz\Phi_j(\sqrt L)(\cdot,y)\widetilde \varphi_j(\sqrt L)(\xi_\tau^{j,\nu}, y)\,d\mu(y)
\end{eqnarray*}
in $\cd'(M)$,
which implies \eqref{DCRF} by setting $\Psi_j(\sqrt L)$ to be the operator  whose associated kernel is
$$\Psi_j(\sqrt L)( x,z)
:= \int_M
\wz\Phi_j(\sqrt L)(x,y)\widetilde \varphi_j(\sqrt L)(z, y)\,d\mu(y),
\qquad y,\,z\in M.
$$
Observe that Proposition \ref{prop2.10x} implies that $\wz\Phi_j(\sqrt L)$
satisfies \eqref{eq:2.12x} and \eqref{eq:2.13x}.
By this,  Lemma \ref{lem9.5x} and Lemma \ref{lem2.1x}(ii),
 we obtain (b) through (d) of the theorem.
This finishes the proof of Theorem \ref{thm-CRF}.
\end{proof}

\noindent{\bf Acknowledgements.} Liguang Liu is very grateful
to Professor Alexander Grigor'yan for his valuable suggestions
and discussions on this article. The authors would also like to thank
referees for their helpful remarks which improved the presentation
of this article.

%%%%%%%%%%%%%%%%%%%%%%%%%%%%%%%%%%%%%%%%%%%%%%%%%%%%%%%%%%%%%%%%%%%%%

%%%%%%%%%%%%%%%%%%%%%%%%%%%%%%%%%%%%%%%%%%%%%%%%%%%%%%%%%%%%%%%%%%%%%

\noindent {\sc Liguang Liu\,}

\smallskip

\noindent  Department of Mathematics, School of Information, Renmin University
of China, Beijing 100872, China

\&

\noindent Department of Mathematics,
University of Bielefeld,
33501 Bielefeld, Germany

\smallskip

\noindent {\it E-mail}: \texttt{liuliguang@ruc.edu.cn}

\bigskip

\noindent {\sc Dachun  Yang\,}(Corresponding author) and {\sc Wen Yuan\,}

\smallskip

\noindent School of Mathematical Sciences, Beijing Normal University,
Laboratory of Mathematics and Complex Systems, Ministry of
Education, Beijing 100875, China

\smallskip

\noindent {\it E-mails}: \texttt{dcyang@bnu.edu.cn}

\hspace{1cm}\texttt{wenyuan@bnu.edu.cn}


\begin{thebibliography}{99}

\bibitem{ax02} D. R. Adams and J. Xiao, Morrey spaces in harmonic
analysis,  Ark. Mat. 50 (2012), 201-230.

\vspace{-0.25cm}
\bibitem{ADM}
P. Auscher, X. T. Duong and A. McIntosh,
{ Boundedness of Banach space valued singular integral operators and Hardy spaces},
Unpublished preprint (2005).


\vspace{-0.25cm}
\bibitem{BGK}  M. T. Barlow, A. Grigor'yan and T. Kumagai,
{ On the equivalence of parabolic Harnack inequalities and heat kernel estimates},
J. Math. Soc. Japan 64 (2012),  1091-1146.


\vspace{-0.25cm}
\bibitem{BH}
M. Bownik and K.-P. Ho,
{Atomic and molecular decompositions of anisotropic Triebel--Lizorkin spaces},
Trans. Amer. Math. Soc. 358 (2006),  1469-1510.

\vspace{-0.25cm}
\bibitem{BDY}
H.-Q. Bui, X. T. Duong and L. Yan, { Calder\'on reproducing formulas and new Besov spaces associated with operators},
Adv. Math. 229 (2012), 2449-2502.


\vspace{-0.25cm}
\bibitem{chr} M. Christ,
{A $T(b)$ theorem with remarks on analytic capacity and the Cauchy integral},
Colloq. Math. 60/61 (1990), 601-628.


\vspace{-0.25cm}
\bibitem{CW1} R. R. Coifman and G. Weiss,
{ Analyse Harmonique Non-commutative sur Certains Espaces Homog\`enes},
Lecture Notes in Math. 242, Springer, Berlin, 1971.

\vspace{-0.25cm}
\bibitem{CW2} R. R. Coifman and G. Weiss,
{Extensions of Hardy spaces and their use in analysis},
Bull. Amer. Math. Soc. 83 (1977), 569-645.


\vspace{-0.25cm}
\bibitem{CKP} T. Coulhon, G. Kerkyacharian and P. Petrushev,
{ Heat kernel generated frames in the setting of Dirichlet spaces},
J. Fourier Anal. Appl. 18 (2012), 995-1066.

\vspace{-0.2cm}
\bibitem{DS} N. Dunford and J. T. Schwartz, {Linear Operators. I. General Theory.} With the assistance of W. G. Bade and R. G. Bartle, Pure and Applied Mathematics, Vol. 7, Interscience Publishers, Inc., New York; Interscience Publishers, Ltd., London 1958, xiv+858 pp.

\vspace{-0.2cm}
\bibitem{DY0} X. T. Duong and L. Yan,
{ Hardy spaces of spaces of homogeneous type},
Proc. Amer. Math. Soc. 131  (2003), 3181-3189.

\vspace{-0.25cm}
\bibitem {DY1}
X. T. Duong and L. Yan,
{ Duality of Hardy and BMO spaces associated with operators with heat kernel bounds},
J. Amer. Math. Soc. 18 (2005),  943-973.

\vspace{-0.25cm}
\bibitem {DY2}
X. T. Duong and L. Yan,  {New function spaces of BMO type,
the John-Nirenberg inequality, interpolation, and applications},
Comm. Pure Appl. Math. 58 (2005), 1375-1420.


\vspace{-0.25cm}
\bibitem{FJ90}
M. Frazier and B. Jawerth,
{A discrete transform and decompositions of distribution spaces},
J. Funct. Anal. 93 (1990), 34-170.



\vspace{-0.25cm}
\bibitem{FJW}
M. Frazier, B. Jawerth and G. Weiss,
{Littlewood-Paley Theory and The Study of Function Spaces},
CBMS Regional Conference Series in Mathematics 79,
American Mathematical Society,
Providence, RI, 1991, viii+132 pp.


\vspace{-0.25cm}
\bibitem{dx} G. Dafni and J. Xiao,
{Some new tent spaces and duality theorems for fractional Carleson
measures and $Q_\alpha({\mathbb R}^n)$}, J. Funct. Anal. 208 (2004),
377-422.

\vspace{-0.25cm}
\bibitem{ejpx} M. Ess\'en, S. Janson, L. Peng and J. Xiao,
{$Q$ spaces of several real variable}s, Indiana Univ. Math. J. 49 (2000),
575-615.


\vspace{-0.25cm}
\bibitem{GLY-ms} L. Grafakos, L. Liu and D. Yang,
{ Vector-valued singular integrals and maximal functions
on spaces of homogeneous type},
Math. Scand. 104 (2009),  296-310.

\vspace{-0.25cm}
\bibitem{Gri}
A. Grigor'yan, { Heat kernels and function
theory on metric measure spaces}, Contemp. Math. 338 (2003),
143-172.

\vspace{-0.25cm}
\bibitem{GH}  A. Grigor'yan and J. Hu,
{ Off-diagonal upper estimates for the heat kernel of the Dirichlet forms on metric spaces}, Invent. Math. 174 (2008), 81-126.

\vspace{-0.25cm}
\bibitem{GHL}  A. Grigor'yan, J. Hu  and K.-S. Lau,
{Heat kernels on metric measure spaces and an application to semilinear elliptic equations}, Trans. Amer. Math. Soc. 355 (2003),  2065-2095.

\vspace{-0.25cm}
\bibitem{GL}
A. Grigor'yan and L. Liu, {Heat kernel and Lipschitz-Besov spaces},
 Forum Math. (2014), DOI: 10.1515/forum-2014-0034.

\vspace{-0.25cm}
\bibitem{GT}
A. Grigor'yan and A. Telcs, { Two-sided estimates of
 heat kernels on metric measure spaces},
 Ann. Probab. 40 (2012),  1212-1284.

\vspace{-0.25cm}
\bibitem{HMY1} Y. Han, D. M\"uller and D. Yang,
{ Littlewood-Paley characterizations for Hardy spaces on spaces of
homogeneous type}, Math. Nachr. 279 (2006), 1505-1537.

\vspace{-0.25cm}
\bibitem{HMY2} Y. Han, D. M\"uller and D. Yang,
{A theory of Besov and Triebel--Lizorkin spaces on metric measure
spaces modeled on Carnot-Carath\'eodory spaces},
Abstr. Appl. Anal. 2008, Art. ID 893409, 250 pp.

\vspace{-0.25cm}
\bibitem{IPX}
K. Ivanov, P. Petrushev and Y. Xu, { Decomposition of spaces of distributions induced by tensor product bases}, J. Funct. Anal. 263 (2012), 1147-1197.


\vspace{-0.25cm}
\bibitem{KP} G. Kerkyacharian and P. Petrushev,
{Heat kernel based decomposition of spaces of
distributions in the framework of dirichlet
spaces}, Trans. Amer. Math. Soc. 367 (2015), 121-189.

\vspace{-0.25cm}
\bibitem{KY}
{H.~Kozono and M.~Yamazaki}, {Semilinear heat
equations and the Navier-Stokes equation with distributions in new
function spaces as initial data}, Comm.
Partial Differential Equations 19 (1994), 959-1014.

\vspace{-0.25cm}
\bibitem{lxy12} P. Li, J. Xiao and Q. Yang,
{Global mild solutions of fractional Navier-Stokes equations with
small initial data in critical Besov-$Q$ spaces},
Electron. J. Differential
Equations 185 (2014), 37 pp.

\vspace{-0.25cm}
\bibitem{lz10} P. Li and Z. Zhai, {Well-posedness and regularity of
generalized Navier-Stokes equations in some critical $Q$-spaces},
J. Funct. Anal. 259 (2010), 2457-2519.

\vspace{-0.25cm}

\bibitem{lsuyy}
Y.~Liang, Y.~Sawano, T.~Ullrich, D.~Yang and W.~Yuan,
{New characterizations of Besov-Triebel--Lizorkin-Hausdorff spaces including
coorbits and wavelets}, J. Fourier Anal. Appl. 18 (2012), 1067-1111.

\vspace{-0.25cm}

\bibitem{lsuyy2} Y. Liang, D. Yang, W. Yuan, Y. Sawano and T. Ullrich,
{A new framework for generalized Besov-type and Triebel-Lizorkin-type spaces},
Dissertationes Math. (Rozprawy Mat.) 489 (2013), 1-114.

\vspace{-0.25cm}
\bibitem{ma03}
{A.~Mazzucato},
{Besov-Morrey spaces: function space theory and applications to non-linear PDE},
Trans. Amer. Math. Soc. 355 (2003), 1297-1369.


\vspace{-0.25cm}

\bibitem{Ou}
E. M.  Ouhabaz,  { Analysis of Heat Equations on Domains}, London Mathematical
Society Monographs Series 31, Princeton University Press, Princeton, NJ, 2005.


\vspace{-0.25cm}

\bibitem{p76}
J. Peetre, New Thoughts on Besov Spaces, Duke University, Durham, N.C., 1976.

\vspace{-0.25cm}

\bibitem{rs}
T. Runst and W. Sickel, {Sobolev Spaces of Fractional Order, Nemytskij Operators,
and Nonlinear Partial Differential Equations}, de Gruyter Series in Nonlinear
Analysis and Applications, 3. Walter de Gruyter \& Co., Berlin, 1996.


\vspace{-0.25cm}

\bibitem{s05} Y. Sawano,
Sharp estimates of the modified Hardy-Littlewood
maximal operator on the nonhomogeneous space via covering
lemmas, Hokkaido Math. J. 34 (2005), 435-458.

\vspace{-0.25cm}

\bibitem{s09} Y. Sawano, A note on Besov-Morrey spaces and Triebel-Lizorkin-Morrey
spaces, Acta Math. Sin. (Engl. Ser.)  25  (2009), 1223-1242.

\vspace{-0.25cm}

\bibitem{s10} Y. Sawano, Besov-Morrey spaces and Triebel-Lizorkin-Morrey
spaces on domains, Math. Nachr.  283  (2010), 1456-1487.

\vspace{-0.25cm}

\bibitem{st}
Y. Sawano and H. Tanaka, {Decompositions of Besov-Morrey spaces and
Triebel--Lizorkin-Morrey spaces}, Math. Z. 257 (2007), 871-904.

\vspace{-0.25cm}

\bibitem{syy}
Y. Sawano, D. Yang and W. Yuan, {New applications of Besov-type
and Triebel--Lizorkin-type spaces}, J. Math. Anal. Appl. 363 (2010), 73-85.


\vspace{-0.25cm}
\bibitem{s012}
W.~Sickel, {Smoothness spaces related to Morrey spaces---a survey. I},
Eurasian Math. J. 3 (2012), 110-149.

\vspace{-0.25cm}
\bibitem{s013}
W.~Sickel, {Smoothness spaces related to Morrey spaces---a survey. II},
Eurasian Math. J. 4 (2013), 82-124.

\vspace{-0.25cm}
\bibitem{TX}
L. Tang and J. S. Xu,
{ Some properties of Morrey type Besov-Triebel spaces},
Math. Nachr. 278 (2005),  904-917.

\vspace{-.25cm}
\bibitem{t83} H. Triebel, {Theory of Function Spaces}, Birkh\"auser
Verlag, Basel, 1983.

\vspace{-0.25cm}
\bibitem{t92}
H. Triebel, {Theory of Function Spaces. II},
 Birkh\"auser Verlag, Basel, 1992.

\vspace{-.25cm}
\bibitem{t95} H. Triebel, Interpolation Theory, Function Spaces, Differential Operators,
Second edition, Johann Ambrosius Barth, Heidelberg, 1995.

\vspace{-0.25cm}
\bibitem{t13}
H. Triebel,
Local Function Spaces, Heat and Navier-Stokes Equations,
EMS Tracts in Mathematics 20,
European Mathematical Society (EMS), Z\"urich, 2013.

\vspace{-0.25cm}
\bibitem{xj01} J. Xiao, {Holomorphic $Q$ Classes}, Lecture Notes in
Math. 1767, Springer, Berlin, 2001.

\vspace{-0.25cm}
\bibitem{xj06} J. Xiao, {Geometric $Q_p$ Functions},
Birkh\"auser Verlag, Basel, 2006.


\vspace{-0.25cm}
\bibitem{x07} J. Xiao, {Homothetic variant of fractional Sobolev space
with application to Navier-Stokes system}, Dyn. Partial Differ. Equ.
4 (2007), 227-245.

\vspace{-0.25cm}

\bibitem{yy1}
D. Yang and W. Yuan, {A new class of function spaces connecting Triebel--Lizorkin
spaces and Q spaces}, J. Funct. Anal. 255 (2008), 2760-2809.

\vspace{-0.25cm}

\bibitem{yy2}
D. Yang and W. Yuan, {New Besov-type spaces and Triebel--Lizorkin-type spaces
including Q spaces}, Math. Z. 265 (2010), 451-480.


\vspace{-0.25cm}
\bibitem{YY1}
D. Yang and W. Yuan,
{Characterizations of Besov-type and Triebel--Lizorkin-type spaces via maximal functions and local means},
Nonlinear Anal. 73 (2010), 3805-3820.

\vspace{-0.25cm}

\bibitem{yy4} D. Yang and W. Yuan,
{Relations among Besov-type spaces, Triebel--Lizorkin-type spaces and
generallized Carleson measure spaces}, Appl. Anal. 92 (2013), 549-561.

\vspace{-0.25cm}

\bibitem{yyz13} D. Yang, W. Yuan and C. Zhuo, Complex interpolation on Besov-type and
Triebel-Lizorkin-type spaces, Anal. Appl. (Singap.) 11  (2013), 1350021, 45 pp.

\vspace{-0.25cm}

\bibitem{yhsy} W Yuan, D.D. Haroske, L. Skrzypczak and D. Yang,
Embedding properties of weighted Besov-type spaces, Anal. Appl. (Singap.) (2014),
DOI: 10.1142/S0219530514500493.

\vspace{-0.25cm}
\bibitem{YSY}
W. Yuan, W. Sickel and D. Yang,
 { Morrey and Campanato meet Besov, Lizorkin and Triebel},
 Lecture Notes in Mathematics 2005, Springer-Verlag, Berlin, 2010.

\vspace{-0.25cm}
\bibitem{ysy13}
W. Yuan, W. Sickel and D. Yang,
On the coincidence of
certain approaches to smoothness spaces related to Morrey spaces,
Math. Nachr. 286 (2013), 1571-1584.


\end{thebibliography}
\end{document}